\documentclass[11pt,reqno]{amsart}

 \usepackage{amsmath,amsfonts,amssymb,amsthm,mathrsfs}
 \usepackage{hyperref}
 \usepackage{bm}
 \usepackage{color}
 \usepackage{esint}
\usepackage{graphicx}
\usepackage{graphics}
\usepackage{geometry}
\usepackage[all]{xy}

\usepackage{tikz}
 
 \newtheorem{theorem}{Theorem}[section]
 \newtheorem{lemma}[theorem]{Lemma}
 \newtheorem{corollary}[theorem]{Corollary}
 \newtheorem{proposition}[theorem]{Proposition}
 \newtheorem{conjecture}[theorem]{Conjecture}
 \newtheorem{remark}[theorem]{Remark}

 \theoremstyle{definition}
 \newtheorem{definition}[theorem]{Definition}
 \newtheorem{question}[theorem]{Question}
 \newtheorem{problem}[theorem]{Problem}

 \theoremstyle{example}
 \newtheorem{example}[theorem]{Example}

\numberwithin{equation}{section}

\newcommand{\p}{\partial}
\newcommand{\bo}{\bm{0}}
\newcommand{\bx}{\bm{x}}
\newcommand{\Ca}{\mathcal{C}}
\newcommand{\dC}{\mathbb{C}}

\newcommand{\dN}{\mathbb{N}}

\newcommand{\cK}{\mathcal{K}}

\newcommand{\dR}{\mathbb{R}}
\newcommand{\dT}{\mathbb{T}}
\newcommand{\dZ}{\mathbb{Z}}

\newcommand{\fA}{\mathfrak{A}}
\newcommand{\fB}{\mathfrak{B}}

\newcommand{\bom} {{\bm\omega}}

\DeclareMathOperator{\Aff}{Aff}
\DeclareMathOperator{\Area}{Area}
\DeclareMathOperator{\Aut}{Aut}
\DeclareMathOperator{\Ch}{Ch}

\DeclareMathOperator{\diam}{diam}
\DeclareMathOperator{\dvol}{dvol}
\DeclareMathOperator{\ess}{ess}
\DeclareMathOperator{\FF}{flat}

\DeclareMathOperator{\Id}{Id}
\DeclareMathOperator{\Ima}{Im}
\DeclareMathOperator{\Injrad}{Injrad}

  \DeclareMathOperator{\Mod}{Mod}

 \DeclareMathOperator{\PI}{PI}
  
\DeclareMathOperator{\Ric}{Ric}
\DeclareMathOperator{\Rm}{Rm}
\DeclareMathOperator{\RCD}{RCD}
\DeclareMathOperator{\GL}{GL}
\DeclareMathOperator{\SO}{SO}
\DeclareMathOperator{\Supp}{Supp}
\DeclareMathOperator{\SL}{SL}
\DeclareMathOperator{\SU}{SU}
\DeclareMathOperator{\Sp}{Sp}

\DeclareMathOperator{\Isom}{Isom}

\DeclareMathOperator{\II}{II}
\DeclareMathOperator{\TF}{tf}
\DeclareMathOperator{\Tr}{Tr}
\DeclareMathOperator{\Vol}{Vol}

\begin{document}
 \title{Collapsing geometry of hyperk\"ahler 4-manifolds and applications}
 \author{Song Sun}
\address{Department of Mathematics, University of California, Berkeley, CA 94720} 
\email{sosun@berkeley.edu}

 \author{Ruobing Zhang}
\address{Department of Mathematics, Princeton University, Princeton, NJ 08544}
\email{ruobingz@princeton.edu}
 \thanks{The first  author is supported by the Simons Collaboration on Special Holonomy in Geometry, Analysis and Physics ($\#$ 488633), and NSF grant DMS-2004261. The second  author is supported by NSF grant DMS-1906265.}

 \maketitle
 \begin{abstract}
 We investigate the collapsing geometry of hyperk\"ahler 4-manifolds. As  applications we prove two well-known conjectures in the field:
 \begin{enumerate}
 	\item Any collapsed limit of unit-diameter hyperk\"ahler metrics on the K3 manifold is isometric to one of the following: the quotient of a flat 3-torus by an involution,  a singular special K\"ahler metric on the 2-sphere, or the unit interval.
 	\item Any complete hyperk\"ahler 4-manifold with finite energy (i.e., gravitational instanton) is asymptotic to a model end at infinity. 
 \end{enumerate}
 	\end{abstract}
 	
 \setcounter{tocdepth}{1}
 \tableofcontents
 
\section{Introduction} 
A Riemannian metric  $g$ on a smooth four manifold $X$ is \emph{hyperk\"ahler}  if its holonomy group is contained in $SU(2)\subset SO(4)$. The latter condition is equivalent to saying that we can choose an orientation so that the bundle $\Lambda^+X$ of self-dual 2-forms is trivialized by parallel sections. In particular, on a hyperk\"ahler 4-manifold, there is a triple of closed self-dual 2-forms ${\bm\omega}\equiv\{\omega_1, \omega_2, \omega_3\}$ satisfying  $$\omega_\alpha\wedge\omega_\beta=2\delta_{\alpha\beta}\dvol_g, \ \ \ \ \ \ \forall\ \  \alpha, \beta\in \{1, 2, 3\}.$$Such a triple is called a \emph{hyperk\"ahler triple}. 
Notice that conversely a hyperk\"ahler triple uniquely determines a hyperk\"ahler Riemannian metric. It is an important fact that hyperk\"ahler 4-manifolds have vanishing Ricci curvature; indeed they form the simplest non-trivial class of Ricci-flat metrics. In this paper we systematically study degenerations of hyperk\"ahler 4-manifolds, focusing on the case when the volume is \emph{collapsing}. Below we describe two main geometric applications.

The first application of our study is to the moduli compactification of hyperk\"ahler metrics on the K3 manifold. Here the K3 manifold $\mathcal K$ is by definition the oriented smooth 4-manifold underlying a complex K3 surface. We know the intersection form on  $H^2(\mathcal K;\dZ)$ has signature $(3, 19)$. 
 Denote by $\mathfrak M$  the set of equivalence classes of all unit-diameter hyperk\"ahler metrics on $\mathcal K$ modulo the natural action of $\text{Diff}(\mathcal K)$, endowed with the Gromov-Hausdorff topology. This space has an explicit description in terms of the \emph{period map}.  
 Recall that a hyperk\"ahler metric $g$ has a \emph{period}, which is the element in  the  Grassmannian of oriented maximal positive subspaces in $H^2(\mathcal K; \dR)$  given by  the space $\mathbb H^{+}(g)$ of self-dual harmonic forms. Taking into account of $\text{Diff}(\mathcal K)$ action we have a well-defined \emph{period map} (see \cite{KoTo})
 \begin{align}\mathcal P: \mathfrak M \longrightarrow \mathcal D \equiv \Gamma\setminus O(3, 19)/(O(3)\times O(19)), \label{period map}\end{align}
where $\Gamma$ is the automorphism group of $H^2(\mathcal K;\dZ)$ preserving the intersection form.    By the global Torelli theorem, $\mathcal P$ is  injective and maps onto an open dense subset of $\mathcal D$. Moreover, $\mathcal P$ extends to a bijection $\mathcal P: \mathfrak M'\rightarrow \mathcal D$, where $\mathfrak M'$ is the partial compactification of $\mathfrak M$ obtained by adding the volume non-collapsing Gromov-Hausdorff limits of smooth hyperk\"ahler triples ; the latter are known to be hyperk\"ahler \emph{orbifolds}, and their periods are maximal positive subspaces in $H^2(\mathcal K; \dR)$ which annihilate at least one homology class with self-intersection $-2$ (see for example \cite{OO}, Chapter 6). 

We are interested in understanding the full Gromov-Hausdorff compactification $\overline{\mathfrak M}$. The elements in $\overline{\mathfrak M}\setminus \mathfrak M'$ are volume \emph{collapsing} Gromov-Hausdorff limits of hyperk\"ahler metrics whose periods diverge to infinity in $\mathcal D$. We prove the following structural results for these limit spaces, hence confirm a folklore conjecture (see for example \cite{OO}, Proposition IV).

\begin{theorem}\label{t:thm1.1}
	Any collapsed limit in $\overline{\mathfrak M}\setminus \mathfrak{M}'$ must be isometric to one of the following:
	\begin{itemize}
		\item (dimension 3) a flat orbifold $\mathbb T^3/\dZ_2$;
		\item (dimension 2) a singular special K\"ahler metric on $S^2$ with local integral  monodromy; 
		\item (dimension 1) a one dimensional unit interval.
	\end{itemize}
\end{theorem}
In this paper will actually consider the more refined notion of \emph{measured Gromov-Hausdorff convergence}, which includes the extra structure of a \emph{renormalized limit measure} on the limit spaces (c.f. Section \ref{ss:2-1}). From the proof we know that in the first two cases, the limit measure is proportional to the Hausdorff measure, while in the third case  the limit measure may be nontrivial and it encodes interesting topological information of the collapsing family (c.f. \cite{HSVZ, SZ, HSZ}, see also Section \ref{ss:3-3}). Notice that in the more general context of collapsing 4 dimensional Ricci-flat metrics, Lott \cite{Lott} has obtained some classification results of limit spaces under certain technical assumptions on the limit spaces.    

There have been extensive recent work  on constructing special examples of collapsing sequences in $\mathfrak M$, which can be viewed as partial converses to Theorem \ref{t:thm1.1}. See for instance \cite{GW, Foscolo, CCIII, HSVZ, OO, CVZ1, CVZ2}. In particular, any flat orbifold $\mathbb{T}^3/\dZ_2$ is in $\overline{\mathfrak M}\setminus\mathfrak{M}'$; further work is needed in order to classify all 2 dimensional limit spaces in $\overline{\mathfrak M}\setminus\mathfrak{M}'$ explicitly.   We also mention that Odaka-Oshima \cite{OO} proposed an interesting conjecture  relating the Gromov-Hausdorff compactification $\overline{\mathfrak M}$ to certain Satake compactification of $\mathcal D$ as a locally symmetric space, and \cite{OO} made some progress toward the conjecture.

\

The second application of our study is concerned with the asymptotic structure of \emph{gravitational instantons}. The latter are by definition complete non-compact hyperk\"ahler 4-manifolds $(X, g)$ with $$\int_{X} |\Rm_g|^2\dvol_g<\infty.$$ These spaces originated from physics, but they also involve very rich geometry and analysis.  
There are a variety of constructions in the literature, such as hyperk\"ahler quotients, twistor theory, gauge theory, complex Monge-Amp\`ere equation, etc. Gravitational instantons are  important in understanding the singularity formation of  collapsing of hyperk\"ahler metrics, since they may arise as  rescaled limits around  points where curvature blows up. The next theorem gives a classification of the asymptotic geometry of gravitational instantons.
\begin{theorem}\label{t:thm1.2}
A non-flat gravitational instanton $(X,g)$ has a unique asymptotic cone $(Y,d_Y,p_*)$ which is a flat metric cone of dimension $d\in \{1, 2, 3,4\}$.
Moreover, the following classification holds:	\begin{itemize}
		\item ($d=4$) $(X,g)$ is ALE. 
		\item ($d=3$) $(X,g)$ is ALF.		\item ($d=2$) $(X,g)$ is ALG or ALG$^*$.
		\item  ($d=1$) $(X,g)$ is ALH or ALH$^*$.
	\end{itemize}
\end{theorem}

The precise definition of these spaces will be given in Section  \ref{ss:7-4}. The above classification result has  been long sought. The most recent result is due to Chen-Chen \cite{CCI} building upon  ideas from earlier work of Minerbe \cite{Minerbe2}, where one assumes the extra  condition $|\Rm|=O(r^{-2-\epsilon})$ for some $\epsilon>0$, and obtains a classification into only the four classes above without the superscript $^*$. This weaker result is proved by studying the behavior of ``short geodesic loops" at infinity using ODE comparison, and the asymptotic fibration is constructed using these short loops; the hyperk\"ahler property mainly enters as a control on the holonomy.  Our proof is based on a completely different approach. First, we make essential use of the Cheeger-Fukaya-Gromov theory on $\mathcal N$-structures, which has the advantage of incorporating multi-scale collapsing phenomenon at infinity. Secondly,  we manifest the role of the hyperk\"ahler equation itself as an elliptic PDE. These ideas could potentially apply to more general situation. 

It is also worth pointing out that there have been numerous works on the construction and classification of gravitational instantons with \emph{given} asymptotics at infinity, see for example \cite{CCII, CCIII,  CVZ1, CVZ2, CK, FMSW, Hein, HSVZ, HSVZ2, Kronheimer, Kronheimer-Torelli, Minerbe3} and the references therein. In particular, it is known that all the families of gravitational instantons listed in Theorem \ref{t:thm1.2} can be compactified in the complex-analytic sense. Together with these results, Theorem \ref{t:thm1.2} has the following corollary, which confirms the compactification conjecture of S.-T. Yau \cite{Yau-problem-section}  in our setting. 
\begin{corollary}\label{c:yau conjecture}
Given any gravitational instanton $(X, g)$, there is a choice of a complex structure $J$ such that $(X, J)$ is bi-holomorphic to $\overline X\setminus D$, where $\overline X$ is an algebraic surface and $D$ is an anti-canonical divisor.  
\end{corollary}

Now we outline some ideas involved in the proof of the above results. As mentioned before the central goal is to understand the collapsing geometry of hyperk\"ahler 4-manifolds with bounded $L^2$-energy. The result of Cheeger-Tian \cite{CT} provides an $\epsilon$-regularity theorem in our context, and as a consequence we know that the collapsing is with bounded curvature away from finitely many singularities. However due to the lack of a suitable monotonicity formula there has been no progress so far in understanding the structure of these singularities. This issue is  unique compared to other geometric analytic problems.
Our study depends  on three key ingredients:

\

$\bullet$ Geometric structures over the regular region (Section \ref{s-3}): we analyze the structure on the regular region of the limit space coming from the hyperk\"ahler structure. The analysis  depends on the dimension $d$ of the limit space; when $d= 1$, this was already done previously in \cite{HSZ}. A byproduct of this analysis  is a new and simple proof of the $\epsilon$-regularity theorem in our context (see Section \ref{s:3-4}).

\

	$\bullet$ Singularity structure of the limit space (Section \ref{s-4}, \ref{s-5}): we   study in detail the singularity structure in the cases $d=3$ and $d=2$. In particular we show that there is always a unique tangent cone which is a metric cone. Theorem \ref{t:thm1.1} follows from  Theorem \ref{t:global} and Theorem \ref{t:global-2d}. Notice again that the case $d=1$ in Theorem \ref{t:thm1.1} is easy (see \cite{HSZ}). 
	
	\
	
$\bullet$ Perturbation to invariant hyperk\"ahler metrics (Section \ref{s-6} and Section \ref{s-7}): The classical theory of nilpotent Killing structures due to Cheeger-Fukaya-Gromov \cite{CFG} asserts that over the regular region the collapsing sequence inherits an \emph{approximate} nilpotent symmetry along the collapsing directions.  We combine this with the perturbation theory of hyperk\"ahler metrics to obtain nearby hyperk\"ahler metrics with \emph{genuine} nilpotent symmetry. This is performed at both the local and global level. The local result improves our understanding of the collapsing fibers (Section \ref{s-6}), whereas the global result yields that a gravitational instanton with non-maximal volume growth at infinity must be asymptotic to a model end which admits a continuous symmetry (Section \ref{ss:6-3}). These allow us to prove Theorem \ref{t:thm1.2} in Section \ref{ss:7-4}. The techniques needed here are closely related to those used in the gluing constructions in \cite{Foscolo, HSVZ}.

\

\textbf{Notations.}
\begin{itemize}
	\item Given a metric space $(M,d)$ and a closed subset $E\subset M$, we denote  \begin{align*} & B_{r}(E)   \equiv\{q\in M|d(q, E)<r\},\\
	 & S_{r}(E)   \equiv\{q\in M|d(q, E)=r\},
	\\
 & A_{r_1, r_2}(E)  \equiv\{q\in M|r_1<d(q, E)<r_2\}. \end{align*}
\item We have various notations for Gromov-Hausdorff convergence:
  
 \begin{tabular}{cc}

$\xrightarrow{GH}$ :  &   (pointed) Gromov-Hausdorff convergence,
\\
 $\xrightarrow{eqGH}$ :  &  equivariant Gromov-Hausdorff convergence,
\\
  $\xrightarrow{mGH}$ : &  measured Gromov-Hausdorff convergence.
\end{tabular}
\item For a group $G$, we denote by $\mathfrak{Z}(G)$ the center of $G$.
\item $\dR_+\equiv [0, \infty)\subset \dR$.
\end{itemize}

\

{\bf Acknowledgements.}
We are grateful to Vitali Kapovitch and Dmitri Panov for sharing their insight on Question \ref{q:removable-singularity}, and to Gao Chen for discussing the positive mass theorem for gravitational instantons. We  thank Antoine Song for helpful conversations on related topics. We also thank Lorenzo Foscolo, Shouhei Honda, Shaosai Huang, Yuji Odaka and Yoshiki Oshima for  useful comments, and Yuji Odaka for kind suggestions on the references.  
We are also indebted to the anonymous referee whose inspiring comments and suggestions substantially improved the exposition of the paper
\section{Premilinaries}
\label{s-2}

\subsection{Pointed Gromov-Hausdorff distance}

\label{ss:2-0}
 The concept of pointed Gromov-Hausdorff convergence has been extensively used in the literature. For our purpose in this paper, it is convenient to exploit a metric space structure, which is likely well-known and we briefly recall the relevant notions.  We refer the readers to \cite{Rong-notes, Herron} for more details. 
Denote by $\mathcal{M}et$ the collection of isometry classes of all pointed complete length spaces $(M, d, p)$ such that every closed ball in $M$ is compact.
\begin{definition} Let $(M_1, d_1, p_1), (M_2, d_2, p_2)\in\mathcal{M}et$. Then the pointed Gromov-Hausdorff distance bwtween them is defined to be $d_{GH}((M_1, d_1, p_1), (M_2, d_2, p_2))\equiv \min\{\epsilon_0,\frac{1}{2}\}$, where  
$\epsilon_0\geq 0$ is the infimum of all $\epsilon\in [0, 1/2)$ such that there is a metric $d$ on $M_1\sqcup M_2$ extending $d_i$ on $M_i$, such that $d(p_1, p_2)\leq\epsilon$, $B_{1/\epsilon}(p_1)\subset B_{\epsilon}(M_2)$, and $B_{1/\epsilon}(p_2)\subset B_{\epsilon}(M_1)$.

\end{definition}
It is straightforward to verify that $d_{GH}$ defines a metric on $\mathcal{M}et$. One can prove that $(\mathcal{M}et, d_{GH})$ is a complete metric space. The convergence in this metric topology is the \emph{pointed Gromov-Hausdorff convergence}. 
For simplicity of notation in this paper we will omit the word \emph{pointed} and simply refer to this as the \emph{Gromov-Hausdorff convergence}. In the applications, one can also use the notion of $\epsilon$-Gromov-Hausdorff approximation (see \cite{Rong-notes}), which gives essentially the same topology.

Let $(X_j^n, g_j, p_j)$ be a sequence of $n$ dimensional  Riemannian manifolds with $\Ric_{g_j}\geq Kg_j$ for some $K\in \dR$ and for all $j$. Given a sequence of numbers $R_j>0$ with $\overline{B_{R_j}(p_j)}$ compact, from Gromov's compactness theorem, by passing to a subsequence we may assume
$(\overline{B_{R_j}(p_j)}, g_j, p_j)\xrightarrow{GH} (X_\infty, d_\infty, p_\infty)$ for a complete length space $(X_\infty,d_\infty, p_\infty)$. If $\{R_j\}$ is unbounded, then $X_\infty$ is noncompact.
Fix such a limit space $(X_\infty, d_\infty, p_\infty)$, we consider the rescaled spaces $(X_\infty, \lambda d_\infty, p_\infty)$ and let $\lambda\rightarrow\infty$.
 Any  Gromov-Hausdorff limit $(Y, p^*)$ for a subsequence $\{\lambda_i\}\rightarrow \infty$ is called a \emph{tangent cone} at $p_\infty$.  Recall  that a tangent cone is known to be a metric cone under a volume non-collapsing situation; but this is not always true in the volume collapsing case. We will denote by $\mathcal{T}_{p_{\infty}}\subset \mathcal{M}et$ the collection of isometry classes of all tangent cones at $p_\infty$.

Now we fix the above convergent subsequence $(\overline{B_{R_j}(p_j)}, g_j, p_j)$. Given any subsequence $\{\lambda_j\}\rightarrow\infty$, there is a  further subsequence $\{\lambda_{m_j}\}$ such that $(\overline{B_{R_{m_j}}(p_{m_j})},\lambda_{m_j}^2 g_{m_j},p_{m_j})\xrightarrow{GH}(Z, d_Z, \bar p)$. We call the space $(Z, d_Z, \bar p)$ a \emph{bubble limit} at $p_\infty$ associated to  the original convergent sequence. Denote by $\mathcal{B}_{p_{\infty}}$ the collection of isometry classes of all bubble limits at $p_\infty$. Immediately, $\mathcal{T}_{p_{\infty}}\subset\mathcal{B}_{p_{\infty}}$.

 Geometrically speaking, tangent cones describe the first order information of the singular behavior of the space $X_\infty$ itself at $p_\infty$, whereas bubble limits characterize more refined behavior for the singularity formation. The terminology should remind the readers the notion of a \emph{bubble tree}  structure in many geometric analytic problems. 
The following is a simple fact which we leave as an exercise for the readers.

\begin{lemma}\label{l:tangent cone compact}
	For any $p_{\infty}\in X_{\infty}$, both $\mathcal{T}_{p_{\infty}}$ and $\mathcal{B}_{p_{\infty}}$ are  compact in $\mathcal {M}et$. Moreover, $\mathcal{T}_{p_{\infty}}$ is connected.
\end{lemma}

Later we will also use an analogous result for asymptotic cones. Let $(X, d_X, p)$ be a complete Gromov-Hausdorff limit of a sequence of Riemannian manifolds with non-negative Ricci curvature. An asymptotic cone of $X$ is by definition a complete metric space  $(Y, d, p_*)$ arising as the Gromov-Hausdorff limit of $(X, \lambda_j^{-1}d_X, p)$ for some sequence $\lambda_j\rightarrow\infty$. Clearly this does not depend on the choice of the base point $p$. Denote by $\mathcal{T}_\infty(X)$ the collection of isometry classes of asymptotic cones $X$. Similar to above we have the following.
\begin{lemma} \label{l:tangent cone at infinity connected}
	$\mathcal{T}_\infty(X)$ is connected and compact in $\mathcal{M}et$. Moreover, it is invariant under rescaling, i.e., if $(Y, d_Y, p_*)$ is in $\mathcal T_\infty(X)$, so is $(Y, \lambda d_Y, p_*)$ for all $\lambda>0$.  
\end{lemma}

We also mention that in this paper various other notions of convergence will also be used, such as Cheeger-Gromov convergence, equivariant Gromov-Hausdorff convergence, etc, and the mixture of them. For definitions of these notions we refer the readers to standard references. 

\subsection{Renormalized limit measure}
\label{ss:2-1}

As above we let $(X_j^n, g_j, p_j)$ be a sequence of $n$-dimensional  Riemannian manifolds with $\Ric_{g_j}\geq K g_j$ such that $\overline {B_2(p_j)}$ is compact. Suppose 
$(X_j^n, g_j, p_j)\xrightarrow{GH} (X_{\infty}, d_{\infty}, p_{\infty})$
for some length space  $(X_{\infty}, d_{\infty})$.
We denote by
\begin{align*}d\nu_j \equiv \frac{\dvol_{g_j}}{\Vol_{g_j}(B_1(p_j))}\end{align*}
 the \emph{renormalized measure} density on $X_j^n$. Then by the work of  Cheeger-Colding (see Theorem 1.10 of \cite{CC1}) we know that by passing to a further subsequence, there is a Radon measure $\nu_{\infty}$ on $X_{\infty}$, called the {\it renormalized limit measure},  such that  for any converging sequence of points $q_j\to q_{\infty}$ and for all $R>0$, we have  
$\nu_j(B_R(q_j))\longrightarrow\nu_{\infty}(B_R(q_{\infty})).$ 
  The metric measure space $(X_{\infty}, d_{\infty}, \nu_\infty, p_{\infty})$ is called a {\it Ricci limit space}, and we have the {\it  measured  Gromov-Hausdorff convergence}
  \begin{align}
	(X_j^n, g_j,  \nu_j, p_j)\xrightarrow{mGH} (X_{\infty}, d_{\infty}, \nu_{\infty}, p_\infty).\label{e:pointed-mGH}
  \end{align}
It is known that  $\nu_{\infty}$ satisfies the relative volume comparison,  and the following volume estimate
\begin{align}
\nu_{\infty}(B_r(x))\leq C \cdot r, \quad \forall\ x\in X_{\infty}, \ \forall\ r\in(0,1]. \label{e:linear-volume-estimate}
\end{align}
See Theorem 1.10 and Proposition 1.22 of \cite{CC1}, respectively.

\begin{definition}\label{d:regular-set}
Let $(X_{\infty},d_{\infty}, \nu_\infty, p_{\infty})$ be a (connected) Ricci limit space of $(X_j^n, g_j, \nu_j, p_j)$.  
 	\begin{enumerate}
 		\item We define the \emph{regular set} $\mathcal R$ to be the set of points $q\in X_\infty$ such that 
		 there exist constants $r_0>0$, $C_0>0$, and a sequence of points $q_j\in X_j^n$ converging to $q$ with $\sup\limits_{B_{r_0}(q_j)}|\Rm_{g_j}|\leq C_0$ for all $j$.
 	\item We define the \emph{smooth set} $\mathcal{G}\subset \mathcal{R}$ to be  the set of  
  points $q$ such that $X_\infty$ is a smooth Riemannian manifold in a neighborhood of $q$. We denote by $g_\infty$ the Riemannian metric on $\mathcal{G}$;
  \item We define the {\it singular set}  $\mathcal{S}\equiv  X_{\infty} \setminus \mathcal{G}$;
   
 \end{enumerate}
 \end{definition} Notice these definitions depend on the convergent sequence $X_j^n$. Clearly, $\mathcal R$ is open. By Colding-Naber  \cite{CN},  there exist a subset $\mathcal{R}^{\#}\subset X_{\infty}$ and an integer $d\in\dZ_+$ such that $\nu_{\infty}(X_{\infty}\setminus \mathcal{R}^{\#}) = 0$
and every point in $\mathcal{R}^{\#}$ has a unique tangent cone which is isometric to $\dR^d$. We call the integer $d$  the {\it essential dimension} of $X_{\infty}$, and we denote it by $\dim_{\ess}(X_\infty)$. It is obvious that $\mathcal{G}= \mathcal{R} \cap \mathcal{R}^{\#}$, so $\dim_{\ess}(X_\infty)=\dim\mathcal G$. 

In this paper we are mainly interested in the collapsing situation so from now on we assume $d<n$. 
  It is worth noting that neither $\mathcal{R}^{\#}\subset\mathcal{R}$ nor $\mathcal{R}\subset \mathcal{R}^{\#}$  necessarily holds in general.  Nevertheless in Section \ref{ss:3-1} we will show  that $\mathcal R\subset \mathcal R^{\#}$ in our setting of collapsing 4 dimensional hyperk\"ahler metrics.

  Fukaya  \cite{Fukaya} showed that on $\mathcal G$ the measure $\nu_\infty$ has an explicit expression, namely, its density is \begin{align*}d\nu_\infty=\chi\cdot \dvol_{g_\infty}\end{align*} for a smooth function $\chi$ determined as follows. Given $q\in \mathcal{G}$, we can find $q_j\in X_j$ converging to $q$ and $\delta>0$ such that the universal cover $(\widetilde{B_\delta(q_j)}, \widetilde g_j, q_j, G_j)$ converges in the equivariant Cheeger-Gromov sense, to a smooth limit $(\widetilde B_\infty, \widetilde g_\infty, q_\infty, G_\infty)$. Here $G_j$ is the fundamental group of $B_\delta(q_j)$, and the identity component of $G_\infty$ is a nilpotent Lie group. Moreover,  a neighborhood 
of $q$ in $X_\infty$ is isometric to the quotient $B_\infty{\equiv}\widetilde B_\infty/G_\infty$. In this context, we can identify the fiber $F_{q'}$ over any $q'\in B_\infty$ of the projection map $\widetilde B_\infty\rightarrow B_\infty$ locally with an open  neighborhood in $G_\infty$. Then $\chi$ is up to constant multiplication given by the ratio between the vertical Riemannian volume form on $F_{q'}$ (of the induced Riemannian metric from $\widetilde g_\infty$) and a fixed left-invariant volume form on $G_\infty$.

We often write $\chi=e^{-f}$. 
As observed by Lott (see Theorem 2 of \cite{Lott-BE}), using O'Neill's formula, the Bakry-\'Emery-Ricci curvature lower bound is preserved on the limit, i.e., on $\mathcal{G}$ we have \begin{align*}\Ric^{n-d}_{g_{\infty}}\equiv \Ric_{g_{\infty}}+\nabla^2_{g_\infty}f-\frac{1}{n-d} df\otimes df\geq K  g_\infty.\end{align*}
 Although not needed in this paper, we notice the fact that  globally one can say $(X_\infty, d_\infty, \nu_\infty)$ has Ricci bounded below by $K$ in the $\RCD$ sense, i.e, it is a $\RCD(K, n)$ space. See \cite{AGS, GMS} for details.

\subsection{Harmonic functions}
\label{ss:2-3}
In this subsection, we introduce some standard concepts and basic results about harmonic functions. For our purpose, we only state them on Ricci limit spaces, and we list the references in the general RCD setting. 
 To begin with, let $(X_j^n, g_j, d\nu_j, p_j)$ be a sequence of $n$-dimensional Riemannian manifolds with $\Ric_{g_j}\geq 0$ and $d\nu_j\equiv \Vol_{g_j}(B_1(p_j))^{-1} \dvol_{g_j}$ such that $
 (X_j^n,g_j,\nu_j,p_j)\xrightarrow{mGH}(X_{\infty},d_{\infty},\nu_{\infty},p_{\infty}). $

A key ingredient in the definition of harmonic functions on a metric measure space is the following notion of {\it minimal weak upper gradient}, which plays the role of $|\nabla u|$ in the smooth case. To define this, let us first introduce the notation of {\it the $2$-modulus} of a family of curves. Let $\Gamma$ be a family of curves on $X_{\infty}$. Then we define {\it the $2$-modulus} $\Mod_2(\Gamma)$ of $\Gamma$ by 
\begin{align*}
\begin{split}
\Mod_2(\Gamma)\equiv \inf\left\{\int_{X_{\infty}}\psi^2 d\nu_{\infty}: \ \psi\geq 0 \ \text{is measurable such that} \ \int_{\gamma}\psi ds \geq 1, \forall \gamma\in\Gamma\right\}.
\end{split}	
\end{align*}
Now we are ready to define the minimal weak upper gradient.
\begin{definition}
[Minimal weak upper gradient]	Let $u$ be a measurable function on $X_{\infty}$.
A nonnegative measurable function $g$ on $X_{\infty}$ is said to be a $2$-weak upper gradient of a function $u$ if for any $z_1,z_2\in X_{\infty}$ and for every rectifiable curve $\gamma:[0,\ell]\to X_{\infty}$ parameterized by arc length with $\gamma(0)=z_1$ and $\gamma(\ell)=z_2$ with the exception in a family of curves $\Gamma$  with $\Mod_2(\Gamma)=0$ such that
\begin{align*}|u(z_2) - u(z_1)| \leq \int_0^{\ell} g(\gamma(s))ds. 	\end{align*}
 {\it The minimal weak upper gradient} $|\nabla u|$ of a function $u$ is the $2$-weak upper gradient such that for all $2$-weak upper gradient $g$, we have $|\nabla u|\leq |g|$ a.e. on $X_{\infty}$.

\end{definition}

Based on the notion of minimal weak upper gradient, the Cheeger energy of $u$ is defined by 
\begin{align*}
	\Ch(u) \equiv \int_{X_{\infty}}|\nabla u|^2 d\nu_{\infty}, 
\end{align*}
and the $W^{1,2}$-Sobolev space is defined by $W^{1,2}(X_{\infty})\equiv \{u\in L^2(X_{\infty})|\Ch(u)<\infty\}$. It is known that the Cheeger energy is quadratic
 on a Ricci limit space (see \cite{AGS} and \cite{GMS}). This enables us to define the following Dirichlet form
\begin{align*}
	\mathcal{E}(u,v) = \int_{X_{\infty}}\langle\nabla u, \nabla v\rangle d\nu_{\infty} \equiv \frac{1}{2}\Big(\Ch(u+v)-\Ch(u-v)\Big).
\end{align*}
where $u,v\in W^{1,2}(X_{\infty})$.
Note that $\langle\nabla u, \nabla v\rangle$ is a well-defined $L^1$-function, but $\nabla u$ itself is not defined in general. We also point out that $\mathcal{E}(u,v)$ coincides with the standard Dirichlet form in the smooth case. 

\begin{definition}
	[Harmonic function]\label{def:harmonic functions} Let $\Omega\subset X_{\infty}$ be an open set.
	A function $u\in W^{1,2}(\Omega)$ is said to be harmonic if 
$\mathcal{E}(u,\varphi) \equiv 0$ for all Lipschitz functions $\varphi$
with compact support in $\Omega$. 
\end{definition}

We will use the following weak Harnack inequality.
\begin{theorem}[Weak Harnack inequality] \label{t:weak-harnack}	Let $(X_{\infty},d_{\infty}, \nu_{\infty}, p_{\infty})$ be a Ricci limit space. For any $p>0$, there exists some constant $C=C_p>0$  depending only on $p$ such that if $u$ is a nonnegative harmonic function on $B_2(p_{\infty})$, then 	\begin{align*}\Big(\int_{B_2(p_{\infty})} u^p d\nu_{\infty}\Big)^{\frac{1}{p}} \leq C \cdot \underset{B_1(p_{\infty})}{\ess\inf} \ u.			\end{align*}	\end{theorem}
This theorem indeed holds in the very general context of metric measure spaces under appropriate assumptions; see Theorem 9.8 of \cite{Bjorn-Bjorn}.  For completeness, we briefly explain the crucial technical ingredients involved in the proof of Theorem \ref{t:weak-harnack}. First, the space is required to support the $(1,p)$-Poincar\'e inequality, i.e., there exists some constant $C_{\PI}>0$ depending on $p$ such that for any function $u\in L^1(B_r(x))$ with $x\in X_{\infty}$ and $r>0$, and for all upper gradients $g$ of $u$, it holds that \begin{align*}\int_{B_r(x)}| u - u_{x,r}|d\nu_{\infty} \leq C_{\PI}\cdot r\cdot \Big(\int_{B_r(x)}g^p d\nu_{\infty}\Big)^{\frac{1}{p}},\end{align*}
where $u_{x,r}\equiv\fint_{B_r(x)}ud\nu_{\infty}$. 
 In the context of Ricci limit spaces, the above $(1,p)$-Poincar\'e inequality (for all $p\geq 1)$ follows from Cheeger-Colding's segment inequality; see Theorem 2.15 in \cite{CC3}. Secondly, one needs to apply the technique of Moser's iteration, which requires a uniform Sobolev inequality. This follows from the Poincar\'e inequality and the volume comparison for the renormalized limit measure; see Theorem 4.21 of \cite{Bjorn-Bjorn}.

Now we consider the setting of Section \ref{ss:2-1}. On the smooth set $\mathcal G$ we have a Riemannian metric $g_{\infty}$, and a  measure  density $d\nu_\infty=e^{-f}\dvol_{g_\infty}$. Suppose $B_\delta(p_\infty)\subset \mathcal G$, then a function $u$ on $B_\delta(p_\infty)$ is harmonic if and only if  $\Delta_{\nu_{\infty}} u = 0$ on $B_\delta(p_{\infty})$, where  \begin{align*}\Delta_{\nu_\infty}u \equiv \Delta_{g_\infty} u -  \langle\nabla_{g_\infty} f, \nabla_{g_\infty} u\rangle \end{align*}
is the Bakry-\'Emery Laplace operator.
 Locally if we pull-back to $\widetilde B_\infty$, then it is easy to see $\Delta_{\nu_\infty}u=\Delta_{\widetilde g_\infty}u$. 
We have
the following Cheng-Yau type gradient estimate (see Theorem 2.1 of \cite{Qian})

\begin{theorem}[Gradient estimate] \label{l:Yau gradient estimate} Suppose $(\mathcal{G}, g_{\infty}, d\nu_{\infty})$ satisfies $\Ric^{n-d}_{g_{\infty}}\geq 0$. Then 
	there exists a constant $C_0=C_0(n)>0$  such that any positive harmonic function $u$ defined on $B_{2r}(x)\subset \mathcal{G}$ satisfies \begin{align*}
\sup\limits_{B_r(x)}|\nabla \log u|\leq C_0r^{-1}.	
\end{align*}

\end{theorem}

\subsection{Deformation of definite triples}\label{ss:triple-deformation}
Here we review \cite{Donaldson, Foscolo, HSVZ}. 
Let $X$ be an oriented smooth $4$-manifold, possibly non-compact or with boundary, and fix a volume form $\dvol_0$. 
Let ${\bm\omega}\equiv\{\omega_1, \omega_2, \omega_3\}$ be a triple of closed 2-forms on $X$. Write 
	$\omega_\alpha\wedge \omega_\beta=2Q_{\alpha\beta}\dvol_0$ for  $1\leq\alpha, \beta\leq 3.$
\begin{definition}	$\bm{\omega}$ is called a  \emph{definite triple} if the matrix-valued function $Q\equiv(Q_{\alpha\beta})$ is positive definite everywhere on $X$.
\end{definition} 
A definite triple  $\bm{\omega}$ uniquely determines a Riemannian metric $g_{\bom}$ such that each $\omega_\alpha$ is self-dual with respect to $g_{\bom}$ and the volume form is given by $ \dvol_{g_\bom}=(\det(Q))^{\frac{1}{3}} \dvol_0.$
Denote by $\Omega^+$ the space of self-dual 2-forms (with respect to $g_\bom$) on $X$. Below we will often identify an element in $\Omega^+\otimes\dR^3$ (i.e.,  a triple of self-dual 2-forms) with a $3\times 3$ matrix-valued function: $\bm{\eta}\in \Omega^+\otimes \dR^3$ corresponds to $ (A_{\alpha\beta})$  if $\eta_\alpha=\sum_\beta A_{\alpha\beta}\omega_\beta$. 

\begin{definition}
 A definite triple $\bom$ is called {\it hyperk\"ahler} if the metric $g_{\bom}$ is hyperk\"ahler, or equivalently,   if the normalized determinant  $
Q_{\bm{\omega}}\equiv(\det(Q))^{-\frac{1}{3}}Q	
$ is the identity matrix.\end{definition}

Fix a definite triple $\bom$. Consider a deformation $\bom'=\bom+{}{\bm{\theta}}$ for a triple  $\bm\theta$ of closed 2-forms. Decompose $\bm\theta = \bm\theta^+ + \bm\theta ^-$, where $\bm\theta^+$ is self-dual and $\bm\theta^-$ is anti-self-dual  with respect to $g_{\bom}$. 
Then as above we may identify ${}{\bm \theta}^{+}$ with a  matrix-valued function $A=(A_{\alpha\beta})$. Define the  matrix-valued function ${}{S}_{{}{\bm\theta}^-}\equiv (S_{\alpha\beta})$ via
\begin{align*}\theta_\alpha^-\wedge\theta_\beta^-=2S_{\alpha\beta}\dvol_{g_\bom}, \ \ \ \ 1\leq\alpha, \beta\leq 3.\end{align*}
If $\bom'$ is definite, then the hyperk\"ahler condition on $\bom'$ can be expressed as
\begin{equation}
\TF(Q_{\bm{\omega}}A^T + AQ_{\bm{\omega}}  + AQ_{\bm{\omega}}A^T) = \TF(-Q_{\bm{\omega}}-S_{\bm{\theta}^-}), \label{e:lftf}
\end{equation}
where we denote by $\TF(B) = B - \frac{1}{3}\Tr(B)\Id$ the trace-free part of a matrix $B$. 
Let $\mathscr{S}_0(\dR^3)$  be the space of trace-free symmetric $3\times 3$ matrices, and denote the non-linear map $\mathfrak{G}: \mathscr{S}_0(\dR^3) \to \mathscr{S}_0(\dR^3); A\mapsto \TF(Q_{\bm{\omega}}A^T + AQ_{\bm{\omega}}+ AQ_{\bm{\omega}}A^T)$. Then $\mathfrak{G}$ is a local diffeomorphism near $0$, and we denote by $\mathfrak{F}:U\rightarrow \mathscr{S}_0(\dR^3)$  its local inverse, where $U$ is a small neighborhood of $0$. A sufficient condition for  \eqref{e:lftf} to hold is 
\begin{align} 
A = \mathfrak{F}\Big(\TF(-Q_{\bm{\omega}}-S_{\bm\theta^-})\Big).
\label{e:elliptic-system}
\end{align} 
Notice this is only a necessary condition if we assume a priori that the matrix $A$ above is symmetric and trace-free.

 We now impose the ansatz $\theta_\alpha=dd^*(\sum_\beta f_{\alpha\beta}\omega_\beta)$, where $\bm{f}\equiv (f_{\alpha\beta})$ is a $3\times 3$ matrix-valued function, and   the Hodge $*$-operator is defined with respect to the metric $g_{\bom}$. We  can write this concisely as \begin{align*} \bm{\theta}=dd^*(\bm{f}\cdot\bom).\end{align*} 
Define the non-linear operator 
\begin{equation} \label{e:230}
	\mathscr{F}:\Omega^{+}\otimes \dR^3\rightarrow \Omega^+\otimes \dR^3;\quad \bm f\mapsto  \mathscr{D}(\bm f)+\mathscr{N}_0(\bm f ),
\end{equation}
where 
$\mathscr{D}(\bm f)\equiv d^+d^*(\bm f\cdot \bom)$ and 
$\mathscr{N}_0(\bm f )\equiv-\mathfrak F(\TF(-Q_{\bom}-S_{d^-d^*(\bm f\cdot\bom)})).$  Strictly speaking, $\mathscr F(\bm f)$ is only well-defined when $|\bm f|_{C^2_{\bom}}$ is small (so that $\TF(-Q_{\bom}-S_{d^-d^*(\bm f\cdot\bom)}(x))$ is in $U$ for all $x\in X$). Clearly  \eqref{e:elliptic-system} follows if $\mathscr{F}(\bm f)=0$.

If $\bom$ is hyperk\"ahler, then $\nabla_{g_{\bom}}\bm \omega\equiv \bm 0$. In this case we have $\mathscr{D}=-\Delta _{{\bm \omega}}$,  where $\Delta_{\bom}$ is the analyst's Laplace operator.  In general we have
  \begin{equation}\label{e:operator error}d^+d^*(\bm{f}\cdot \bm{\omega})=-(\Delta _{{\bm \omega}}{\bm f})\cdot \bm\omega+\nabla_{{\bm\omega}}{\bm f}\star\nabla_{\bm\omega} {\bm\omega}, 
 \end{equation}
 where $\star$ denotes a general tensor contraction.
 This follows from
 \begin{lemma}[see also \cite{HSVZ2}]
 	 Given a closed self-dual  2-form $\gamma$  and a function $f$, we have
 	 \begin{equation*}d^+d^*(f\gamma)=(-\Delta_{{\bm\omega}} f) \gamma+\nabla_{{\bm\omega}}f\star \nabla_{{\bm\omega}}{\gamma}.
 \end{equation*}
 \end{lemma}
 \begin{proof}
 Given a point $p_0$. We choose a local oriented orthonormal frame $\{e_i\}$ with dual co-frame $\{e^i\}$, such that $\nabla_{\bm\omega} e_i(p_0)=0$ and in a neighborhood of $p_0$,$\Lambda^+$ is spanned by $e^{12}+e^{34}, e^{13}+e^{42}, e^{14}+e^{23}$, where $e^{ij}\equiv e^i\wedge e^j$.
Since $\gamma$ is closed and self-dual we have 
 $d^+d^*(f\gamma)=d^+*(df\wedge\gamma)$. 
 We write $\gamma=\sum\limits_{i<j}\gamma_{ij}e^{ij}$, with $\gamma_{12}=\gamma_{34}, \gamma_{13}=-\gamma_{24}, \gamma_{14}=\gamma_{23}$. 
 Then the conclusion follows from straightforward computation, using the fact that $\Delta_{\bm\omega} f=-\sum\limits_{i=1}^4e_ie_if$ at $p_0$. 
 \end{proof}
We also notice that  algebraically we have the  pointwise estimate
 \begin{equation}\label{e:pointwise non-linear estimate}
 |\mathscr{N}_0(\bm f)-\mathscr{N}_0(\bm g)|\leq C \left(|\nabla^2_{\bom}\bm f|_{\bm\omega}+|\nabla^2_{\bom} \bm g|_{\bm\omega}\right) \cdot |\nabla^2_{\bom}(\bm f-\bm g)|_{\bom},
 \end{equation}
 as long as $\bm f$ and $\bm g$ are in the domain of definition of $\mathscr F$. 

\

Now we assume
$\mathscr{F}(\bm f)=\mathscr{L}(\bm f)+\mathscr{N}(\bm f),$
where $\mathscr{L}$ is a linear operator. In our applications  $\mathscr{L}$ will be a slight modification of $-\Delta_{\bm\omega}$. The following is an application of the standard quantitative implicit function theorem on Banach spaces. 
\begin{proposition} \label{p:implicit function theorem}
	Suppose we have two Banach spaces $(\fA, \|\cdot\|)$, $(\fB, \|\cdot\|)$ and numbers $\eta>0, L>0$, such that the following hold
\begin{enumerate}
\item $\fA\subset C^{2}(\Omega^+\otimes \dR^3)$, $\fB\subset C^0(\Omega^+\otimes \dR^3)$;
\item   For all $\bm f\in B_{\eta}(0)\subset \fA$, the triple $\bm\omega+dd^*(\bm f\cdot\bom)$ is definite, and $\TF(Q_{\bom}(x)+S_{d^-d^*(\bm f\cdot\bom)}(x))\in U$ for all $x\in X$;
	\item $\mathscr L$ and $\mathscr N$ are both differentiable maps from $B_{\eta}(0)\subset \fA$ to $\fB$;
	\item There exists a bounded linear map $\mathscr P: \fB\rightarrow \fA$ with $\mathscr L\circ \mathscr P=\Id$, and $\|\mathscr P\|\leq L$;
	\item $\|\mathscr  N(\bm f)-\mathscr N(\bm g)\|\leq (3L)^{-1}\|\bm f-\bm g\|$ for all $\bm f, \bm g\in B_{\eta}(0)\subset \fA$;
	\item $\|\mathscr F(\bm 0)\|\leq \eta (3L)^{-1}$.
	\end{enumerate}
Then we can find an $\bm f\in \fA$  that solves $\mathscr F(\bm f)=0$ and satisfies $\|\bm f\|\leq 2L\|\mathscr F(0)\|$. In particular, $\widetilde{\bm\omega}=\bm\omega+dd^*(\bm f\cdot\bm\omega)$ is a hyperk\"ahler triple.
\end{proposition}
 In practice the condition (2) will be achieved by making sure that $|dd^*(\bm f\cdot \bom)|_{\bm\omega}$ is small for all $\bm f\in B_\eta(0)$,  which in particular guarantees that $\widetilde{\bm\omega}$ is equivalent to ${\bm\omega}$. 
 The above strategy was used for example by Foscolo \cite{Foscolo} (see also   \cite{HSVZ}) to construct degeneration families of hyperk\"ahler metrics on the K3 manifold with precise geometric information. The main technical issue there is to find a right inverse $\mathscr L$ with uniform estimates on suitable weighted spaces. In the setting of a K3 manifold due to topological reasons $\Delta_{\bm\omega}$ can not be surjective; this is circumvented by  the extra freedom of adding a finite dimensional space of self-dual harmonic forms. 
 
 For our purpose in this paper we will work on manifolds with boundary (and possibly noncompact), so we do not encounter the topological obstruction to the surjectivity of $\Delta_\bom$. We also make the trivial observation that suppose we have a compact group $G$ acting on $X$ preserving $\bm\omega$, then if the assumptions (1)-(6) of Proposition \ref{p:implicit function theorem} holds for $G$-invariant objects, then we can find a nearby hyperk\"ahler triple $\widetilde{\bm\omega}$ which is also $G$-invariant.

Another remark is that the ansatz used in \cite{Foscolo, HSVZ} is slightly different from what is used above. Namely, in those situations one would write $\bm\theta=d\bm\eta+\bm\xi$ for a triple of $d^*$-closed 1-forms $\bm\eta$ and a triple of self-dual harmonic forms $\bm\xi$. The corresponding linearized operator involves the Dirac operator $d^*+d^+$ acting on 1-forms. In our formulation above we have further specified $\bm\eta=d^*(\bm f \cdot \bom)$. This provides some technical simplifications since it allows us to only deal with elliptic operator on functions. To our knowledge this trick goes back to Biquard \cite{Biquard} (see also \cite{HSVZ2}).

\section{Geometric structures over the regular region}
\label{s-3}

In this section, we consider a measured collapsed Gromov-Hausdorff limit $(X_{\infty}^d, d_\infty, p_\infty, \nu_\infty)$ of a sequence of hyperk\"ahler $4$-manifolds $(X_j^4, g_j, p_j,\nu_j)$ with $d\equiv\dim_{\ess}(X_{\infty}^d)<4$. We will always fix a choice of hyperk\"ahler triple $\bm\omega_j$ on $X_j^4$. Our goal is to understand the refined geometric structure on the regular region $\mathcal{R}\subset X_{\infty}^d$ inherited from the hyperk\"ahler structure on $X_j^4$. 
Notice that due to volume collapsing, one can not make obvious sense of convergence of the hyperk\"ahler triples ${\bm\omega}_j$. Instead, we will take the limit of $\bom_j$ on local universal covers, which descends to a local structure on $\mathcal R$, then gluing them together yields certain global structure on $\mathcal R$. The precise structure we obtain on $\mathcal R$ depends on its dimension $d$.  

Now we make the above rigorous. Without loss of generality we always assume $p_\infty\in \mathcal R$ in this section. Then there exists some $\delta>0$ independent of $j$ such that the universal cover $\widetilde B_j$ of $B_j\equiv B_\delta(p_j)$ is volume non-collapsing as $j\rightarrow\infty$. Let $G_j$ be the deck transformation group and $\widetilde{\bm {\omega}}_j$ be the pullback of ${\bm\omega}_j$ to $\widetilde B_j$.  The isometry action of $G_j$ preserves $\widetilde{\bm {\omega}}_j$. Passing to a subsequence, we have the equivariant Gromov-Hausdorff convergence 
\begin{align*}
 \begin{split}
 \xymatrix{
(\widetilde{B_j}, \tilde{g}_j, G_j, \tilde{p}_j) \ar[d]_{\pi_j} \ar[rr]^{eqGH} &&    (\widetilde{B}_{\infty}, \tilde{g}_{\infty}, G_{\infty}, \tilde{p}_{\infty}) \ar [d]^{\pi_{\infty}}
 \\
 (B_j, g_j)\ar[rr]^{GH} &&  (B_{\delta}(p_{\infty}), g_{\infty}),
 } 
 \end{split}
 \end{align*}
 where $G_{\infty}\leq \Isom(\widetilde{B}_{\infty})$ is a closed subgroup such that $B_{\delta}(p_{\infty})=\widetilde{B}_{\infty}/G_{\infty}$.  We refer readers to \cite[Definition 3.3]{FY} for the detailed definition of equivariant Gromov-Hausdorff convergence.
Moreover, the standard regularity theory for non-collapsing Einstein metrics implies that the convergence of $\widetilde{B}_j$ can be improved to the $C^k$-convergence for any $k\in\dZ_+$. Then we obtain a smooth limit hyperk\"ahler triple $\widetilde{\bm {\omega}}_\infty$ on $\widetilde{B}_{\infty}$ which is preserved by $G_\infty$.

We make the following observation that any point in the regular set $\mathcal{R}$ is a manifold point. 
\begin{proposition} \label{p:regular set smooth}
	In our setting, $\mathcal R\cap \mathcal S=\emptyset$. In particular, $\mathcal{R}\subset\mathcal{R}^{\#}$. 
\end{proposition}
\begin{proof}
To prove the proposition, we claim that $G_\infty$ acts freely on $\widetilde{B}_{\infty}$. To see this, suppose otherwise that $\widetilde p_\infty$ is a fixed point of some non-trivial element $\phi\in G_\infty$ such that there exists a sequence of $\phi_j\in G_j$ that converges to $\phi_j$ equivariantly. 

Now using the exponential map at $\widetilde p_\infty$, we may identify the action of $\phi$ with the linear action $L=d\phi$ on $T_{\widetilde p_\infty}\widetilde{B}_{\infty}$.
Since $G_\infty$ preserves $\widetilde{\bm {\omega}}_\infty$,  we can identify $T_{\widetilde p_\infty}\widetilde{B}_{\infty}$ with the quaternions $\mathbb H$ so that   $d\phi$ acts by left multiplication by a unit quaternion. In particular, $1$ is not an eigenvalue of $L$. Now for any sufficiently large $j$, we may write $\phi_j\in G_j$ as 
\begin{align*}\phi_j=L+E_j,\end{align*} where $\|E_j\|_{C^2}$ is small. By a simple application of the implicit function theorem we see that for $j$ large $\phi_j$ must also have a nearby fixed point. This contradicts the fact that the $G_j$ action is free.
\end{proof}

\begin{remark}
Here we used crucially the property that $\SU(2)(\cong\Sp(1))$ acts freely on the unit sphere $S^3$. This Proposition was also implicitly proved in \cite{CCI} using a different argument. 
\end{remark}

Using similar arguments, we also obtain the following.
\begin{proposition}
The Lie group $G_\infty$ is connected.
\end{proposition}
\begin{proof}
Notice that the projection map $\widetilde B_\infty\rightarrow B_\infty$ has connected fibers, on which $G_\infty$ acts transitively.  Then the conclusion follows from the fact that the $G_\infty$ action is free. 
\end{proof}

Now we divide the discussion into 3 cases depending on the dimension $d$.  The case $d=1$ was studied in detail in \cite{HSZ}. So our main focus below is in the other two cases $d=2$ and $d=3$.

\subsection{Case $d=3$}
\label{ss:3-1}

\subsubsection{Geometric structure on $\mathcal R$}
 In this case $G_\infty=\dR$.  Choosing a generator of $G_\infty$ gives a  Killing field $\p_t$ on $\widetilde B_\infty$ which preserves the hyperk\"ahler triple $\widetilde{\bm\omega}_\infty$. Shrinking $\widetilde B_\infty$ if necessary, we may assume there is a triple of moment maps for $\p_t$ with respect to $\widetilde{\bm\omega}_\infty$ given by $\pi_\infty=(x, y, z): \widetilde B_\infty\rightarrow \dR^3$.  These serve as local coordinates on $\mathcal R$, and the Riemannian metric $g_{\infty}$ can  be written as 
 $$g_\infty=V(dx^2+dy^2+dz^2),$$
where $V=|\p_t|^{-2}$ satisfies $\Delta_{V^{-1}g_\infty}V=0$. This is the well-known description of hyperk\"ahler metrics with an $\mathbb R$ symmetry, i.e., the \emph{Gibbons-Hawking ansatz}; see \cite{Gibbons-Hawking, HKLR}  for details of the construction.
We can write the hyperk\"ahler triple  on $\widetilde B_\infty$  as 
\begin{equation*}\begin{cases}\widetilde\omega_{\infty,1}=Vdx\wedge dy+dz\wedge\theta\\
\widetilde\omega_{\infty,2}=Vdy\wedge dz+dx\wedge\theta	\\
\widetilde\omega_{\infty,3}=Vdz\wedge dx+dy\wedge\theta,		
\end{cases}
\end{equation*}
where $\theta$ is the 1-form dual to $\p_t$. See Section 2 of \cite{GW} or Section 2 of \cite{HSVZ} for more details. By the discussion in Section \ref{ss:2-1}, the renormalized limit measure $\nu_\infty$ has  the expression 
\begin{align}\label{e:measure in codimension 1 case}d\nu_\infty=c\cdot e^{-f} \dvol_{g_\infty}, \quad f=\frac{1}{2}\log V, \quad c\in \dR_+.
\end{align}
  Moreover, the Bakry-\'Emery-Ricci tensor is nonnegative:
\begin{align*}\Ric^1_{g_\infty}\equiv\Ric_{g_\infty}+\nabla^2_{g_\infty}f- df\otimes df\geq 0. \end{align*}
An immediate consequence of \eqref{e:measure in codimension 1 case} is that the function $V$ is well-defined  up to a global multiplicative constant on each connected component of $\mathcal R$.
Fixing a choice of $V$ determines the Killing field $\p_t$,  hence the exact frame $\{dx, dy, dz\}$, up to multiplication by $\pm 1$. In particular, $\mathcal R$ is endowed with an affine structure with monodromy contained in $\dR^3\rtimes \dZ_2\subset \Aff(\dR^3)$.
It is easy to see that $V$ is a harmonic function on $\widetilde B_\infty$. Hence on $\mathcal R$, we have 
$\Delta_{\nu_\infty}V=0$.  

\begin{definition}
A special affine structure on a 3-manifold $Y^3$ is an affine structure  with monodromy contained in $\dR^3 \rtimes \dZ_2 $.
\end{definition}
In particular, a special affine structure determines a flat Riemannian metric $g^{\flat}$ on $Y^3$ up to constant multiplication. In local special affine coordinates $(x, y, z)$ we have that $g^{\flat}=C(dx^2+dy^2+dz^2)$.  
\begin{definition}
	A function $u$ on a special affine 3-manifold is harmonic if $\Delta_{g^\flat}u=0$.
\end{definition}

\begin{definition}
	A special affine  metric on a 3-manifold $Y^3$ consists of a special affine structure together with a smooth Riemannian metric $g$ such that $g=Vg^\flat$ for a positive harmonic function $V$ on $Y^3$. \end{definition}
	Here $V$ is well-defined up to constant multiplication. A choice of $V$ determines the flat metric $g^\flat=V^{-1}g$, which we call the \emph{flat background geometry}; it also yields a measure $\nu$ with density $d\nu=V^{-1/2}\dvol_{g}$, so as a metric measure space we have $\Ric^1_g \geq 0$.  Our discussion above shows 
\begin{proposition}\label{p:special affine 3 manifold}
	In the case $d=3$, $\mathcal R$ is endowed with a  special affine metric structure.
\end{proposition}
Notice that if we perform hyperk\"ahler rotations, i.e., changing the choice of hyperk\"ahler triples on each $X_j^4$, then the resulting metric is unchanged but the affine structure undergoes a rotation in $\SO(3)$.

\subsubsection{Convergence of special affine metrics}
Now we discuss the convergence of special affine metrics. Suppose we are given a  sequence of special affine metrics $(Y_i^3, g_i,p_i)$ such that $\overline{B_2(p_i)}$ is compact. We can first normalize the harmonic function $V_i$ on $Y_i$ by requiring $V_i(p_i)=1$. This fixes the measure $d\nu_i=V_i^{-1/2}\dvol_{g_i}$. Since $\Ric_{g_i}^1\geq 0$, by Theorem \ref{l:Yau gradient estimate}, we have $0<C^{-1}\leq V_i\leq C$  uniformly on $B_{3/2}(p_i)$. Hence the diameter of $B_1(p_i)$ with respect to the flat background metric $g_i^\flat$ is  also uniformly bounded above and below. Then by passing to a subsequence, we have $(B_1(p_i), g_i^\flat, p_i)\xrightarrow{GH}(Z_\infty,  d_{\infty}, p_\infty)$.

If $\dim_{\ess} Z_\infty=3$, then it is a  flat 3-manifold. Since $\Delta_{g_i^\flat}V_i=0$, the uniform $L^\infty$ bound on $V_i$ gives uniform interior bounds on all derivatives. In particular,  passing to a further subsequence we may assume the local frames $\{dx_i, dy_i, dz_i\}$ converge smoothly to a limit, giving a special affine structure on $Z_\infty$, and the function $V_i$ converges smoothly to a limit harmonic function $V_\infty$. Globally it follows that in this case $Y_\infty^3$ is also a smooth 3 manifold with a special affine metric, and the convergence of $Y_i^3$ to $Y_\infty^3$ is smooth.

If $\dim_{\ess} Z_\infty<3$, then the flat metrics $g_i^\flat$ collapse. Using the fact that the monodromy is contained in $\dR^3 \rtimes \dZ_2 $ it is easy to see that for $i$ large,  $g_i^\flat$ is locally isometric to a product $\mathbb T^k\times \mathbb R^{3-k}$ ($k=1, 2$) for some flat torus $\dT^k$ and Euclidean space $\mathbb R^{3-k}$. Passing to local universal covers we may assume $V_i$ still converges smoothly. Notice that through each point in $B_1(p_i)\subset Y_i^3$ there is a unique flat totally geodesic $\mathbb T^k$ with respect to $g_i^\flat$, which are all isometric as the point varies. 

The upshot of the above discussion is that we have a good understanding of convergence of special affine metrics. In particular, we always have a priori interior curvature bound and its covariant derivatives.

\subsection{Case $d=2$}
\label{ss:3-2}
 \subsubsection{Geometric structure on $\mathcal R$}

In this case we have $G_\infty=\dR^2$.   Fixing a basis of $G_\infty$ yields two  Killing fields $\{\p_{t_1}, \p_{t_2}\}$ on $\widetilde B_\infty$ which  preserve the limit triple $\widetilde{\bm {\omega}}_\infty$. Choose  moment maps $x_{\alpha j}$ for the vector field $\p_{t_\alpha}$ with respect to the symplectic form $\widetilde\omega_\infty^j$.  Since $[\p_{t_1}, \p_{t_2}]=0$, we have $d(\mathcal L_{\p_{t_1}} x_{2 j})=\mathcal L_{\p_{t_1}}dx_{2 j}=0$, so $\p_{t_{1}} x_{2 j}$ is a constant. It then follows that there is a unit vector $\bm a=(a_1, a_2, a_3)\in \dR^3$ such that $\p_{t_1}\sum_j a_jx_{2 j}=0$. Rotating the hyperk\"ahler triple by an element in $SO(3)$ we may assume $\bm a=(1, 0,0)$. Then we have $\widetilde\omega_\infty^1(\p_{t_1}, \p_{t_2})\equiv0$; in other words, the $G_\infty$ orbit is Lagrangian with respect to $\widetilde\omega_\infty^1$. Notice the choice of $\bm a $ (hence of $\widetilde\omega_\infty^1$) is only unique up to $SO(2)$ rotation, but we will fix a choice in the following discussion.      Then we obtain local moment maps $\pi=(x_1, x_2): \widetilde B_\infty\rightarrow \dR^2$ for the $G_\infty$ action with respect to $\widetilde\omega_\infty^1$. We can view $(x_1, x_2)$ as local  coordinates on $\mathcal R$ which depend on the choice of the basis of $G_\infty$, so are well-defined up to $\dR^2\rtimes \GL(2;\dR)$ action.
 Denote  $$W^{\alpha\beta}\equiv\widetilde g_\infty(\p_{t_\alpha}, \p_{t_\beta}), \ \ \ \ 1\leq \alpha, \beta\leq 2,$$ and let $(W_{\alpha\beta})$ be the inverse matrix of $(W^{\alpha\beta})$. Clearly these descend to $\widetilde B_\infty/G_\infty\subset \mathcal R$. As in Section 2.5 of \cite{SZ}, it is easy to see using the hyperk\"ahler equation    that on $\widetilde B_\infty/G_\infty$ we have 
\begin{equation} \label{e:dimensional reduction}
	\begin{cases}
	\det(W_{\alpha\beta})=C>0,\\
	\p_{x_\gamma} W_{\alpha\beta}=\p_{x_\beta} W_{\alpha\gamma}, \ \ \ \  1\leq\alpha, \beta, \gamma\leq 2.
	\end{cases}
\end{equation}
Moreover, the metric $g_\infty$ is given by 
\begin{equation}g_\infty=W_{\alpha\beta} dx_\alpha dx_\beta+W^{\alpha\beta}\theta_\alpha\theta_\beta,
\label{e:semiflat metric formula}	
\end{equation}
where $\theta_\alpha$ is the dual 1-form of the Killing field $\p_{t_\alpha}$.

The second equation in \eqref{e:dimensional reduction} implies that locally we can write $(W_{\alpha\beta})$ as the Hessian $(\phi_{\alpha\beta})$ of a convex function $\phi$. We can rescale the coordinates $\{x_1, x_2\}$ simultaneously by a constant so that the constant $C=1$. In terms of the normalized local coordinates we obtain an affine structure on $\mathcal R$ with monodromy group contained in $ \dR^2\rtimes\SL(2;\dR)$,  and a Riemannian metric $g_\infty=\phi_{\alpha\beta}dx_\alpha dx_\beta$ with $\det(\phi_{\alpha\beta})=1$. 

The discussion in Section 2 implies that in this case the renormalized limit measure $\nu_\infty=dx_1\wedge dx_2$ is simply the volume measure of $g_\infty$. Moreover, we have 
$\Ric_{g_{\infty}}\geq 0$. 
Since $\mathcal R$ has real dimension 2, the metric $g_\infty$ defines a complex structure $J$ on $\mathcal R$ via the Hodge star operator:
$$J dx_1=-\phi^{12}dx_1+\phi^{11}dx_2, \ \ J dx_2=-\phi^{22}dx_1+\phi^{12}dx_2. $$
The corresponding K\"ahler form is 
$\omega=dx_1\wedge dx_2$.
Let $\nabla$ be the flat connection defined by the affine structure. Then we have $\nabla\omega=0$ and 
$d^\nabla J=0$,
where $d^\nabla J(\xi, \eta)\equiv (\nabla_\xi J)(\eta)-(\nabla_\eta J)(\xi)$. Now recall 
\begin{definition}[\cite{Freed}]
	A special K\"ahler manifold is a K\"ahler manifold $(M, \omega, I)$ together with a torsion-free flat symplectic connection $\nabla$ satisfying $d^{\nabla}I=0$.
\end{definition}

Therefore, we have proved the following. 
\begin{proposition}\label{p:special Kahler structure}
	If $d=2$, then $\mathcal R$ is endowed with a special K\"ahler structure.
\end{proposition}
\begin{remark}
Notice the construction depends on the symplectic form $\widetilde{\omega}_\infty^1$ we choose at the beginning. 	
\end{remark}

By \cite{Freed}, once we fix the choice of local affine coordinates $x_1, x_2$,  there is a pair of conjugate special holomorphic coordinates $z, w$, such that $\text{Re}(z)=x_1$ and $\text{Re}(w)=-x_2$. They are unique up to transformations $z\mapsto z+c$, $w\mapsto w+c'$ for $c, c'\in \sqrt{-1}\mathbb R$. With respect to these special holomorphic coordinates, the monodromy is then contained in $\sqrt{-1}\mathbb R^2\rtimes \SL(2;\mathbb R)$.
Moreover, the K\"ahler form can be written as 
\begin{equation}\label{e:special kahler formula}\omega=\frac{\sqrt{-1}}{2} \text{Im}(\tau) dz\wedge d\bar z, 	
\end{equation}
where the local holomorphic function $\tau\equiv \frac{\p w}{\p z}$ satisfies  $\text{Im}(\tau)>0$.
 We can view $\tau$ as a multi-valued holomorphic map from $Z$ to the upper half  plane $\mathcal H$.

If we go around a loop $\gamma$, then we obtain new local special holomorphic coordinates $(\widetilde z, \widetilde w)$, with affine transformation given by
\begin{equation}\label{e:affine monodromy}
	\begin{pmatrix}
\widetilde z\\
\widetilde w
\end{pmatrix}=\begin{pmatrix}
a & b \\
c & d 
\end{pmatrix}
\begin{pmatrix}
z\\
w
\end{pmatrix}+\begin{pmatrix}
c_1\\
c_2
\end{pmatrix}.
\end{equation}
In particular,
\begin{equation*}
	\widetilde\tau=\frac{d\tau+c}{b\tau+a}.
\end{equation*}
For notational convenience we will call the following matrix the \emph{monodromy matrix} along $\gamma$
\begin{equation*}
	A_\gamma\equiv\begin{pmatrix}
d & c \\
b & a
\end{pmatrix}\in SL(2;\mathbb R)
\end{equation*}
 So we have a monodromy representation $$\rho: \pi_1(\mathcal R)\rightarrow SL(2;\dR);\ \gamma\mapsto A_\gamma,$$ which is well-defined up to conjugation, i.e., up to the choice of the base point and the local special holomorphic coordinates around the base point.

Conversely, suppose we are given a Riemann surface $Z$  with a K\"ahler metric $\omega$. If we can find local conjugate holomorphic coordinates $(z, w)$ with $\tau=\frac{\p w}{\p z}$ satisfying $\text{Im}(\tau)>0$, such that \eqref{e:special kahler formula} holds and the monodromy for $(z, w)$ along any loop is contained in $SL(2; \dR)$. Then there is a unique special K\"ahler structure on $Z$ associated to the metric $\omega$, so that $\text{Re}(dz)$ and $\text{Re}(dw)$ are parallel with respect to the associated torsion-free connection $\nabla$.

\begin{definition}
Let $M$ be a special K\"ahler Riemann surface. We say
	\begin{itemize}
	\item $M$ has \emph{integral monodromy} if the monodromy representation $\rho: \pi_1(M)\rightarrow SL(2;\dR)$ is conjugate to a representation in $SL(2;\dZ)$. 
		\item $M$ has \emph{local integral monodromy} if the monodromy matrix associated to each loop $\gamma$ in $M$ is conjugate to an element in $SL(2;\dZ)$. 
	\end{itemize}
  \end{definition}
In general the two notions are not equivalent; see Remark \ref{r:integeral monodromy}. 
Now we give some singularity models.
\begin{example}\label{example: cone-special-kahler}
  A flat metric cone  $$\omega=\frac{\sqrt{-1}}{2}\beta^2|\zeta
	|^{2\beta-2}d\zeta\wedge d\bar \zeta, \ \ \  \beta\in (0, 1)$$ on $\dC^*$ induces a natural special K\"ahler structure, with local special holomorphic coordinates given by $z=\zeta^{\beta}$ and $w=\sqrt{-1}z$ and $\tau\equiv \sqrt{-1}$. The monodromy matrix around the generator of $\pi_1(\mathbb C^*)$ is \begin{equation*}
	R_\beta=\begin{pmatrix}
\cos(2\pi\beta) &-\sin(2\pi\beta)\\
\sin(2\pi\beta) & \cos(2\pi\beta)
\end{pmatrix}
\end{equation*} 
We denote by \emph{$\textbf{C}_\beta$} such a special K\"ahler cone. By \cite{Freed} the cotangent bundle $T^*\textbf{C}_\beta$ admits a canonical flat hyperk\"ahler metric.
\end{example}
\begin{remark}\label{rmk:flat special kahler}
 \cite{Freed} showed that  on a special K\"ahler manifold, there is a globally defined holomorphic cubic differential, given in local special holomorphic coordinates by $\Theta=\frac{\p \tau}{\p z} dz^{\otimes 3}$.  Moreover, the scalar curvature satisfies $S=4|\Theta|^2$. In particular, $\Theta\equiv0$ if and only if the metric is flat.   Using this one can see that if a special K\"ahler metric is flat, then  the flat symplectic connection $\nabla$ agrees with the Levi-Civita connection. So the special K\"ahler structure is uniquely determined by the flat K\"ahler metric itself. 
\end{remark}
\begin{example}
Consider the metric on the punctured unit disk $(\mathbb D^*, \zeta)$ given by $$\omega=-\frac{\sqrt{-1}}{4\pi}\log |\zeta| d\zeta\wedge d\bar \zeta.$$ This has a global special holomorphic coordinate $z=\zeta$, and local conjugate coordinate $w=-\frac{\sqrt{-1}}{2\pi} z\log z+\frac{\sqrt{-1}}{2\pi}z$. 
The period map is  $\tau=-\frac{\sqrt{-1}}{2\pi}\log \zeta$. The monodromy around the generator of $\pi_1(\mathbb D^*)$ is given by 

  \begin{equation*}
	I_1 =\begin{pmatrix}
1 & 1 \\
0 & 1
\end{pmatrix}
\end{equation*} 
The tangent cone is the flat space $\dR^2$ with standard special K\"ahler structure.  
\end{example}

\begin{example}
Consider the metric on $(\mathbb D^*, \zeta)$ given by $$\omega=-\frac{\sqrt{-1}}{32\pi} |\zeta|^{-1}\log |\zeta|  d\zeta\wedge d\bar \zeta.$$ We can use $z=\sqrt{\frac{\zeta}{2}}$ to be a local special holomorphic coordinate.  The period map is  $\tau=-\frac{\sqrt{-1}}{2\pi}\log \zeta$. The monodromy is given by 

  \begin{equation*}
	I_1^*=\begin{pmatrix}
-1 & -1 \\
0 & -1
\end{pmatrix}
\end{equation*} 
The tangent cone is the flat cone $\dR^2/\dZ_2$, with monodromy $R_{\frac{1}{2}}$. Indeed, the metric here is a $\dZ_2$ quotient of the metric in the previous example.

\end{example}

We will need the following elementary results on the classification of conjugacy classes in $SL(2;\dR)$ and in $SL(2; \dZ)$. 
\begin{lemma}

Let $A$ be an element in $SL(2;\dR)$. Then the following holds:
\begin{enumerate}
\item if $A$ is \emph{parabolic}, i.e.,  $|\Tr(A)|=2$, then $A$ is $SL(2;\dR)$ conjugate to $\Id$, $-\Id$, $I_1$, $I_1^{-1}$,  $I_1^*$, or  $(I_1^*)^{-1}$;
\item  if $A$ is \emph{elliptic}, i.e., $|\Tr(A)|<2$, then  $A$ is $SL(2;\dR)$ conjugate to  $R_\beta$ for some $\beta\in (0, 1)\setminus \{\frac{1}{2}\}$;
\item if $A$ is \emph{hyperbolic}, i.e., $|\Tr(A)|>2$, then $A$ is $SL(2;\dR)$ conjugate to  
\begin{equation*}
	D_r=\begin{pmatrix}
r &0\\
0 & r^{-1}
\end{pmatrix}
\end{equation*}
for some $r\notin\{0, 1, -1\}$. 
\end{enumerate}	
\end{lemma}
\begin{lemma}\label{l:integral monodromy classification}
Let  $A$ be an element in in $ SL(2;\dZ)$, then the following holds: \begin{itemize}
	\item if $A$ is elliptic, then $A$ is $SL(2;\dZ)$ conjugate to  one of the following
\begin{align} \label{e:rotation matrix admissible}
	\widetilde R_{1/4}\equiv\begin{pmatrix}
0 &-1\\
1 & 0
\end{pmatrix}; \quad \widetilde 	R_{3/4}\equiv	\begin{pmatrix}
0 &1\\
-1 & 0
\end{pmatrix}	; \quad  \widetilde R_{1/6}\equiv	\begin{pmatrix}
1 &-1\\
1 & 0
\end{pmatrix}; \nonumber\\
\widetilde R_{1/3}\equiv	\begin{pmatrix}
0 &-1\\
1 & -1
\end{pmatrix}	; \quad \widetilde  R_{2/3}\equiv	\begin{pmatrix}
-1 &1\\
-1 & 0
\end{pmatrix}		;\quad \widetilde  R_{5/6}\equiv\begin{pmatrix}
0 &1\\
-1 & 1
\end{pmatrix}.	
\end{align}
\item if $A$ is parabolic, then $A$ is $SL(2;\dZ)$ conjugate to one of the following
\begin{align*}
& I_n\equiv \begin{pmatrix}
1 &n\\
0 & 1
\end{pmatrix}(n\in \dZ); \quad     I_n^*\equiv \begin{pmatrix}
-1 &-n\\
0 & -1
\end{pmatrix}(n\in \dZ);
\\ 
& \widetilde{R}_1\equiv\Id;  \quad   \widetilde{R}_{1/2}\equiv-\Id.
\end{align*}
\end{itemize}
\end{lemma}

Notice that each $\widetilde R_\beta$ is $SL(2;\dR)$ conjugate to $R_\beta$, so it is a rotation of $\dR^2$ that preserves a lattice.

\subsubsection{Convergence of special K\"ahler structures}
Let $(M_i, p_i)$ be a sequence of 2 dimensional manifolds  with special K\"ahler metrics $(\omega_i, J_i)$, where $\overline{B_2(p_i)}$ is compact. Since the curvature is non-negative, passing to a subsequence we first obtain a  Gromov-Haudorff limit $(M_\infty, d_\infty, p_\infty)$. For our purpose we may assume $M_\infty$ is not a single point.

 Let $\widetilde U_i$ be the universal cover of $\overline{B_2(p_i)}$, endowed with the induced special K\"ahler structure. Then it has trivial monodromy representation. Let $\widetilde p_i$ be a lift of $p_i$, and let $(z_i, w_i)$ be a choice of special holomorphic coordinates on $\widetilde U_i$. Then we can write 
$$\widetilde\omega_i=\frac{\sqrt{-1}}{2} \text{Im}(\tau_i) dz_i\wedge d\bar z_i$$
for some  holomorphic function $\tau_i$.  Applying a linear transformation to $(z_i, w_i)$ by an element in $SL(2;\dZ)$ we may assume $\text{Re}(\tau_i(\tilde p_i))\in [-\frac12, \frac12]$ and $\text{Im}(\tau_i(\tilde p_i))\geq \frac{\sqrt{3}}{2}$. Then replacing $(z_i, w_i)$ by $(\lambda_i^{-1}z_i, \lambda_iw_i)$ for a suitable $\lambda_i>0$ we may further assume $\text{Im}(\tau_i(\tilde p_i))\in[\frac{\sqrt{3}}{2}, 1]$.

 Notice that $\text{Im}(\tau_i)$ is a positive harmonic function on $\widetilde U_i$, so by Theorem \ref{l:Yau gradient estimate} we have $|\log \text{Im}(\tau_i)|\leq C$ uniformly on $B_{3/2}(\widetilde p_i)$. Then on this ball the metric $\widetilde\omega_i$ is uniformly equivalent to the flat metric $\widetilde{\omega}_i^\flat= \text{Im}(\tau_i)^{-1}\widetilde\omega_i$. Clearly then we have local smooth convergence of $\widetilde\omega_i^\flat$ by identifying each ball $(B_1(\widetilde p_i),z_i)$ holomorphically with a domain in $(\dC,z)$. Then passing to a subsequence we may assume $\text{Im}(\tau_i)$ converges smoothly hence we obtain a smooth limit K\"ahler metric  $(B_1(\tilde p_\infty), \widetilde\omega_\infty$). 

 Since $\text{Im}(\tau_i)$ is harmonic and bounded, its derivative is uniformly bounded on $B_{1}(\tilde p_i)$.  Using the Cauchy-Riemann equation  and the assumption that  $\text{Re}(\tau_i(\tilde p_i))\in [-\frac12, \frac12]$,  one can show that $\text{Re}(\tau_i)$ is also uniformly bounded on $B_{1}(\tilde p_i)$, so passing to a further subsequence we may assume $\tau_i$ converges smoothly to a limit $\tau_\infty$. At the same time, since $\tau_i=\frac{\p w_i}{\p z_i}$, we may  ensure the holomorphic function $w_i$ also converges to a limit $w_\infty$. 
Then we can write 
$$\widetilde\omega_\infty=\frac{\sqrt{-1}}{2} \text{Im}(\tau_\infty) dz_\infty\wedge d\bar z_\infty.$$
In particular, there is a special K\"ahler structure on the limit space $B_1(\tilde p_\infty)$.

A consequence of the above discussion is that for a special K\"ahler metric we have a priori interior curvature bound as well as its covariant derivatives. It is worth mentioning that even though not needed in this paper, the above arguments hold for special K\"ahler metrics in any dimension.

Now we divide into two cases

\

\textbf{Case 1}: $\text{Vol}(B_2(p_i))\geq \kappa$ uniformly for some $\kappa>0$.  In this case $M_\infty$ is a smooth Riemann surface and $M_i$ converges smoothly to $M_\infty$ in the Cheeger-Gromov sense. We may assume the K\"ahler metrics $\omega_i$ converge smoothly to a limit $\omega_\infty$ on $M_\infty$.

 We claim that by passing to a further subsequence $M_\infty$ can be naturally endowed with a special K\"ahler structure. To see this, first we may find a $\delta>0$ such that $B_\delta(p_i)$ is diffeomorphic to a ball for all $i$ large. Then in the above discussion we can directly work with the ball $B_\delta(p_i)$, and find special conjugate holomorphic coordinates $(z_i, w_i)$ which converge to $(z_\infty, w_\infty)$ on  $B_\delta(p_\infty)$. Now for any $q\in M_\infty$ which is the limit of $q_i\in M_i$, we can choose a path $\gamma$ in $M_\infty$ connecting $p$ and $q$. Using the smooth convergence of $M_i$ to $M_\infty$, we may view $\gamma$ as a path $\gamma_i$ in $M_i$ (for $i$ large) connecting $p_i$ and $q_i$. Then we can analytically continue the special holomorphic coordinates $(z_i, w_i)$ along $\gamma_i$ to obtain special coordinates in a neighborhood of $q_i$. Applying the Harnack inequality for $\text{Im}(\tau_i)$ along $\gamma_i$, we see $|\log \text{Im}(\tau_i)|$ is uniformly bounded along $\gamma$, which implies a uniform bound of $|\nabla_{\omega_i} z_i|$ and $|\nabla_{\omega_i}w_i|$ along $\gamma$. Passing to a subsequence we may assume $(z_i, w_i)$ converges uniformly to $(z_\infty, w_\infty)$ along $\gamma$. In particular, $(z_\infty, w_\infty)$ serve as local conjugate holomorphic coordinates in a neighborhood of $q$. Now these coordinates depend on the homotopy class of $\gamma$. But the fact that $M_i$ has monodromy contained in $SL(2;\dR)$ implies the limit special holomorphic coordinates also have monodromy contained in $SL(2;\dR)$.  So we obtain a global special K\"ahler structure on $M_\infty$. 
 
By the above construction, we also have the convergence of the conjugacy classes of monodromy representations. More precisely, if we fix a choice of the monodromy representation $\rho_i: \pi_1(M_i;p_i)\rightarrow SL(2; \dR)$ (for example, by fixing a choice of special conjugate holomorphic coordinates in a neighborhood of $p_i$), then there exists some $P_i\in SL(2; \dR)$ such that 
 for every $\sigma\in \pi_1(M_\infty;p_\infty)$, 
\begin{equation}\lim_{i\rightarrow\infty} P_i A_{\sigma, i} P_i^{-1}=A_{\sigma, \infty}, 
\label{e:convergence of monodromy matrix}	
\end{equation}
where we denote by $A_{\sigma, i}$ and $A_{\sigma, \infty}$ the monodromy matrix of the special K\"ahler structure on $M_i$ and $M_\infty$ along the loop $\sigma$, respectively. An immediate consequence is that 
\begin{equation}\label{e: convergence of trace}\Tr A_{\sigma,\infty}=\lim_{i\rightarrow\infty} \Tr A_{\sigma, i}.
\end{equation}

\

\textbf{Case 2}: $\text{Vol}(B_2(p_i))\rightarrow 0$ as $i\rightarrow\infty$. Then we know $M_i$ collapses with locally uniformly bounded curvature along circle fibrations. Let $\sigma_i$ be a loop in $M_i$ corresponding to the collapsing circle fiber. Since on the local universal cover we have smooth convergence of the special K\"ahler metrics, it follows easily that we can find $P_i\in SL(2;\dR)$ with 
\begin{equation}\label{e:jumping monodromy}
	\lim_{i\rightarrow\infty}P_i A_{\sigma_i, i}P_i^{-1}=\text{Id}.
\end{equation}
In particular,  $A_{\sigma_i,i}$ must be conjugate to $\Id, I_1$, or $I_1^{-1}$ for $i$ large. 

\begin{remark}\label{r:integeral monodromy}
If $M_i$ has integral monodromy for all $i$, then in the above \emph{\textbf{Case 1}} the limit $M_\infty$ must have local integral monodromy. To see this,  we make a choice of local conjugate special holomorphic coordinates $(z_i, w_i)$ near $p_i$ such that the induced monodromy matrices along all the loops are integral. 
Then we look at more closely the above discussion. First by a transformation in $SL(2;\dZ)$ we may assume $\text{Re}(\tau_i(p_i))\in[-1/2, 1/2]$. Now if $\text{Im}(\tau_i( p_i)$ is bounded then we may take $P_i=\Id$ in the above, and it follows that $A_{\sigma, \infty}\in SL(2;\dZ)$ for all loops $\sigma$ based at $p_\infty$. In this case $M_\infty$ indeed has integral monodromy. If $\text{Im}(\tau_i(p_i))$ is unbounded, then we may take 
$$P_i=\begin{pmatrix}
	\lambda_i^{-1} & 0\\
	0& \lambda_i
\end{pmatrix},
$$
and we must have  $\lambda_i\rightarrow0$. Write 
	$$A_{\sigma, i}=\begin{pmatrix}
a_i & b_i\\
c_i & d_i
\end{pmatrix}\in SL(2;\dZ),\ \ \ \ A_{\sigma, \infty}=\begin{pmatrix}
a & b\\
c & d
\end{pmatrix}\in SL(2;\dR).$$
Then  \eqref{e:convergence of monodromy matrix} implies that for $i$ large $a_i\equiv a, d_i\equiv d$, and $\lim_{i\rightarrow\infty}\lambda_i^{2}c_i=c$, $\lim_{i\rightarrow\infty}\lambda_i^{-2}b_i= b$. So for $i$ large we must have $b_i=b=0$, $a_i=d_i=\pm 1$. Therefore we know $A_{\sigma, \infty}$ is parabolic hence is conjugate to a matrix in $SL(2;\dZ)$. 

On the other hand, without extra assumptions one can not expect  $M_\infty$ to have integral monodromy globally. For example,  consider the punctured domain $\Omega=\mathbb D\setminus \{0, 1/2\}$ endowed with the special K\"ahler metrics $$\omega_{m, n}=(-m\log |z|-n\log |z-\frac{1}{2}|)\frac{\sqrt{-1}}{2}dz\wedge d\bar z.$$ These obviously have integral monodromy. Now we take a sequence $m_j, n_j\rightarrow\infty$ such that the ratio $m_j/n_j$ converging to an irrational number. Then the limit of $m_j^{-1}\omega_{m_j, n_j}$ has local integral monodromy but the global monodromy is not integral. 

\end{remark}

\subsubsection{Singular special K\"ahler metric}
We refer the readers to \cite{Callies-Haydys, Haydys-Xu} for discussion on  local models of more general singularities of special K\"ahler metrics. 
For the convenience of our later discussion we introduce the notion of a singular special K\"ahler metric adapted to our context. 
\begin{definition} \label{def:singular special kahler metric}
A singular special K\"ahler metric on a 2 dimensional Riemann surface $M$ is a smooth special K\"ahler metric $\omega$ on $M\setminus\{p_1, \cdots, p_k\}$, such that near each $p_i$, there exists $\delta>0$ and a holomorphic embedding $B_\delta(p_i)\setminus \{p_i\}$ into a domain in $(\dC^*, \zeta)$ which extends to a topological embedding of $B_\delta(p_i)$ into $\dC$, such that one of the following holds
\begin{itemize}
	\item (Type I) $z=\zeta$ is a special holomorphic coordinate on $B_\delta(p_i)$, the local  period map is given by $\tau=-\frac{\sqrt{-1}}{2\pi}\log\zeta+f(\zeta)$ for $f$ holomorphic across $0$, and $$\omega=\frac{\sqrt{-1}}{4\pi}(-\log |\zeta|+\text{Im}(f))d\zeta\wedge d\bar \zeta.$$
	In this case the monodromy matrix around the counterclockwise generator of $\pi_1(B_\delta(p_i)\setminus \{p_i\})$	 is given by $I_1$. 
	\item (Type II) $z=\zeta^{1/2}$ is local special holomorphic coordinate, the local period map is of the form $\tau=-\frac{\sqrt{-1}}{4\pi}\log \zeta+f$ for $f$ holomorphic across 0, and  $$\omega=\frac{\sqrt{-1}}{32\pi} (-\log |\zeta|+\text{Im}(f)) |\zeta|^{-1} d\zeta\wedge d\bar \zeta.$$
		In this case the monodromy around the counterclockwise  generator of $\pi_1(B_\delta(p_i)\setminus \{p_i\})$	 is given by $I_1^*$.  
	\item (Type III) $\zeta=(z-\sqrt{-1}w)^{1/\beta}$ for local conjugate special holomorphic coordinates $(z, w)$ (for some $\beta\in \{\frac 12, \frac 13, \frac 23, \frac 14, \frac34, \frac 16, \frac 56\}$), and we may locally write 
$$\omega= \frac{\sqrt{-1}}{8} (1-|\xi|^2)\beta^2|\zeta|^{2\beta-2}d\zeta\wedge d\bar \zeta, $$
	where $\xi$ is a multi-valued holomorphic function and is related to the local period map $\tau$ by 
\begin{align*}\xi=\frac{\tau-\sqrt{-1}}{\tau+\sqrt{-1}}.\end{align*}
	Moreover, 
	\begin{itemize}
	\item if $\beta=\frac{1}{2}$, then $\xi$ is a holomorphic function of $\zeta$.
	\item if $\beta\in\{\frac{1}{4},  \frac{3}{4}\}$, then $\xi=F(\zeta)^{1/2}$ for a holomorphic function $F$ with $F(0)=0$; 
	\item if $\beta\in \{\frac{1}{6}, \frac{5}{6}, \frac{1}{3},  \frac{2}{3}\}$, then $\xi=F(\zeta)^{1/3}$ for a holomorphic function $F$ with $F(0)=0$.
	\end{itemize}
\end{itemize}
\end{definition}
 
 In particular, a singular special K\"ahler metric is asymptotic to one of the model singularities in Examples \ref{example: cone-special-kahler}, and has local integral monodromy. It is also easy to check that there is always a unique tangent cone at the singularity given by a flat cone of angle in $(0,2\pi]$. More general examples of singular special K\"ahler metrics satisfying the above conditions are given by the base of an elliptic fibration with singular fibers (see for example \cite{Hein}).

\subsection{Case $d=1$}\label{ss:3-3}
In this case, the group $G_\infty$ is either the abelian group $\dR^3$ or the Heisenberg group $\mathscr{H}_1$ with Lie algebra $\mathfrak{h}_1$, where 
\begin{equation*}
\mathscr{H}_1 \equiv\left\{ \begin{pmatrix}
1 & x & t\\
0 & 1 &y\\
0 & 0& 1
\end{pmatrix}: x,y,t\in\dR\right\}, \quad \mathfrak{h}_1 \equiv\left\{ \begin{pmatrix}
0 & x & t\\
0 & 0 &y\\
0 & 0& 0
\end{pmatrix}: x,y,t\in\dR\right\}.
\end{equation*}
 This case has already been studied in detail in \cite{HSZ}. The analysis  is simpler than the discussion in the case $d>1$. We briefly summarize the results here, and  the readers may refer to \cite{HSZ} for proofs. 

If $G_\infty=\dR^3$, then the local universal cover converges to the flat metric on $\dR^4=\dR\times \dR^3$. The limit metric is locally isometric to a one dimensional interval $(a, b)_z$ endowed with the standard metric $g_\infty=dz^2$ and the renormalized limit measure $\nu_\infty=cVdz$ for a constant $c>0$, and $V\equiv 1$. 

If $G_\infty=\mathscr{H}_1$, then the local universal cover converges to $\mathscr{H}_1\times (a,b)_z$, and the limit hyperk\"ahler metric is given by applying the Gibbons-Hawking ansatz to a linear function $V=z+l$ with $l\in \dR$, such that $V=|\p_t|^{-2}$ for some generator $\p_t$ of the center $\mathfrak{z}(\mathfrak{h}_1)$.  Here $z$ is the moment map for the action of the center $\mathfrak{Z}(\mathscr{H}_1)$, and is well-defined up to an affine linear transformation of the type $z\mapsto \lambda z+\mu$. The limit metric $g_\infty=Vdz^2$ and the renormalized limit measure $\nu_\infty=cVdz$ for a constant $c>0$. The Lemma below follows from direct computation, and we omit the details.
\begin{lemma}\label{l:second fund form formula}
The second fundamental form of the limit $\mathscr{H}_1$-fibers satisfies
\begin{align*}
|\II_\infty|= \frac{\sqrt{3}}{2}V^{-\frac{3}{2}},
\end{align*}
and the Bakry-\'Emery Laplace operator is given by $\Delta_{\nu_\infty}=V^{-1}\p_z^2$.
\end{lemma}

Notice since $d=1$, the limit space  $X_\infty$ globally must be a one dimensional manifold, possibly with boundary. The singular set $\mathcal S$ consists of finitely many points in $X_\infty$. The main result of \cite{HSZ} is that these local affine structures indeed patch together to define a global affine coordinate $z$ on $X_\infty$, such that $g=Vdz^2$ and $\nu_\infty=cVdz$, for a concave piecewise linear function $V=V(z)$. Furthermore, in \cite{HSZ} some conjectures are posed on the structure of $V$ in the case when $X_\infty$ is the collapsing limit of hyperk\"ahler metrics on the K3 manifold. Odaka \cite{Odaka-density} and Oshima \cite{Oshima} have made connections with the algebro-geometric study of Type II degenerations of K3 surfaces.  

\subsection{$\epsilon$-regularity theorem}
\label{s:3-4}

The following  was proved by Cheeger-Tian \cite{CT} for general Einstein metrics in dimension $4$.  In the hyperk\"ahler setting, we provide a simple alternative proof, as an application of the study in this section. 

\begin{theorem}
[$\epsilon$-regularity theorem]
\label{t:epsilon regularity HK}
There are universal constants $\epsilon>0$, $\mathfrak C>0$ so that if  a hyperk\"ahler 4-manifold $(X^4, g, p)$ with $\overline{B_{10}(p)}$ compact  satisfies $\int_{B_{10}(p)}|\Rm_g|^2\dvol_g\leq \epsilon$,  then $\sup\limits_{B_1(p)}|\Rm_g|\leq \mathfrak C$. 
\end{theorem}
An immediate consequence is  
 \begin{corollary}\label{c:finite singular points}
 	Let $(X_j^4, g_j, p_j)$ be a sequence of hyperk\"ahler 4-manifolds converging to a  Gromov-Hausdorff limit $(X_\infty, d_\infty, p_\infty)$. If there exists  a $C>0$ such that $\int_{X_j^4} |\Rm_{g_j}|^2\dvol_{g_j}\leq C$ for all $j$, then the singular set $\mathcal S$ consists of at most finitely many points. 
 \end{corollary}

To prove the theorem, we denote $A=|\Rm_g(p)|^{1/2}$, and denote by $\mathcal G_p$ the set of all $q\in B_2(p)$ such that $|\Rm_g(q)|\geq A^2$ and $|\Rm_g(q')|\leq 4|\Rm_g(q)|$ for all $q'$ with $d(q',q)\leq A|\Rm_g(q)|^{-1/2}\leq 1$. The following  point-selection lemma is well-known. 
\begin{lemma}
We have $\mathcal G_p\neq\emptyset$.
\end{lemma}
\begin{proof}
If not, then we can find a sequence $q_j(j=0, 1,\cdots)$ with $q_0=p$, such that $d(q_{j+1}, q_{j})\leq |\Rm_{g}(q_j)|^{-1/2}\cdot A$, and $|\Rm_{g}(q_{j+1})|\geq 4|\Rm_g(q_j)|\geq 4A$. So we have $|\Rm_g(q_{j})|\geq 4^{j}A^2$, and $d(q_j, p)<2$ for all $j$. Clearly we get a contradiction if $j\rightarrow\infty$. 
\end{proof}

\begin{lemma}\label{l:kappa noncollapsing}
There exists $A_0>0$ and $\kappa>0$, such that if $|\Rm_g(p)|\geq A_0$, then for any $q\in \mathcal G_p$, 
we have $\Vol(B_{r}(q))\geq \kappa r^4$ with $r=|\Rm_g(q)|^{-1/2}$. 
\end{lemma}

\begin{proof}Suppose otherwise, then there is a  sequence $(X_j, g_j, p_j)$ and $q_j\in \mathcal G_{p_j}$ with $A_j\equiv |\Rm_{g_j}(p_j)|\rightarrow\infty$ and
$\text{Vol}(B_{Q_j^{-1/2}}(q_j))\leq j^{-1} Q_j^{-2},
$
where $Q_j\equiv |\Rm(q_j)|^{1/2}\geq A_j$. Then consider the rescaled sequence $(X_j, Q_j^2 g_j, q_j)$. Passing to a subsequence we obtain a  Gromov-Hausdorff limit $(X_\infty, q_\infty)$. By assumption we know $\dim X_\infty<4$. Moreover, for any fixed $R>0$, the collapsing is with curvature uniformly bounded by $4$ on $B_{R}(q_\infty)$. So by  Proposition \ref{p:regular set smooth}, $X_\infty$ is a complete Riemannian manifold and the limit geometry of the local universal cover around $q_j$ is not flat.    On the other hand, below we will show that the limit geometry of the local universal covers is everywhere flat, so a contradiction. We divide into 3 cases:

\textbf{Case (1)}: $\dim X_\infty=3$. By Proposition \ref{p:special affine 3 manifold}  we know $X_\infty$ is a special affine metric 3-manifold. In particular we know $\Ric^1(g_{X_\infty})\geq 0$. Let $V$ be the associated positive harmonic function on $X_\infty$. Since $X_\infty$ is complete, by the gradient estimate (Theorem \ref{l:Yau gradient estimate}) we know $V$ must be a constant, which implies $X_\infty$ is flat, and the limit geometry of local universal covers is flat. 

	\textbf{Case (2)}: $\dim X_\infty=2$. By Proposition \ref{p:special Kahler structure} we know $X_\infty$ is a complete special K\"ahler 2-manifold.  Lu's theorem \cite{Lu} implies that $X_\infty$ is flat,  hence the limit geometry of local universal covers is also flat. 
	
\textbf{Case (3)}: $\dim X_\infty=1$.   By the discussion in  Section \ref{ss:3-3}  we know $X_\infty$ is an interval equipped with an affine coordinate $z$ so that $g_\infty=Vdz^2$ and $d\nu_\infty=Vdz$ for a positive affine function $V$. Since $X_\infty$ is complete the interval has to be the entire $\dR$ but then the positivity of $V$ implies it must be a constant. Hence the limit geometry of local universal covers is again flat.   	\end{proof}

\begin{proof}[Proof of Theorem \ref{t:epsilon regularity HK}]
If not, then we have a sequence $(X_j^4, g_j, p_j)$ with $\int_{B_{9}(p_j)}|\Rm_{g_j}|^2\leq j^{-1}$, but $|\Rm_{g_j}(p_j)|\rightarrow\infty$. We choose $q_j\in \mathcal G_{p_j}$, and consider the rescaled sequence $(X_j, Q_j^2g_j, q_j)$, where $Q_j\equiv |\Rm_{g_j}(q_j)|$.  Passing to a subsequence it converges with uniformly bounded curvature to a limit $(X_\infty, g_\infty, q_\infty)$. Lemma \ref{l:kappa noncollapsing} implies $\dim X_\infty=4$ and $|\Rm_{g_\infty}(q_\infty)|=1$. By scaling invariance it follows that for any fixed $R>0$, $\int_{B_R(q_j, Q_j^2g_j)}|\Rm_{g_j}|^2\dvol_{g_j}\rightarrow 0$, so the limit metric metric $g_\infty$ must be flat. Contradiction. 
 \end{proof}

\subsection{Perturbation to invariant hyperk\"ahler metrics}\label{s-6}Now we go back to the set-up at the beginning of this section. Suppose a sequence of hyperk\"ahler manifolds $(X_j^4, g_j, \nu_j,p_j)$ converge in the measured Gromov-Hausdorff topology to a limit metric measure space $(X_\infty, d_\infty, \nu_\infty, p_\infty)$ with $d=\dim_{\ess}(X_\infty)<4$. The goal of this subsection is to  show that over the regular set $\mathcal R$, one can deform $g_j$ to a  nearby \emph{hyperk\"ahler} metric which exhibits local nilpotent symmetries of rank $4-d$. To prove this we need to combine the foundational results of Cheeger-Fukaya-Gromov with a quantitative implicit function theorem argument. The following is proved in \cite{CFG}, and we give an explanation in Appendix \ref{a1}.

\begin{theorem}[Regular fibration]\label{t:CFG} Let $\mathcal Q\subset\subset \mathcal R$  be a connected compact domain with smooth boundary. Then we can find $j_0=j_0(\mathcal Q)>0$ and a sequence $\tau_j\rightarrow 0$,  such that for all $j\geq j_0$, there exist a compact connected domain $\mathcal Q_j\subset X_j^4$ with smooth boundary, together with a smooth fiber bundle map 
$F_j: \mathcal Q_j\rightarrow \mathcal Q$  
such that the following properties hold.
\begin{enumerate}
\item $F_j: \mathcal Q_j \to \mathcal Q$ is a $\tau_j$-Gromov-Hausdorff approximation;
\item  For any $k\in\dZ_+$, there exists $C_k>0$ such that for all  $j\geq j_0$, we have 
	\begin{equation}\label{e:higher derivative bound on projection}
	|\nabla^k F_j|  \leq C_k;
	\end{equation}
 \item 
There exists a uniform constant $C_0>0$ such that for all $q\in\mathcal Q$ and $j\geq j_0$, we have\begin{align*}|\II_{F_j^{-1}(q)}|\leq C_0,\end{align*} 
where $\II_{F_j^{-1}(q)}$ denotes the second fundamental form of
 the fiber $F_j^{-1}(q)$ at $q\in {\mathcal Q}$.
		\item $F_j$ is an almost Riemannian submersion,  in the sense that for any vector $v$ orthogonal to the fiber of $F_j$, we have
	\begin{align}\label{e:almost riem submersion}
		(1-\tau_j)|v|_{g_j} \leq |dF_j(v)|_{g_\infty} \leq(1+\tau_j) |v|_{g_j};
	\end{align}
 	\item There are flat connections with parallel torsion on $F_j^{-1}(q)$, which depend smoothly on $q\in {\mathcal Q}$, such that each fiber of $F_j$
 affine diffeomorphic to an infranilmanifold $\Gamma\setminus N$, where $N$ is a simply-connected nilpotent Lie group, and $\Gamma$ is a cocompact subgroup of $N_L\rtimes \Aut(N)$ with $N_L\simeq N$ acting on $N$ by left translation. Moreover, the structure group of the fibration is reduced to  $((\mathfrak{Z}(N)\cap\Gamma)\setminus \mathfrak{Z}(N))\rtimes \Aut(\Gamma)\subset \text{Aff}(\Gamma\setminus N)$;
 \item  $\Lambda\equiv \Gamma \cap N_L$ is normal in $\Gamma$ with $\#(\Lambda\setminus\Gamma)\leq w_0$ for some constant $w_0$ independent of $i$.

	 \end{enumerate}
\end{theorem}

This is a special case of the {\it nilpotent Killing structure} ($\mathcal N$-structure) defined in \cite{CFG}. We say a tensor field $\xi$ on ${\mathcal Q}_j$ is $\mathcal N$-invariant if for any $x\in {\mathcal Q}$, there exists a neighborhood $U$ of $x$ with $F_j^{-1}(U)\cong U\times (\Gamma\setminus N)$, such that  the lift of $\xi$ to the universal cover $U\times N$ is $N_L$-invariant. Below we will construct an $\mathcal{N}$-invariant \emph{hyperk\"ahler} triple approximating the original hyperk\"ahler triple $\bm{\omega}_j$. First we have
\begin{proposition} \label{p:existence of local invariant definite triples}
For any sufficiently large $j$, there exists an $\mathcal N$-invariant  definite triple	$\bm\omega^{\dag}_j$ on ${\mathcal Q}_j$ such that $|\nabla^k_{\bm\omega_j}(\bm\omega^{\dag}_j-\bm\omega_j)|_{\bm\omega_j}\leq C_k\tau_j$, where $\tau_j\rightarrow 0$ is given by Theorem \ref{t:CFG}. 	
\end{proposition}
\begin{proof}
The construction is via the averaging argument as in Section 4 of \cite{CFG}, with the Riemannian metric replaced by the definite triple. Let $h\in N_L$ be any element and let $\tilde{v},\tilde{w}$ be any tangent vectors on the universal cover $U\times N$ of $F_j^{-1}(U)$. Let $\tilde{\bm{\omega}}_j$ be the lift of $\bm{\omega}_j$ on $U\times N$.
Then the function $h\mapsto \tilde{\bm{\omega}}_j(Dh\cdot \tilde{v}, Dh\cdot \tilde{w})$ is constant on each $\Lambda$-orbit in $N_L$. Since the nilpotent group $N_L$ is unimodular, there is a canonical bi-invariant measure $\tilde{\mu}$ on $N_L$ which descends to a unit-volume bi-invariant measure $\mu$ on $\Lambda\setminus N_L$. Therefore, \begin{align*}
\tilde{\bm{\omega}}_j'(\tilde{v},\tilde{w})\equiv\int_{\Lambda\setminus N_L}\tilde{\bm{\omega}}_j(Dh\cdot \tilde{v},Dh\cdot \tilde{w})d\mu 	
\end{align*}
is $N_L$-invariant on $U\times N$. We denote by $\bar{\bm{\omega}}_j'$ the descending $2$-form on $U\times (\Lambda\setminus N_L)$, and for any tangent vectors $\bar{v}, \bar{w}$ on $U\times (\Lambda\setminus N_L)$, we define
\begin{align*}
\bar{\bm{\omega}}_j^{\dag}(\bar{v},\bar{w})\equiv \frac{1}{\#(\Lambda\setminus\Gamma)}\cdot \sum\limits_{\gamma\in\Lambda\setminus\Gamma}\bar{\bm{\omega}}_j'(D\gamma\cdot\bar{v}, D\gamma\cdot\bar{w}),
\end{align*} 
where $\#(\Lambda\setminus\Gamma) \leq w_0$ for some constant $w_0>0$ independent of $j$.
We claim the above $(\Lambda\setminus\Gamma)$-invariant form is $N_L$-invariant. In fact, let $\bar{\gamma}\in\Gamma$ be any lift of $\gamma\in\Lambda\setminus\Gamma$ to $\Gamma$. Since $\Gamma\leq N_L\rtimes\Aut(N)$,  for any $h\in N_L$, there is some element $\bar{h}\in N_L$ such that $\bar{\gamma}\cdot h = \bar{h}\cdot \bar{\gamma}$. Then it is easy to verify that $\bar{\bm{\omega}}_j^{\dag}(Dh\cdot \bar{v},Dh\cdot\bar{w})= \bar{\bm{\omega}}_j^{\dag}(\bar{v},\bar{w})$.

Now $\bar{\bm{\omega}}_j^{\dag}$ descends to an $\mathcal{N}$-invariant definite triple $\bm{\omega}_j^{\dag}$ on $\Gamma\setminus N_L$. 
Notice that the average of a closed form is still closed. The approximation estimate follows from Proposition 4.9 of \cite{CFG}.	
\end{proof}

It is clear  that the Riemannian  metric $g_{\bm\omega_j^{\dag}}$ determined by  the definite triple $\bm\omega_j^{\dag}$ is also $\mathcal N$-invariant. Also the estimates  \eqref{e:higher derivative bound on projection} and  \eqref{e:almost riem submersion}  continue to hold if replace $\bm\omega_j$ by $\bm\omega_j^{\dag}$. 
 
 \begin{theorem} \label{t:local inv hk triple}
	For all sufficiently large $j$, there is an $\mathcal N$-invariant hyperk\"ahler triple $\bm\omega_j^{\Diamond}$ on $\mathcal Q_j$ of the form $\bm\omega_j^{\Diamond}=\bm\omega_j^{\dag}+dd^*(\bm f_j\cdot\bm\omega_j^{\dag})$, where  $\bm f_j$ is  an $\mathcal N$-invariant $(3\times 3)$-matrix valued  function on $\mathcal Q_j$ satisfying that for all $k\in\dN$, \begin{equation}\label{e:convergence of derivatives}\sup_{\mathcal Q_j}|\nabla^k_{\bm\omega_j^{\dag}}\bm f_j|\rightarrow 0.
 \end{equation}
  In particular, $\bm\omega_j^{\Diamond}$ has the same Gromov-Hausdorff collapsed limit as $\bm\omega_j$.
\end{theorem}
 \begin{remark}
In \cite{CFG} (Open Problem 1.10), Cheeger-Fukaya-Gromov asked the question that when a sufficiently collapsed Riemannian metric satisfies extra properties such as being Einstein or K\"ahler, whether one can perturb it to be an $\mathcal N$-invariant Riemannian metric in the same category. The above Theorem can be viewed as giving an affirmative answer to this question in the setting of local 4 dimensional hyperk\"ahler structures. We mention that Huang-Rong-Wang \cite{HRW} made some related progress on this question of Cheeger-Fukaya-Gromov using Ricci flow. 	
 \end{remark}

 Before proving Theorem \ref{t:local inv hk triple} we make some preparations.
Denote by $g_{\mathcal Q,j}$ the quotient metric on $\mathcal Q$ induced by the metric $\bm{\omega}_j^{\dag}$, and by $ H_j$ the mean curvature vector of the fibers of $F_j$. Because of the $\mathcal N$-invariance we may view $ H_j$ as a vector field on the quotient $\mathcal Q$. Recall we have the density function $\chi$ on $\mathcal Q$ for the renormalized limit measure $\nu_\infty$, as given in Section \ref{ss:2-1}.
 \begin{lemma} \label{l: convergence of second fundamental form}
On $\mathcal Q$, the metrics $g_{\mathcal Q,j}$ converge smoothly to $g_\infty$ in the Cheeger-Gromov topology, and $ H_j$  converges smoothly to $\nabla_{g_\infty}\log\chi$.  
 \end{lemma}
\begin{proof}
Given a point $q\in \mathcal Q$, we can find a coordinate neighborhood $\mathcal O$ with local coordinates $u_1, \cdots, u_d$. Let  $\widehat {\mathcal O}_j$ denote the universal cover of $F_j^{-1}(\mathcal O)$ endowed with the pull-back metric $\tilde{g}_{\bm{\omega}_j^{\dag}}$. The deck transformation group of $\widetilde{\mathcal O}_j$ is $\Gamma$. Then by Theorem \ref{t:CFG} we know $(\widetilde{\mathcal O}_j, \tilde{g}_{\bm{\omega}_j^{\dag}}, \Gamma)$ equivariantly $C^k$-converges to a limit $(\widetilde{\mathcal O}_\infty, \tilde{g}_\infty, N)$ for any $k\in\dZ_+$.  By \eqref{e:almost riem submersion} and \eqref{e:higher derivative bound on projection} we may assume $\pi_j\circ F_j$ converges smoothly to a Riemannian submersion $F_\infty: \widetilde{\mathcal O}_\infty\rightarrow \mathcal O$. In particular,  $u_\alpha\circ F_j$ converges smoothly to $u_\alpha\circ F_\infty$. So for any $\alpha, \beta$, we have \begin{align*}\lim_{i\rightarrow\infty}\langle \nabla_{\tilde{g}_{\bm{\omega}_j^{\dag}}} (u_\alpha\circ F_j), \nabla_{\tilde{g}_{\bm{\omega}_j^{\dag}}} (u_\beta\circ F_j)\rangle=\langle \nabla_{\tilde g_\infty} (u_\alpha\circ F_\infty), \nabla_{\tilde g_\infty} (u_\beta\circ F_\infty)\rangle.\end{align*}It follows that the quotient metric $g_{\mathcal Q,j}$ converges smoothly to $g_\infty$ in the coordinates $\{u_\alpha\}$.

For the second statement, we notice that the second fundamental form $\Pi_j$ of fibers of $F_j$ can be computed in terms of the derivatives of $u_\alpha\circ F_j$. In particular $\Pi_j$ also converges smoothly to a limit $\Pi_\infty$, which is the second fundamental form of the fibers of $F_\infty$. So  the corresponding mean curvature vectors $  H_j$ converges to 
$ H_\infty$. It is an easy calculation that $ H_\infty$ descends to the vector field $\nabla_{g_\infty}\log\chi$ on $\mathcal Q$.
\end{proof}

\begin{proof}[Proof of Theorem \ref{t:local inv hk triple}]
We will apply Proposition \ref{p:implicit function theorem}.
As in Section \ref{ss:triple-deformation} we may identify an element $\xi\in\Omega^+(\mathcal Q_j)\otimes \dR^3$ with a $(3\times 3)$ matrix valued function $\bm f$ on $\mathcal Q_j$, and $\xi$ is $\mathcal N$-invariant if and only if $\bm f$ is $\mathcal N$-invariant hence descends to a function on $\mathcal Q$.  
We define the Banach space $\fA$ (reps. $\fB$) to be the completion of the space of $\mathcal N$-invariant elements in $\Omega^{+}({\mathcal Q}_j)\otimes \dR^3$ under  the  $C^{2,\alpha}_{g_{\mathcal Q,j}}$ (resp. $C^{\alpha}_{g_{\mathcal Q,j}}$) norm.
 Then by Proposition \ref{p:existence of local invariant definite triples} for $\eta>0$ small we know the map $\mathscr{F}: B_\eta(\bo)\subset \fA\rightarrow \fB$ as given by \eqref{e:230} is well-defined, and $\|\mathscr{F}(\bo)\|\leq C\tau_j$ for some constant $C>0$. 
 
 For any $\mathcal N$-invariant function $f$ on $\mathcal Q_j$, we have 
$\Delta_{\bm{\omega}_j^{\dag}}f=\Delta_{g_{\mathcal Q,j}}f+\langle  H_j, \nabla_{g_{\mathcal Q,j}}f\rangle. $
 As in Section \ref{ss:2-3}, the Bakry-{\'E}mery Laplace operator   on  $(\mathcal Q, g_\infty, \nu_\infty)$ is given by \begin{align*}\Delta_{\nu_\infty}f\equiv \Delta_{g_\infty}f+\langle \nabla_{g_\infty}\log \chi, \nabla_{g_\infty}f\rangle,\end{align*} where $\nu_{\infty}=\chi \dvol_{g_\infty}$.
Let $\mathscr L(\bm f)\equiv \Delta_{\nu_\infty}\bm f$, and $\mathscr N(\bm f)\equiv \mathscr F(\bm f)-\mathscr L(\bm f)$. Then using the above convergence and the definition of $\mathscr F$ it is easy to see that for $\bm f, \bm g\in B_{\eta}(\bm 0)\subset\fA$ we have 
 \begin{align*}\|\mathscr N(\bm f)-\mathscr N (\bm g)\|\leq (C\eta+\epsilon_j)\|\bm f-\bm g\|\end{align*}
 for some $\epsilon_j\rightarrow 0$. 
 On the other hand, by standard elliptic theory, 
there exists a bounded linear operator $\mathscr P:\fB\rightarrow \fA$ such that $\mathscr L\circ \mathscr P(\bm v)=\bm v$, 
and $\|\mathscr P\bm v\|\leq L\|\bm v\|$ for some $L>0$ and all $\bm v\in \fB$. 
So for $i$ large we may apply Proposition \ref{p:implicit function theorem} to get a solution $\bm f_j$ satisfying $\mathscr F (\bm f_j)=0$ such that  \eqref{e:convergence of derivatives} holds for $k=2$. 
 For $k\geq 2$, \eqref{e:convergence of derivatives} follows from standard elliptic estimates. 
\end{proof}

Now we draw a few consequences of Theorem \ref{t:local inv hk triple}. 

\begin{corollary}[Fibers are Nil]
	In the statement of Theorem \ref{t:CFG}, we may assume $\Gamma$ is contained in $N_L$ so that the collapsing fibers are nilmanifolds. 
	\label{c: no infranil}
\end{corollary}
\begin{proof}
Locally on a coordinate chart $\mathcal O\subset\mathcal Q$, we can trivialize the fibration as $\mathcal O\times (\Gamma\setminus N)$. On the universal cover $\widetilde{\mathcal O}_j$ of $F_j^{-1}(\mathcal O)$, the action of $\Gamma$ preserves the hyperk\"ahler triple $F_j^*(\bm\omega_j^{\Diamond})$. It also acts by affine transformations on $N$. On the other hand, $N_L$ acts transitively on the fibers of the local universal cover.  Given any $\phi\in \Gamma$, we can find an element $\psi\in N_L$ such that $\psi\circ\phi$ fixes a section of $F_j$. By Theorem \ref{t:local inv hk triple},  $\bm{\omega}_j^{\Diamond}$ is $\mathcal N$-invariant, and hence $\psi\circ\phi$ preserves the hyperk\"ahler triple $F_j^*(\bm{\omega}_j^{\Diamond})$.  As in the proof of Proposition \ref{p:regular set smooth} we know the fixed point set of $\psi\circ\phi$ is either isolated or open. Since it is not isolated, it follows that $\psi\circ\phi$ must be identity, hence $\phi\in N_L$.
\end{proof}

When $d=2$, by the discussion in Section \ref{ss:3-2}, the limit metric $g_\infty$ on $\mathcal R$ is special K\"ahler. 

\begin{corollary}[Local integral monodromy] \label{c:integral monodromy}
	 $g_\infty$ has local integral monodromy.
\end{corollary}

\begin{proof}
  For each  $j$, the metric $\bm{\omega}_j^{\Diamond}$ is $\mathcal N$-invariant. Locally consider a trivialization of the fibration $\mathcal O\times (\Gamma\setminus N)$ where $N$ is the abelian group $\dR^2$ and $\Gamma$ is a lattice. Choose an integral basis $\p_{t_1}, \p_{t_2}$ of $\Gamma$, then over the local universal cover as in Section \ref{ss:3-2} we may find moment maps $(x_1, x_2)$ for the symplectic form $\omega_{i, 1}$, which serve as local coordinates on $\mathcal Q$. These are not canonical but are unique up to 
  $\dR^2\rtimes SL(2;\dZ)$.  This shows that the quotient metric $g_{\mathcal Q,j}^{\Diamond}$ on $\mathcal Q$ is naturally a special K\"ahler metric with integral monodromy. 
   Then the conclusion follows from Lemma \ref{l: convergence of second fundamental form} and the discussion in Remark \ref{r:integeral monodromy}. 
\end{proof}

\section{Singularity structure I: Case $d=3$}
\label{s-4}

\subsection{Main results}

We first state the main results of this section.

\begin{theorem}[Local version] \label{t:3d local version}
Let $(X_j^4, g_j, p_j)$ be a sequence of  hyperk\"ahler manifolds such that $\overline{B_2(p_j)}$ is compact and 
$(X_j^4, g_j, \nu_j,p_j) \xrightarrow{mGH} (X_\infty^3, d_{\infty}, \nu_\infty, p_{\infty})$
 with $\dim_{\ess}(X_\infty^3)=3$. Assume that the singular set $\mathcal{S}$ consists of a single point $p_\infty$. Then the following holds. \begin{enumerate}
\item $p_\infty$ is a conical singularity. More precisely, there exists a $\delta>0$ such that the corresponding flat background geometry $(B_\delta(p_\infty)\setminus\{p_\infty\}, g_\infty^{\flat})$ is isomorphic to a  punctured neighborhood of the origin in $\mathbb R^3$ or $\mathbb R^3/\dZ_2$, and $g_\infty=Vg_{\infty}^{\flat}$ for a  smooth  positive harmonic function of the form $V=\sigma r^{-1}+V_0$, where $\sigma\in[0,\infty)$ and $V_0$ is orbifold smooth. 
\item If in addition 
 \begin{equation}\label{e:curvature local  l2 bound}\int_{B_2(p_j)}|\Rm_{g_j}|^2\dvol_{g_j}\leq\kappa_0	
 \end{equation}
 uniformly for some $\kappa_0>0$, then $p_\infty$ is an orbifold singularity, i.e., the function $V$ in item (1)  is orbifold smooth and $\sigma=0$. 
 \end{enumerate}

\end{theorem}

\begin{remark}
It is not hard to see that \eqref{e:curvature local  l2 bound} is equivalent to a uniform bound on the Euler characteristic.  Notice that without assuming \eqref{e:curvature local  l2 bound} the constant {$\sigma$} does not have to vanish. As an example, consider the flat orbifold $Y_k=\mathbb C^2/\mathbb Z_{k+1}$, where $\mathbb Z_{k+1}\subset SU(2)$ is the standard diagonal subgroup acting on $\mathbb C^2$. As $k$ tends to infinity, $Y_k$ collapses to $(\dR^3, \frac{1}{2r} g_{\dR^3})$. Now let $X_k$ be the minimal resolution of $Y_k$ endowed with an ALE hyperk\"ahler metric $g_k$ such that the exceptional set has diameter comparable to $k^{-1}$. Choose a point $p_k$ on the exceptional set in $X_k$, then $(X_k, g_k, p_k)$ also collapses to $(\dR^3, \frac{1}{2r} g_{\dR^3}, 0)$. Here we have $\chi(X_k)=k+1\rightarrow\infty$. In this example we also see an infinite bubble tree of ALE gravitational instantons. We will show that the above cannot occur under the assumption \eqref{e:curvature local  l2 bound}; see Proposition \ref{p:maximal bubble}. 
\end{remark}
\begin{theorem}[Compact version]\label{t:global}
Let $g_j$ be a sequence of hyperk\"ahler metrics on the K3 manifold $\cK$ with $\diam_{g_j}(\mathcal{K})=1$ 
such that
$(\cK, g_j,\nu_j) \xrightarrow{mGH} (X_{\infty}^3, d_{\infty}, \nu_{\infty})$. 
Then $(X_{\infty}^3, d_{\infty})$ is isometric to a flat orbifold $\dT^3/\dZ_2$
and $\nu_{\infty}$ is a multiple of the Hausdorff measure on $\dT^3/\dZ_2$.
\end{theorem}

\begin{theorem}[Complete version]
\label{t:complete-space}
Let $(X_j^4, g_j, p_j)$ be a sequence of   hyperk\"ahler  manifolds such that 
$(X_j^4, g_j, \nu_j, p_j) \xrightarrow{mGH} (X_{\infty}^3, d_{\infty}, \nu_{\infty}, p_{\infty})
$. Assume that $X_\infty^3$ is complete non-compact and  the singular set $\mathcal S$ is finite. Then the following holds. 
\begin{enumerate}
	\item  The corresponding flat background geometry of $X_{\infty}^3$ is a complete flat orbifold of the form $\dR^3/\Gamma$, where $\Gamma$ is a subgroup of $\dR^3\rtimes\dZ_2$.  More precisely, we have the following classification (in terms of  the asymptotic volume growth):
\begin{enumerate}
\item[(a)] Euclidean space $\dR^3$, and its quotient $\dR^3/\dZ_2$.
\item[(b)] flat product $\dR^2\times S^1$,  
and its quotient  $(\dR^2\times S^1)/\dZ_2$,
\item[(c)] flat product
 $\dR\times \dT^2$, and its quotient $(\dR\times \dT^2)/\dZ_2$. 
\end{enumerate}
\item In Case (a) the positive harmonic function $V$ is of the form $\sigma r^{-1}+c$ with $\sigma\geq 0, c\in \dR$; in Case (b) and (c), $V$ must be a constant;
\item Assume that
\begin{align} \label{e:complete L2 bound}
\int_{X_j^4}|\Rm_{g_j}|^2\dvol_{g_j} \leq \kappa_0,
\end{align}
uniformly for some $\kappa_0>0$. Then $V$ must be a constant. 
\end{enumerate}
\end{theorem}

\subsection{Asymptotic analysis near the singularity}

\label{ss:4-2}

Now we focus on the local situation in the setting of Theorem \ref{t:3d local version}. 
The discussion in Section \ref{ss:3-1} implies that there is a special affine metric $g_\infty$ on $B_2(p_\infty)\setminus\{p_\infty\}$. We fix a choice of the harmonic function $V$, and denote by $g_\infty^{\flat}\equiv V^{-1}g_\infty$ the flat background metric.  The main goal of this subsection is to obtain a lower bound of $V$ near $p_\infty$ (Corollary \ref{p:lower-bound-V}), which  gives control of the flat background geometry near $p_\infty$ (Proposition \ref{p:one-point-completion}).

We start with a simple lemma. The proof follows directly from the volume comparison theorem for the renormalized limit measure $\nu_{\infty}$.

\begin{lemma}\label{l:vitali}
For any $r\in (0,\frac{1}{10})$, consider the annulus $A_{r,2r}(p_{\infty})$ centered at the singular point $p_{\infty}$.
Let $\{x_{\alpha}\}_{\alpha=1}^N\subset A_{r,2r}(p_{\infty})$ be a $\frac{r}{4}$-dense subset such that $\{B_{r/20}(x_{\alpha})\}_{\alpha=1}^N$
are disjoint.
Then the following hold:
\begin{enumerate}
\item $A_{r,2r}(p_{\infty})\subset \bigcup\limits_{\alpha=1}^N B_{r/4}(x_{\alpha})\subset A_{\frac{r}{4},\frac{9r}{4}}(p_{\infty})$.

\item There is a uniform constant $N_0>0$ independent of $r$ such $ N \leq N_0$.
\end{enumerate}
\end{lemma}
Let $\mathcal{C}_1(r)$, $\mathcal{C}_2(r),\ldots, \mathcal{C}_{\ell}(r)$ be 
the  connected components of the union $\bigcup\limits_{\alpha=1}^N B_{r/4}(x_{\alpha})$. Obviously $\ell\leq N\leq N_0$.  The following is a direct application of Theorem \ref{l:Yau gradient estimate}.

\begin{lemma}
[Harnack inequality] \label{l:harnack-annulus}  There exists a uniform constant $c_0>0$ independent of $r$ and the choice of the covering, such that 
for any $x,y\in A_{r,2r}(p_{\infty})\cap\mathcal{C}_k(r)$ with $1\leq k\leq \ell$,
\begin{align}
c_0^{-1}\leq \frac{V(x)}{V(y)} \leq c_0.
\end{align}
\end{lemma}

\begin{proposition}\label{p:upper-bound-V}There exists a constant  $\ell_0>0$ such that 
$\sup\limits_{S_r(p_{\infty})}V\geq \ell_0 \cdot r^{\frac{3}{2}}$ for all $r\in(0, 1]$. 
\end{proposition}

\begin{proof}
 Suppose not, then there are a sequence of numbers $r_i\to 0$
such that
\begin{equation}\label{e:bounded-growth-r_i}
\sup\limits_{S_{r_i}(p_{\infty})} V \leq   r_i^{\frac{3}{2}}.
\end{equation}
Since $\Delta_{\nu_\infty}V=0$  on $A_{r_i,1}(p_{\infty})$,  applying Lemma \ref{l:harnack-annulus} and Theorem \ref{l:Yau gradient estimate}, we have that 
\begin{align*}
\sup\limits_{A_{r_i,2r_i}(p_{\infty})}|V|+r_i|\nabla_{g_\infty}V|\leq Cr_i^{\frac{3}{2}}.
\end{align*}
For any Lipchitz function $\phi$ with $\Supp(\phi) \subset A_{r_i,1}(p_{\infty})$, using integration by parts, \begin{align}
\int_{A_{r_i, 1}(p_{\infty})} \langle \nabla_{g_{\infty}} V , \nabla_{g_{\infty}} \phi\rangle_{g_{\infty}} d{\nu}_{\infty} = 0. \label{e:harmonic-distribution}
\end{align}
We choose a cut-off function $\chi_i$ with $\Supp(\chi_i)\in A_{r_i,1}(p_{\infty})$,  $\chi_i\equiv 1$ on $A_{2r_i, 1/2}(p_{\infty})$,  and
\begin{align*}\sup\limits_{A_{r_i,2r_i}(p_{\infty})} |\nabla_{g_{\infty}} \chi_i|_{g_{\infty}}\leq C\cdot r_i^{-1},\quad  \sup\limits_{A_{1/2,1}(p_{\infty})} |\nabla_{g_{\infty}} \chi_i|_{g_{\infty}}\leq C.\end{align*} 
Applying $\phi \equiv  \chi_i\cdot V$ to   \eqref{e:harmonic-distribution}, we obtain \begin{align}
 \int_{A_{2r_i, 1/2}(p_{\infty})}|\nabla_{g_{\infty}} V|_{g_{\infty}}^2 d{\nu}_{\infty}
  \leq   \int_{A_{r_i, 2r_i}(p_{\infty})\cup A_{1/2,1}(p_{\infty})} V\cdot |\nabla_{g_{\infty}} V|_{g_{\infty}} \cdot |\nabla _{g_{\infty}}\chi_i|_{g_{\infty}} d{\nu}_{\infty} 
  \leq   C .
\end{align}
Letting $r_i\to 0$, we find that \begin{equation}\label{e:finite W12 norm}\int_{B_{1/2}(p_\infty)\setminus\{p_\infty\}}|\nabla_{g_{\infty}} V|_{g_{\infty}}^2d\nu_{\infty}<\infty\end{equation} Now we claim $V$ is a harmonic function on $B_{1/2}(p_\infty)$, in the sense of Definition \ref{def:harmonic functions}. First, by Theorem 5.1 of \cite{Cheeger-Differentiability}, given a Lipschitz function, the minimal upper gradient can be characterized by the local slope.
Also applying Lemma 1.42 of \cite{Bjorn-Bjorn}, $\Mod_2(\{p_{\infty}\})=0$.  Therefore,
 the function $u: B_{1/2}(p_\infty)\rightarrow \dR\cup \{\infty\}$ defined by setting $u(x)\equiv |\nabla_{g_\infty}V|$ for $x\neq p_\infty$ and $u(p_\infty)=\infty$, is a minimal weak upper gradient of $V$ on $B_{1/2}(p_\infty)$. So \eqref{e:finite W12 norm} implies $V\in W^{1,2}(B_{1/2}(p_\infty))$, and the Cheeger energy is given by $\Ch(V)=\int_{B_{1/2}(p_\infty)}|\nabla_{g_\infty}V|^2d\nu_\infty$. Moreover, 
 applying similar arguments as in the proof of  \eqref{e:finite W12 norm},  one can see that \eqref{e:harmonic-distribution} implies that
 \begin{align*}
\int_{B_{1/2}(p_{\infty})}\langle\nabla_{g_\infty} V, \nabla_{g_\infty} \phi\rangle d{\nu}_{\infty}= 0,
\end{align*}
for any compactly supported Lipschitz function $\phi$ on $B_{1/2}(p_{\infty})$. 
This proves the claim. Now  by Theorem \ref{t:weak-harnack} we obtain 
\begin{align}
\underset{B_{1/4}(p_{\infty})}{\ess\inf} V 	\geq C \cdot \Big(\int_{B_{1/2}(p_{\infty})} V^2 d\nu_{\infty}\Big)^{\frac{1}{2}} \geq c_0 > 0.
\end{align}
This contradicts \eqref{e:bounded-growth-r_i}.
\end{proof}

\begin{proposition}
\label{l:3d-tangent-cones}
Any tangent cone $Y$ at $p_{\infty}$ satisfies $\dim_{\ess}(Y)=3$. 
\end{proposition}
\begin{proof}
 We rule out the possible occurence of lower dimensional tangent cones. 
Suppose $(Y, \bar p)$ is a tangent cone at $p_\infty$ with $\dim_{\ess}(Y)\in \{1, 2\}$.    
Then we can find a sequence $r_j\rightarrow0$ such that the rescaled annulus $r_j^{-1}\cdot\mathcal{A}_j$  converges to an annulus in $Y$, where $\mathcal A_j\equiv A_{r_j, 2r_j}(p_\infty)$. By Proposition \ref{p:upper-bound-V} there is a point $q_j\in \mathcal A_j$ with $V(q_j)\geq l_0 r_j^{3/2}$. Without loss of generality we may assume  $q_j$ belongs to the connected component $\mathcal C_1(r_j)$ in the covering constructed in Lemma \ref{l:vitali}.
Denote $\mathcal A_j^1\equiv \mathcal C_1(r_j)\cap \mathcal A_j$. We may also assume the rescaled space $r_j^{-1}\cdot {\mathcal A}^1_j$ converges to a connected open set in  $Y$.  By Lemma \ref{l:harnack-annulus} we have $c_0^{-1}\leq V/V(q_j)\leq c_0$ uniformly on $\mathcal A_j^1$. This implies the  flat background metric $V(q_j)g_\infty^{\flat}=V(q_j)V^{-1}g_\infty$ on $\mathcal A_j^1$  is uniformly equivalent to $g_\infty$. Hence the corresponding rescaled sequence of flat manifolds $(\mathcal A_j^1, r_j^{-2}V(q_j)g_\infty^\flat)$ also collapses to a lower dimensional space. By the discussion in Section \ref{ss:3-1} we can find a  totally geodesic torus $\mathbb T_j\subset (\mathcal A_j^1, g_\infty^\flat)$ passing $q_j$, whose diameter with respect to the metric $r_j^{-2}V(q_j)g_\infty^\flat$ is $\epsilon_j\rightarrow 0$. So the diameter of $\mathbb T_j$ with respect to the metric $g_\infty^\flat$ is $\epsilon_j r_j V(q_j)^{-1/2}\rightarrow 0$. 

Now choose a point $w\in A_{1,2}(p_\infty)$ and a smooth curve $\gamma_j$ in $A_{r_j/2, 2}$ connecting $w$ and $q_j$. 
We can slide the  torus $\mathbb T_j$ along $\gamma_j$ and obtain a totally geodesic torus $\mathbb T'_j$ (with respect to $g_\infty^\flat$) passing through $w$. Notice in this process we can keep the family of flat tori along $\gamma_j$ to be outside $\mathcal A_{j+1}$ (in particular we do not encounter the singularity $p_\infty$). Since the diameter of the tori is invariant along the sliding, we then obtain a sequence of totally geodesic tori  contained in $A_{1,2}(p_\infty)$ with diameter going to zero. This is clearly impossible. 
\end{proof}

\begin{lemma}\label{l:connectedness}
	 There exists a constant $\delta_0>0$ such that for every $r\in(0,1/3)$,  any two points in $A_{r/2, r}(p_\infty)$  can be connected by a smooth curve $\gamma\subset A_{\delta_0\cdot r, 3r}(p_\infty)$ with  arc-length $|\gamma|\leq 10 r$. In particular, $B_1(p_\infty)\setminus\{p_\infty\}$ is path-connected. 
	 	\end{lemma}
\begin{proof} 
We argue by contradiction. 
Suppose there are sequences
 $\delta_j\rightarrow 0$, $r_j\in (0, 1)$ and  sequences of points  $x_j, y_j\in A_{r_j/2, r_j}(p_\infty)$
 such that any  smooth curve $\gamma_j$ connecting $x_j$ and $y_j$ with  $|\gamma_j|\leq 10 r_j$ satisfies $\gamma_j\cap B_{\delta_j\cdot r_j}(p_{\infty})\neq \emptyset$. 
Choose  minimizing geodesics $\sigma_{x_j}$ and $\sigma_{y_j}$ from $p_{\infty}$ to $x_j$ and $y_j$  respectively.
Then we take two points
$\underline{x}_j \in \sigma_{x_j}\cap S_{3\delta_j \cdot r_j}(p_{\infty})$ and $\underline{y}_j\in \sigma_{y_j}\cap S_{3\delta_j\cdot r_j}(p_{\infty})$. By assumption we have
\begin{enumerate}\item[(a)]  Any minimizing geodesic $\bar{\gamma}_j$ connecting $x_j$ and $y_j$ must satisfy $\bar{\gamma}_j\cap B_{\delta_j\cdot r_j}(p_{\infty})\neq\emptyset$
\item[(b)] 
Any smooth curve $\underline{\gamma}_j$ connecting $\underline{x}_j$ and $\underline{y}_j$ with length $|\underline{\gamma}_j|\leq 4 r_j$ must satisfy
 $\underline{\gamma}_j\cap B_{\delta_j\cdot r_j}(p)\neq\emptyset$.
\end{enumerate}
 Now we  define the rescaled metric $\widehat d_j\equiv \delta_j^{-1}\cdot r_j^{-1} \cdot d_\infty$. Letting $j\rightarrow\infty$ and passing to a subsequence, we obtain
 \begin{align}
 (X_{\infty}^3, \widehat{d}_j, p_{\infty}) \xrightarrow{GH} (\widehat{Z}_{\infty}, \widehat{d}_{\infty} , \widehat{z}_{\infty}),
 \end{align}
where $(\widehat{Z}_{\infty}, \widehat{d}_{\infty} , \widehat{z}_{\infty})$ is a tangent cone at $p_{\infty}\in X_{\infty}^3$, and $\underline{x}_j,\underline{y}_j$ converges to $\underline{x}_{\infty},\underline{y}_{\infty}\in S_3(\widehat{z}_{\infty})$  respectively. By the discussion in Section \ref{ss:3-1} the convergence is smooth away from $\widehat z_\infty$. 

 Now choose a sequence of minimizing geodesics $\bar{\gamma}_j$ connecting $x_j$ and $y_j$.  Since \begin{align}
\widehat{d}_j(x_j, p_{\infty})\geq  \frac{\delta_j^{-1}}{2} \quad \text{and} \quad \widehat{d}_j(y_j, p_{\infty})	\geq\frac{\delta_j^{-1}}{2}\end{align}
as $j\to +\infty$,  it follows from the Arzel\`a-Ascoli Lemma that by passing to a further subsequence, $\bar{\gamma}_j$ converges to a geodesic line $\bar{\gamma}_{\infty}\subset\widehat{Z}_{\infty}$. Applying Cheeger-Colding's splitting theorem, $\widehat{Z}_{\infty}$ is isometric to $\dR\times W$ for a complete length space $W$. 
 If $W$ is compact then we can slow down the rescaling slightly and obtain a tangent cone $\dR$ at $p_\infty$, which contradicts Proposition \ref{l:3d-tangent-cones}. So $W$ must be non-compact. Then it follows easily that the
the complement  $\widehat{Z}_\infty\setminus B_2(\widehat{z}_{\infty})$ is path connected. In particular, we can find a smooth curve  $\underline{\sigma}_{\infty}\subset \widehat{Z}_{\infty}\setminus B_2(\widehat{z}_{\infty})$ connecting $\underline{x}_{\infty}$ and $\underline{y}_{\infty}$. Denote $\ell_0=|\underline{\sigma}_{\infty}|$.
 Then passing back to the sequence, for $j$ large, we see
$\underline{x}_j$ and $\underline{y}_j$ can be connected by a smooth curve $\underline{\sigma}_j\subset A_{2\delta_j\cdot r_j , (l_0+10)\delta_j\cdot r_j}(p_{\infty})$ of length $|\underline{\sigma}_j|\leq (\ell_0+10)\cdot \delta_j \cdot r_j$. 
This contradicts item (b). 
\end{proof}

As an immediate consequence, we obtain an improvement of Lemma \ref{l:harnack-annulus} and Proposition \ref{p:upper-bound-V}.

\begin{corollary}\label{c:Harnack}
There exists a constant $C_0>0$ such that for any $r\in(0,1)$ and $x,y\in A_{r/2, r}(p_{\infty})$, we have
$C_0^{-1}\leq V(x)/V(y) \leq C_0$.
 \end{corollary}

\begin{corollary}\label{p:lower-bound-V}
There exists a constant  $\ell_0>0$ such that 
$\inf\limits_{S_r(p_{\infty})}V\geq \ell_0 \cdot r^{\frac{3}{2}}$ for all $r\in(0, 1]$. 
\end{corollary}

The above corollary immediately implies the following. 
\begin{proposition}\label{p:one-point-completion}The metric completion of the flat background  $(B_{1}(p_{\infty})\setminus \{p_{\infty}\}, V^{-1}g_\infty)$ at $p_\infty$ is given by adding a single point. 
\end{proposition}

In particular, we can identify this metric completion topologically as $B_1(p_\infty)$ itself, and we denote by $d_\infty^\flat$ the metric induced by the flat metric  $g_\infty^\flat$. 
At this point we encounter a non-standard \emph{singularity removal} question. 
\begin{question}\label{q:removable-singularity} Let $\mathcal{U}$ be a connected smooth  Riemannian manifold in dimension $m\geq 3$ with uniformly bounded sectional curvature. If the metric completion $\overline{\mathcal{U}}$ is obtained by adding {\it one point} $p$ such that $\overline{\mathcal{U}}$ is locally compact, is it true that $p$ is a Riemannian orbifold singularity?
\end{question}

 In our setting we are only interested in the special case when the Riemannian metric is flat. Even in this case the above innocent looking question seems to be subtle. There is an analogous statement when $m=2$, but one needs to allow  a general conical singularity.  Notice the conclusion fails if the metric completion is not locally compact; for an example in dimension 2, consider the universal cover of $\mathbb R^2\setminus\{0\}$ equipped with the flat metric.  In the next subsection we get around this technical point in our setting using the fact the conformal metric $g_\infty$ is a Ricci limit space.

\subsection{Proof of Theorem \ref{t:3d local version}}

By Lemma \ref{l:tangent cone compact}, the isometry classes $\mathcal{T}_{p_{\infty}}$ of all tangent cones at $p_\infty$ is compact in $(\mathcal{M}et, d_{GH})$.  
Let $(Y, p^*)\in\mathcal{T}_{p_{\infty}}$ satisfy  $(X_{\infty}, r_i^{-1}d_{\infty},p_{\infty})\xrightarrow{GH}(Y,d_Y,p^*)$ for $r_i\rightarrow0$. By the discussion at the end of Section \ref{ss:3-1} we know that away from $p^*$ the convergence is smooth and there is a special affine metric on $Y\setminus \{p^*\}$. Notice that, by Lemma \ref{l:connectedness}, for all $r>0$, any two points in $A_{r/2, r}(p^*)\subset Y$ can be connected by a smooth curve $\gamma\subset A_{\delta_0\cdot r, 3r}(p^*)$ with  arc-length $|\gamma|\leq 20 r$. In particular,
 $Y\setminus \{p^*\}$ is path-connected and $Y$ has only one end at infinity. By Proposition  \ref{c:Harnack}, the flat background $(Y\setminus\{p^*\},g_Y^{\flat})$ has a one point completion near $p^*$ and is homeomorphic to $Y$. We always normalize the harmonic function $\widehat{{V}}_Y^*$ by a multiplicative constant so that $\sup\limits_{S_1(p^*)} \widehat{{V}}_Y^* =1$.

\begin{lemma}
\label{l:V-growth-tangent-cone}
For every $\epsilon>0$, there exists
a tangent cone $(Y,p^*)$  such that 
 \begin{align}
\limsup\limits_{R\to+\infty}\frac{|\widehat{{V}}_Y^*|_{L^{\infty}(S_{R}(p^*))}}{R^{\frac{3}{2}+\epsilon}} \leq 2.
\end{align}

\end{lemma}

\begin{proof}
We argue by contradiction. Suppose the conclusion fails for some $\epsilon_0>0$. Then  for every tangent cone $(Y,p^*)$, we can find $R_Y>1$ such that 
\begin{align}
\sup\limits_{S_{R_Y}(p^*)}\widehat{{V}}_Y^* > 2 \cdot (R_Y)^{\frac{3}{2}+\epsilon_0}.\label{e:arbitrary-tangent-cones-V-growth}
\end{align}

 {\bf Claim}. Given a tangent cone $(Y, p^*)$, there exists $\tau=\tau(Y)>0$ such that
for any $(W,q^*)\in B_{\tau}((Y, p^*))\cap \mathcal T_{p_{\infty}}$, we have
$
	\sup\limits_{S_{R_Y}(q^*)}\widehat{{V}}_W^* > \frac{3}{2} \cdot R_Y^{\frac{3}{2}+\epsilon_0}.
	$
	Indeed, if not, then we can find a sequence of tangent cones $(W_i, q_i^*)$ converging to $(Y, p^*)$ in the  Gromov-Hausdorff topology, such that 	
$
	\sup\limits_{S_{R_Y}(q_i^*)}\widehat{{V}}_{W_i}^* \leq \frac{3}{2} \cdot R_Y^{\frac{3}{2}+\epsilon_0}.\
	$
	Applying the Harnack inequality and using the convergence of special affine metrics discussed in Section \ref{ss:3-1}, $\widehat{{V}}_{W_i}^*$ converges uniformly away from $p^*$ to $\widehat{{V}}_{Y}^*$. This contradicts \eqref{e:arbitrary-tangent-cones-V-growth}.

\vspace{0.5cm}
Since $\mathcal{T}_{p_{\infty}}$ is compact in $(\mathcal{M}et, d_{GH})$,  it can be covered by finite metric balls of the form $B_{\tau_{\ell}}((Y_{\ell}, p_{\ell}^*))$, $\ell=1, \cdots, N$. By \textbf{Claim}, it follows that for any $(Y, p^*)\in \mathcal{T}_{p_{\infty}}$, we have
$
\sup\limits_{S_{R_{Y_{\ell}}}(p^*)}\widehat{{V}}_Y^* > (R_{Y_{\ell}})^{\frac{3}{2}+\epsilon_0}$.
for some $1\leq \ell\leq N$. Then using a simple contradiction argument one can show that for all $0<r\ll1$, there exists $\ell_0\in \{1, \cdots, N\}$ with $R_0\equiv R_{Y_{\ell_0}}$ such that
$\sup_{B_r(p_\infty)} V< R_0^{-3/2-\epsilon_0}\sup_{B_{R_0\cdot r}(p_\infty)} V.$
By iteration we obtain a sequence $r_i\rightarrow 0$ with 
$\sup\limits_{S_{r_i}(p_\infty)} V\leq C r_i^{\frac{3}{2}+\epsilon_0}$, which contradicts Corollary \ref{p:lower-bound-V}.
\end{proof}

Now we fix $\epsilon=\frac{1}{4}$ and let $(Y,d_Y,p^*)$ be a tangent cone given in Lemma \ref{l:V-growth-tangent-cone}. 

\begin{proposition}\label{p:complete-background}
The associated flat background geometry on $(Y,d_Y, p^*)$ is  complete at infinity.  
\end{proposition}
\begin{proof}

For $R$ large, we consider the annulus $A_{R, 2R}(p^*)$ in $Y$ with respect to the metric $d_Y$. By Lemma \ref{l:V-growth-tangent-cone}, we have 
$\widehat{{V}}_Y^*\leq 8R^{\frac{7}{4}}$ on $A_{R, 2R}(p^*)$. Notice the flat background metric $g_Y^{\flat}=(\widehat{{V}}_Y^*)^{-1}g_Y$. Given any smooth curve $\gamma:[0, L]\rightarrow A_{R, 2R}(p^*)$ connecting $S_R(p^*)$ and $S_{2R}(p^*)$, which is parametrized by the arc-length with respect to the metric $g_Y^{\flat}$, its length with respect to $g_Y$ satisfies $L_{g_Y}(\gamma)\leq 4LR^{\frac{7}{8}}$. Since $L_{g_Y}(\gamma)\geq R$, we see $L\geq \frac{1}{4} R^{1/8}$. From this it is easy to draw the conclusion. 
\end{proof}

We will use the following classification result  for  flat ends of Riemannian manifolds.
\begin{theorem}
[Eschenburg-Schroeder \cite{ES}] \label{t:classification-flat-ends}
Let $Z$ be a flat end in a complete Riemannian manifold $(X^n,g)$. Then there exists a compact subset $K$ such that $Z\setminus K$ is isometric to the interior of $(\Omega\times\dR^k)/\Gamma$ 
and one of the following three cases hold: 
\begin{enumerate}
\item[(A)] $\dim(\Omega)=1$, $\Omega=\dR_+$ and $\Gamma$ is a Bieberbach group on $\dR^{n-1}$.
\item[(B)] $\dim(\Omega)=2$, $\Omega$ is diffeomorphic to $\dR\times \dR_+$ and $\Gamma$ is a Bieberbach group on $\dR\times \dR^{n-2}$ which preserves the Riemannian product structure of $\Omega\times \dR^k$.
\item[(C)] $\dim(\Omega)\geq 3$, $\Omega$ is the complement of a ball in $\dR^{n-k}$, and $\Gamma$ is a finite extension of a Bieberbach group on $\dR^k$.
\end{enumerate}

\end{theorem}

\begin{proposition}
\label{p:tangent cone flat geometry}
	$(Y, d_Y^{\flat})$ is isometric to  the Euclidean space $\dR^3$ or the flat cone $\dR^3/\dZ_2$, and $\widehat{{V}}_Y^*=1$. 
\end{proposition}
\begin{proof}[Proof of Proposition \ref{p:tangent cone flat geometry}] 
 Proposition \ref{p:one-point-completion} and  \ref{p:complete-background} imply that $(Y, d_Y^\flat)$ has one complete end at infinity. Since an asymptotic cone of $(Y, d_Y)$ is itself a tangent cone at $p_\infty$, by Proposition \ref{l:3d-tangent-cones} the asymptotic cones of $(Y, d_Y)$ must all be 3 dimensional, then by the discussion at the end of Section \ref{ss:3-1} we know the asymptotic cones  of $(Y, d_Y^\flat)$ are also 3 dimensional. Applying Theorem \ref{t:classification-flat-ends} to the end of $(Y, d_Y^\flat)$ we see that we are in  Case (C) and $\Gamma$ is finite (the other cases have collapsed asymptotic cones). So $(Y, d_Y^\flat)$ is isometric to either $\dR^3$ or $\dR^3/\dZ_2$ outside a compact set. Then using the developing map and the fact that the metric singularity of $(Y, d_Y^\flat)$ consists of at most one point, we conclude that $(Y, d_Y^\flat)$ must be isometric to $\dR^3$ or $\dR^3/\dZ_2$.
\end{proof}
Now we are ready to prove Theorem \ref{t:3d local version}.
\begin{proof}[Proof of Theorem \ref{t:3d local version}]
First we prove Theorem item (1). Let $(Y, p^*)$ be a tangent cone at $p_\infty$ whose flat background geometry $(Y, g_Y^\flat)$ is the flat cone $\dR^3$ or $\dR^3/\dZ_2$. For simplicity of notation we will assume $(Y, g_Y^\flat)$ is $\dR^3$. The other case can be dealt with in the same manner. 
We can find $r_i\rightarrow 0$ such that $(X_\infty^3, r_i^{-1}d_\infty, p_\infty)$ converges to $(Y, d_Y, p^*)$. As before we also have the convergence of the corresponding flat background geometry
$(X_\infty^3, r_i^{-1}d_\infty^\flat, p_\infty)\xrightarrow {GH} (\dR^3, g_{\dR^3}, 0). $
This means that the annulus $r_i^{-1} A_{r_i/2, 2r_i}(p_\infty)$ converges to the flat annulus $A_{1/2,2}(0)$ in $\dR^3$. In particular, we can find smooth hypersurfaces $\Sigma_i\in A_{r_i, 2r_i}(p_\infty)$ with constant curvature 1 such that $r_i^{-1}\Sigma_i$ converges to the unit sphere in $\dR^3$. Then by a simple argument using the developing map one can see that a punctured neighborhood of $p_\infty$ in $(X_\infty^3, d_\infty^\flat)$ can be isometrically embedded in $\dR^3$ as a punctured domain. So the flat background geometry is smooth near $p_\infty$. Now $V$ can be viewed as a positive harmonic function in a punctured domain in $\dR^3$. The singular behavior of $V$ then follows from the classical B\^ocher's theorem. This finishes the proof of item (1) of Theorem \ref{t:3d local version}.

The rest of this subsection is devoted to the proof of item (2) of Theorem \ref{t:3d local version}.
We already know the flat background geometry on the limit $X_\infty^3$ is a flat orbifold near $p_\infty$, and a neighborhood of $p_\infty$ can be identified with an open set in $\dR^3$ or $\dR^3/\dZ_2$. Moreover, the positive harmonic function $V$ is of the form $\sigma r^{-1}+V_0$, where $r$ is the radial function on $\dR^3$, $\sigma$ is a positive constant, and $V_0$ extends smoothly as an orbifold harmonic function.
 It suffices to show $\sigma=0$.
 
  Suppose $\sigma>0$. Notice item (1) in Theorem \ref{t:3d local version}  implies that the tangent cone $(Y, p^*)$ at $p_\infty$ is unique, and $(Y, d_Y^\flat)$ can be identified with $\dR^3$ or $\dR^3/\dZ_2$. After rescaling we may assume  $\widehat{V}_Y^*=\frac{1}{2r}$. Notice $Y$ is a metric cone over a round 2-sphere with radius $1/2$. We may identify the cross section of the cone with  $\Sigma=\{r=1/2\}\subset Y$. Let $B_\infty$ be a small tubular neighborhood of $\Sigma$. Then we can find a domain $U_i$ contained in $X_i^4$  that converges to $B_\infty$ with uniformly bounded curvature and all its covariant derivatives. Furthermore, by theorem \ref{t:CFG} there is a smooth fibration map $F_i: U_i\rightarrow B_\infty$ with fibers given by smooth circles with uniformly bounded second fundamental form and all covariant derivatives. Let $\Sigma_i=F_i^{-1}(\Sigma)$. Then $\Sigma_i$ collapses to $\Sigma$ along the circle bundle with uniformly bounded curvature and covariant derivatives.

 Given any point $q\in B_\infty$, by assumption, we can find $q_i\in U_i$ and $\delta>0$ such that the universal cover $\widetilde{B_{\delta}(q_i)}$ converges smoothly to a hyperk\"ahler limit $\widetilde B_\infty$, and a neighborhood of $q$ is given by the $\dR$ quotient of $\widetilde B_{\infty}$. Since $V=1/{2r}$, the limit metric on $\widetilde B_\infty$ is of the form 
$\frac{1}{2r}(dr^2+r^2g_{S^2})+2r\theta^2, $
 where $\theta$ is dual to the Killing field generating the $\dR$ action. Changing the coordinate by $r=\frac{1}{2}s^2$, one can see that this metric is flat. Moreover, the local universal covers of $\Sigma_i$
 converge to some subset of 
 the level set $\{r=1/2\}$ in $\widetilde B_\infty$ which has constant curvature 1. In particular  the sectional curvature of $\Sigma_i$ converges uniformly 1.    It follows from Klingenberg's estimate that   the universal cover $\widetilde{\Sigma}_i$ of $\Sigma_i$ has a uniform lower bound on the injectivity radius. This also implies that the universal cover $\widetilde{U}_i$ of $U_i$ converges smoothly to a flat manifold $\widetilde{U}_{\infty}$, and $\widetilde\Sigma_i$ converges smoothly to the round sphere $\widetilde\Sigma_\infty\subset \widetilde U_\infty$.
 Since the two boundary components of $\widetilde{U}_{\infty}$ 
 are convex, applying Sacksteder's theorem \cite{Sac}, $\widetilde{U}_{\infty}$
is isometric to a tubular neighborhood of the round sphere $\mathbb{S}^3$ in $\dR^4$. Denote $G_i\equiv\pi_1(U_i)$, then we have the following diagram 
\begin{align}
 \xymatrix{
(\widetilde{U}_i, \widetilde{g}_i, G_i) \ar[d]_{\pi_i} \ar[rr]^{eqGH} &&  (\widetilde{U}_{\infty}, \widetilde{g}_{\infty},G_{\infty}) \ar [d]^{\pi_{\infty}}
 \\
 (U_i, g_i)\ar[rr]^{GH} &&  (U_{\infty}, g_{\infty}),
 }\label{e:eq-GH-convergence}
\end{align}
where $G_{\infty}\leq \Isom(\widetilde{U}_{\infty})$ is a closed subgroup so that $U_{\infty}=\widetilde {U}_{\infty}/G_{\infty}$. 
 
 For our purposes we need to investigate more closely the above convergence. Notice we have fixed a choice of a hyperkahler triple  $\bm \omega_i$ on each $X_i$. Then we get a triple of 2-forms on $\widetilde{\bm \omega}_i$ on $\widetilde{U}_i$, and passing to a further subsequence we may assume these converge to a hyperk\"ahler triple $\widetilde{\bm \omega}_\infty$ on $\widetilde{U}_\infty$. Since the limit metric on $\widetilde{U}_\infty$ is flat, we may assume that, via the embedding $\widetilde{U}_{\infty}\hookrightarrow \dR^4$, $\widetilde{\bm \omega}$ is given by the restriction of the standard hyperk\"ahler triple on $\dR^4$. 
 Notice the $G_i$ action on $\widetilde{U}_i$ preserves $\widetilde {\bm \omega}_i$, so $G_\infty$ preserves the triple $\widetilde {\bm\omega}_\infty$. If follows that $G_\infty$ is contained in $\SU(2)=\Sp(1)$. 
 
 Now we can restrict our attention to the smooth convergence of  $\widetilde{\Sigma}_i$ to $\widetilde{\Sigma}_\infty$. 
 Since $G_{\infty}$ is a closed subgroup in the compact Lie group $\Aut(\widetilde{\Sigma}_{\infty},\widetilde{\bm{\omega}}_{\infty})=\SU(2)$,  
it follows that there is a group isomorphism $\varphi_i: G_i\simeq \overline G_i< G_{\infty}$  (see lemma 3.2 in \cite{MRW} for instance). 
Moreover, for any sufficiently large $i$, there exists a $G_i$-equivariant diffeomorphism $\mathscr{F}_i:  \widetilde{\Sigma}_i\to \widetilde{\Sigma}_{\infty}$  
such that the following holds.
\begin{enumerate}
\item  $\mathscr{F}_i \circ \gamma  =\varphi_i(\gamma) \circ \mathscr{F}_i$  for all $\gamma\in G_i$.
\item $\mathscr{F}_i$ is an $\epsilon_i$-Gromov-Hausdorff approximation with $\epsilon_i\to 0$. 
\item For any unit tangent vector $v$, 
 \begin{align}\label{e:c1 close}\Big||d\mathscr{F}_i(v)| - 1 \Big| \leq  \Psi(\epsilon_i), \quad \lim\limits_{\epsilon_i\to0}\Psi(\epsilon_i)=0.\end{align}	
\end{enumerate}
A key technique in constructing the above $G_i$-equivariant diffeomorphism $\mathscr{F}_i$ is to use the center of mass technique.   
We refer the readers to \cite{GK} and Theorem 2.7.1 in \cite{Rong-notes} for more details. 

Now we identify $\widetilde\Sigma_i$ with $\widetilde \Sigma_\infty$, and $G_i$ with $\overline G_i$ using $\mathscr{F}_i$.
Consider the form $\widetilde\omega_{\infty}^1$ on $\widetilde \Sigma=\mathbb{S}^3$ given by $dx^1\wedge dx^2+dx^3\wedge dx^4$ in the standard coordinates. One can write down a standard contact 1-form $\eta_\infty=\frac{1}{2}((x^1dx^2-x^2dx^1)+(x^3dx^4-x^4dx^3))$ with $\widetilde\omega_{\infty}^1=d\eta_\infty$, and $\eta_\infty$ is $\SU(2)$-invariant. So for $i$ large 
 one can also write $\widetilde\omega_{i}^1=d\eta_i$ such that $\eta_i$  converges smoothly to $\eta_\infty$. Then we can average out $\eta_i$ by the group $G_i$ to make $\eta_i$ invariant under $G_i$ and by \eqref{e:c1 close} we can assume $\eta_i$ still converge to $\eta_\infty$ in $C^0$. Notice $d\eta_i=\widetilde\omega_i^1$ also converges to $\widetilde\omega_\infty^1$ in $C^0$. In particular, for $i$ large $\eta_i$ is a $G_i$-invariant contact 1-form. Moreover, there is an obvious isotopy of $G_i$-invariant contact 1-forms $\eta_t=t\eta_\infty+(1-t)\eta_i$ for $t\in[0, 1]$.  Applying Gray's Stability Theorem (see \cite[theorem 2.20]{Geiges}  or \cite[page 135-136]{MD-S} for more details), we conclude that $\eta_i$ and $\eta_\infty$ define isomorphic contact structures on $\widetilde\Sigma_{\infty}/\overline{G}_i$. Now we can apply a result in contact geometry \cite{Symplectic-filling}, which states that the minimal symplectic filling of the space $\widetilde\Sigma/\overline{G}_i=\mathbb{S}^3/\overline{G}_i$ has a unique diffeomorphism type. For simplicity, we will not distinguish the notations $G_i$ and $\overline{G}_i$. Notice that in our setting the subset $W_i$ enclosed by $\Sigma_i$ inside $M_i^4$ provides a minimal symplectic filling, but on the other hand, the minimal resolution $\widetilde{\dC^2/G_i}$ provides another, so in particular $\chi(W_i)=\chi(\widetilde{\dC^2/G_i})$. Now in our setting $|G_i| \to +\infty$, so for $i$ large $G_i$ is either a finite cyclic subgroup $\dZ_{k_i+1}\leq\SU(2)$ or a binary dihedral group $2D_{2(k_i-2)}$. So it follows that 
 \begin{align}
 	\chi(\widetilde{\dC^2/G_i}) =
 	\begin{cases}
 		k_i + 1, & \text{when}\ G_i = \dZ_{k_i+1},
 		\\
 		k_i + 1, & \text{when}\ G_i = 2D_{2(k_i-2)},
 	\end{cases}
 \end{align}
 and hence $\chi(W_i)=\chi(\widetilde{\dC^2/G_i})\rightarrow\infty$. 
 
 We recall the Chern-Gau\ss-Bonnet theorem on an Einstein 4-manifold $(M^4, g)$ with boundary:
 \begin{equation}
 \label{e:GBC1}
8\pi^2\chi(M^4) = \int_{M^4} |\Rm_g|^2\dvol_g + 8\pi^2 \int_{\p M^4} TP_{\chi}.	
\end{equation}
Denote by $\II$ and $H$ the second fundamental form and the mean curvature of $\p M^4$, respectively. Then the above transgression of the Pfaffian is given by
\begin{equation}\label{e:GBC2}
TP_{\chi} = \frac{1}{4\pi^2} \cdot \Big(\lambda\cdot H - \Rm_{ikkj}\cdot \II_{ij} + \frac{1}{3}H^3 + \frac{2}{3}\Tr(\II^3)- H\cdot |\II|^2\Big) \dvol_{\p M^4},	
\end{equation}
where $i$, $j$, $k$ are in the tangential direction of $\p M^4$, and $\lambda$ is the Einstein constant.
Applying this to $W_i$, 
 since the second fundamental form of $\Sigma_i$ is uniformly bounded and the volume is collapsing, the boundary integral goes to $0$. So  by \eqref{e:curvature local  l2 bound} we obtain a uniform bound on $\chi(W_i)$. Contradiction.  	
\end{proof}

\begin{remark}

In the above proof we make use of the symplectic structure more than the Ricci-flat structure. It is possible to use signature formula on manifolds with boundary to give a proof, but we are not aware of the formula of the eta invariant on general collapsing manifolds $(M_j^{2k+1}, g_j)$. For a fixed manifold $M^3$ with collapsing metrics, the convergence of eta invariants 
is studied in \cite{limiting-eta}.
\end{remark}

\subsection{Proof of  Theorem \ref{t:global} }\label{ss:4-4}
By the Chern-Gau\ss-Bonnet theorem, $\int_{\mathcal K}|\Rm_{g_j}|^2\dvol_{g_j}=192\pi^2$. Then it follows from Corollary \ref{c:finite singular points} that the singular set $\mathcal S$ consists of finite points. The rest of the proof does not require the $L^2$ bound on curvature. We will only use item (1) of Theorem \ref{t:3d local version}.

\begin{proposition}\label{p:flat-orbifold-limit}
 $(X_{\infty}^3, d_{\infty})$ is isometric to a flat orbifold locally modeled on $\dR^3/\dZ_2$, and the limit measure $\nu_\infty$ is proportional to the Hausdorff measure. 
\end{proposition}

\begin{proof} 
We consider the positive harmonic function $V$ on $X_\infty^3\setminus \mathcal S$. Near each $p_\alpha\in \mathcal S$,  $V=c_\alpha r_\alpha^{-1}+h_\alpha$, where $r_\alpha(x)\equiv d_{\infty}^{\flat}(x,p_{\alpha})$ and $c_{\alpha}\in[0,\infty)$. Let $\mathcal S'$ be the subset of $\mathcal S$ consisting of those $p_\alpha$'s with $c_\alpha>0$. Then $V$ is orbifold smooth on $X_\infty^3\setminus \mathcal S'$. On the other hand, $V\rightarrow\infty$ near $\mathcal S'$, so the minimum of $V$ is achieved at some point in $X_\infty\setminus \mathcal S'$. By the strong maximum principle on the flat space, $V$ is a constant. \end{proof}

Let us review some standard facts regarding flat orbifolds.
Bieberbach's theorem (see theorem 3.2.1 in \cite{Wolf}) states that if $\Gamma\leq \Isom(\dR^n)$ is a discrete and co-compact, then the lattice $\Lambda \equiv \Gamma\cap\dR^n$  is  normal  in $\Gamma$ with  bounded index $[\Gamma:\Lambda]\leq w(n)$ and yields the following exact sequence
\begin{align}
1 \to \Lambda \to \Gamma \to H \to 1,
\end{align}
where $H=\Gamma/\Lambda\leq O(n)$.
Now we consider a {\it closed} flat orbifold $X^n$. Applying Thurston's developability theorem (see Chapter 13 of \cite{Thurston} or Chapter III.$\mathcal{G}$ of \cite{Bridson-Haefliger}), the universal covering orbifold of $X^n$ is isometric to $\dR^n$ so that
$X^n = \dR^n/\Gamma$ for some discrete co-compact group $\Gamma\in \Isom(\dR^n)$. Then Biberbach's theorem implies that $X^n=\mathbb{T}^n/H$ for some finite group $H\leq O(n)$, where $\mathbb T^n\equiv \dR^n/\Lambda$ is a flat torus.

In our setting, 
we can write $X^3_\infty=\dT^3/H$ for some finite group $H\leq O(3)$. Let $q\in X_\infty^3$ be an orbifold point.  Then the tangent cone at $q$ is isometric to $\dR^3/\dZ_2$, where the group $\dZ_2$ is generated by the reflection $\iota:\dR^3\to \dR^3$, $x\mapsto -x$. Moreover, $\iota$  induces an element in $H$ with the fixed point $q$ and $\det(\iota)=-1$.
In particular, $H\not\subset SO(3)$. 
Let us denote $H_0 \equiv H\cap SO(3)$. Then $H = H_0 \cup (\iota\cdot H_0)$.

Next, we claim that any element $\gamma \in H_0$ acts freely on $\dT^3$. If not, suppose $\gamma$ has some fixed point $x_0\in \dT^3$ and recall $\gamma\in H_0\leq SO(3)$. Then $\gamma$ fixes the rotation axis passing through $x_0$. However, this contradicts the assumption that $X_\infty^3$ has only isolated singularities. 

 Now since $\pi_1(\mathcal K)=\{1\}$, by \cite{SW} we know $\pi_1(X_\infty^3)=\{1\}$. Then we must have $H_0=\{1\}$. This implies that $X_\infty^3=\mathbb T^3/\dZ_2$.

\subsection{Proof of Theorem \ref{t:complete-space}}
Let $(X_\infty^3, d_{\infty}, p_\infty)$ be given as in Theorem \ref{t:complete-space}. By  Theorem \ref{t:3d local version} (1),  the flat background geometry has orbifold singularity near each point in $\mathcal S$. Fix a normalization of the harmonic function $V$, and let $g_\infty^\flat$ be the associated flat background metric.   Near a point in $\mathcal S$ we have $V=\sigma r^{-1}+h$ for $\sigma\in[0,\infty)$ and $h$  orbifold smooth. Let $\mathcal S'$ be the subset of $\mathcal S$ consisting of points where $\sigma>0$. 

\begin{lemma} \label{l:complete infimum V}
We have 	$\limsup\limits_{r\rightarrow\infty}\inf\limits_{S_r(p_\infty)}V<\infty$. 
\end{lemma}
\begin{proof}
Otherwise, we can find an increasing sequence $r_j\rightarrow\infty$, such that $\inf\limits_{S_{r_j}(p_\infty)}V\rightarrow\infty$. By the  maximum principle for harmonic functions we actually have $V\rightarrow\infty$ uniformly at infinity.  The minimum of $V$ is then achieved at some point in $\mathcal S\setminus \mathcal S'$. Then by the strong maximum principle for harmonic functions we conclude $V$ must be constant. Contradiction.
\end{proof}

\begin{proposition} \label{p: flat background geometry for complete}
The flat background geometry $(X_\infty^3, d_\infty^\flat)$ is a complete flat orbifold.
\end{proposition}
\begin{proof}
If $(X_\infty^3,g_{\infty})$ has two ends, then $X_{\infty}^3$ isometrically splits off an $\dR$. Then it follows that $X_\infty^3$ is smooth and $\kappa=0$. Then $V>0$ is harmonic on the complete smooth metric measure space $(X_\infty^3, g_\infty, \nu_\infty)$ where $d\nu_{\infty}=V^{-1/2}\dvol_{g_\infty}$. By Theorem \ref{l:Yau gradient estimate}, $V$ is constant.

If $(X_\infty^3,g_{\infty})$ has only one end, then first we claim that $\dR$ is not an asymptotic cone at infinity. Otherwise  one can find $r_i\rightarrow\infty$, such that $A_{r_i, 2r_i}(p_{\infty})$ consists of two connected components. Similar to the proof of Proposition \ref{l:3d-tangent-cones} we can find in each connected component foliations by totally geodesic flat tori with respect to $g_\infty^\flat$. Then we can slide these tori between $A_{r_i, 2r_i}(p_{\infty})$ and $A_{r_{i+1}, 2r_{i+1}}(p_{\infty})$, and we see that $X_\infty^3$ have two ends each diffeomorphic to $\dT^2\times \dR_+$. Contradiction.   

Now the following lemma can be proved similarly to Lemma \ref{l:connectedness}. 
\begin{lemma}\label{l:connectedness at infinity}
	There exists  $\delta_0>0$ such that for every sufficiently large  $r$,  any two points in $A_{r/2, r}(p_\infty)$  can be connected by a smooth curve $\gamma\subset A_{\delta_0\cdot r, 3r}(p_\infty)$ with  arc length $|\gamma|\leq 10 r$. 
\end{lemma}
As in Section \ref{ss:4-2}, using Lemma \ref{l:complete infimum V}, Lemma \ref{l:connectedness at infinity}, and the Harnack inequality for harmonic functions, we obtain that $\limsup\limits_{r\rightarrow\infty}\sup\limits_{S_r(p_\infty)}V<\infty$, which implies that  $g_\infty^\flat$ is complete at infinity. 
\end{proof}

By Thurston's developability theorem, $(X_{\infty}^3, g_{\infty}^{\flat})$ is isometric to $\dR^3/\Gamma$ for some $\Gamma\leq \Isom(\dR^3)$. If $\Gamma$ is a free action, then the special affine  structure on $(X_\infty^3,g_{\infty}^{\flat})$ implies that $X_{\infty}^3$ is isometric to $\dR^3$, $S^1\times \dR^2$, or $\mathbb T^2\times \dR^1$. If $\Gamma$ is not free, then there is some $\sigma\in \Gamma$ which acts as reflection at one point in $\dR^3$. Let $\Gamma_0$ be the subgroup of $\Gamma$ that preserves the orientation of $\dR^3$. Then $[\Gamma:\Gamma_0]=2$, and $\Gamma_0$ acts freely on $\dR^3$. 
Otherwise, $\Gamma_0$ has an element that fixes an axis in $\dR^3$ and the singularity of $X_{\infty}^3\equiv\dR^3/\Gamma$ cannot be isolated.
It follows that $X_{\infty}^3\equiv(\dR^3/\Gamma_0)/\dZ_2$, where $\dR^3/\Gamma_0$ is isometric to $\dR^3$, $\dR^2\times S^1$, or $\dR\times \dT^2$. 
 
 Next, we classify $V$ in the above cases. If $(X_{\infty}^3, g_{\infty}^{\flat})$ is isometric to $\dR^3$, then B\^ocher's theorem implies $V=\sigma\cdot r^{-1}+\ell$ for some constants $\sigma\geq 0$ and $\ell\geq 0$. If $(X_{\infty}^3, g_{\infty}^{\flat})\equiv\dR^2\times S^1$, then we consider the $S^1$-average $\overline{V}=\int_{S^1}Vd\theta$ which is harmonic on $\dR^2\setminus\{0^2\}$. Since the composition $\overline{V}(e^z)>0$ is harmonic on $\dC_z$, $\overline{V}$ is constant. Therefore, the harmonic function $V>0$ is smooth on $\dR^2\times S^1$ which implies that $V$ is constant. 
 If $(X_{\infty}^3, g_{\infty}^{\flat})\equiv\dR\times \dT^2$, the same average argument implies $V$ is a positive constant.
  In the other cases, one can analyze the lifting of $V$ on the $\dZ_2$-cover and the same conclusion follows. 
 
 Finally, if \eqref{e:complete L2 bound} holds, then using item (2) of Theorem \ref{t:3d local version}, $V$ is a positive constant.

\section{Singularity structure II: Case $d=2$}
\label{s-5}

\subsection{Main results}

We first state  the main results of this section.

\begin{theorem}[Local version] \label{t:2d local version}
Let $(X_j^4, g_j, p_j)$ be a sequence of  hyperk\"ahler manifolds such that $\overline{B_2(p_j)}$ is compact,  and 
$(X_j^4, g_j,\nu_j, p_j) \xrightarrow{mGH} (X_\infty^2, d_{\infty}, \nu_\infty, p_{\infty})$
with $\dim_{\ess}(X_\infty^2)=2$. If the singular set $\mathcal{S} = \{p_\infty\}$, then the limit metric on $X_\infty$ is a singular special K\"ahler metric in the sense of Definition \ref{def:singular special kahler metric}, and $\nu_{\infty}$ is a multiple of the 2 dimensional Hausdorff measure  on $X_{\infty}^2$.
\end{theorem}

\begin{theorem}[Compact version]\label{t:global-2d}
Let $g_j$ be a sequence of hyperk\"ahler metrics on the K3 manifold $\cK$ with $\diam_{g_j}(\mathcal{K})=1$ 
such that
$(\cK, g_j,\nu_j) \xrightarrow{mGH} (X_{\infty}^2, d_{\infty}, \nu_{\infty})$ with $\dim_{\ess}(X_{\infty}^2)=2$.
  Then $X_{\infty}^2$ is homeomorphic to $S^2$, endowed with  a singular special K\"ahler metric.
\end{theorem}
 
\begin{remark}
By definition a singular special K\"ahler metric has local integral monodromy around each singular point. With a uniform bound on the $L^2$ curvature, which is automatic in the setting of Theorem \ref{t:global-2d},  one expects the limit should indeed have integral monodromy. See Conjecture \ref{conj:integral monodromy}.
\end{remark}
\begin{remark}
We were informed by Shouhei Honda that using the theory of RCD spaces,  one can show that the above limit space $(X_\infty^2,d_{\infty})$ is indeed an Alexandrov space of non-negative curvature. 
\end{remark}

\begin{theorem}[Complete version]\label{t:complete-space-2d}
Let $(X_j^4, g_j, p_j)$ be a sequence of hyperk\"ahler  manifolds such that 
$(X_j^4, g_j, \nu_j, p_j) \xrightarrow{mGH} (X_{\infty}^2, d_{\infty}, \nu_{\infty}, p_{\infty})$. Assume $(X_{\infty}^2, d_{\infty})$ is complete non-compact
and  $\dim_{\ess}(X_\infty^2)=2$. If $\mathcal{S}=\{p_\infty\}$, then  $X_\infty^2$ is isometric to either a flat metric cone \emph{$\textbf{C}_\beta$} for $\beta\in \{\frac 12, \frac 13, \frac 23,\frac 14, \frac34, \frac 16, \frac 56, 1\}$  or the flat product $\dR\times S^1$, with the standard special K\"ahler structure.
\end{theorem}

\subsection{Proof of the main results}\label{ss:5-2}
The main part of this subsection is devoted to the proof of Theorem \ref{t:2d local version}.  At the end of this section we prove Theorem \ref{t:global-2d} and  Theorem \ref{t:complete-space-2d}.

 Now  assume we are in the setup of Theorem \ref{t:2d local version}. By Section \ref{ss:3-2}, we know $B_1(p_\infty)\setminus\{p_\infty\}$ is a special K\"ahler manifold.  
Recall that if $\Sigma$ is a smooth surface with boundary and with  Gaussian curvature $K\geq0$, then by the Gau\ss-Bonnet theorem we have 
\begin{equation}\label{e:2d Gauss-Bonnet}2\pi(2-2g(\Sigma)-n)=2\pi \chi(\Sigma)=\int_{\Sigma}K+\int_{\p \Sigma} k\geq \int_{\p \Sigma}k,
\end{equation}
where $k$ denotes the boundary geodesic curvature, and $n$ is the number of boundary circles. Given a tangent cone $(Y, \bar p)$ at $p_\infty$, suppose it is given by the limit of $(X_\infty, r_i^{-1}d_\infty, p_\infty)$ for  some $r_i\rightarrow 0$. 
If $\dim Y=2$, then by the interior curvature bound discussed in Section \ref{ss:3-2} we know $Y\setminus\{\bar p\}$ is smooth and special K\"ahler. 
If $\dim Y=1$, then $Y$ is either $\dR$ or $\dR_+$, and the collapsing is locally along a smooth circle fibration. We first claim $Y$ can not be $\dR$. Otherwise 
 we can choose a sequence $r_i\rightarrow 0$ such that the annulus $r_i^{-1}A_{r_i, 2r_i}$ collapses to the union of intervals $[-2,-1]\cup [1, 2]$ with bounded curvature. Then we can choose a smooth fiber $C_i$ with uniformly bounded geodesic curvature, which is given by the union of two circles. In particular $\int_{C_i}k\rightarrow 0$. 
 Let $\Sigma_i$ be the region bounded by $C_i$ and $C_{i+1}$, whose boundary consists of 4 disjoint circles. Applying \eqref{e:2d Gauss-Bonnet} we easily reach a contradiction. 	
 
So  we know any tangent cone $Y$ is either 2 dimensional or is isometric to $\dR_+$. In both cases we can choose a smooth circle $C_i$ in the annulus $r_i^{-1}A_{r_i, 2r_i}$ with $\lim_{i\rightarrow\infty}\int_{C_i}k=c$. In particular, $c = 0$ when $Y = \dR_+$.
Again applying \eqref{e:2d Gauss-Bonnet} to the region $\Sigma_i$ bounded by  $C_i$ and $C_{i+1}$, we see that for $i$ large $\Sigma_i$ is diffeomorphic to a cylinder. In particular we have shown that

\begin{lemma} \label{l:diffeo type}  For $\delta>0$ small, $B_\delta(p_\infty)\setminus\{p_\infty\}$ is diffeomorphic to a  punctured disc in $\dR^2$.	
\end{lemma}
 Without loss of generality  we may assume $\delta=1$, and denote $B=B_1(p_\infty)$, $B^*=B_1(p_\infty)\setminus\{p_\infty\}$.

\

We now prove Theorem \ref{t:2d local version}.  Choose a loop $\sigma$ generating $\pi_1(B^*)$,  oriented so that it goes counterclockwise around $p_\infty$ (notice $B^*$ being a Riemann surface is naturally oriented).  Denote by $A$ the monodromy of the special K\"ahler structure along $\sigma$. Denote by $\omega$ the K\"ahler form on $B^*$.  We now divide into 3 cases

\

\textbf{Case 1}: $A$ is conjugate to $\Id, I_1$, or $I_1^{-1}$. In this case $A$  has an invariant vector. So we can choose a local holomorphic coordinate $z$ so that $dz$ is a globally defined holomorphic 1-form on $B^*$. Then we have
$$\omega=\frac{\sqrt{-1}}{2} \text{Im}(\tau)dz\wedge d\bar z$$
where $\Ima(\tau)$ is positive harmonic function on $B^*$.

\begin{lemma}\label{l:5.7}
For $r>0$ small, we have $\Ima(\tau)\geq Cr^{{3/2}}$ on $S_r(p_\infty)$. 
\end{lemma}
\begin{proof}
The proof is similar to the arguments in Section \ref{ss:3-2}. If the estimate does not hold then $\text{Im}(\tau)$ is a global harmonic  function on $B$ in the sense of Definition \ref{def:harmonic functions}. This leads to  a contradiction by the weak Harnack inequality Theorem \ref{t:weak-harnack}.
\end{proof}

Notice from the proof of Lemma \ref{l:diffeo type} one can see that for any $r$ small, we can find a loop $\sigma_r$ contained in the annulus $A_{r/2, 2r}$ which is homotoptic to $\sigma$ and with length bounded by $Cr$.  Since $|dz|=\frac{1}{\sqrt{\text{Im}(\tau)}}$, by letting $r\rightarrow 0$ it follows that $\int_{\sigma}dz=0$. So $z$ is single-valued on $B^*$, and it extends continuously across $p_\infty$. Adding a constant  we may assume $z(p_\infty)=0$. Now $z$ defines a covering map from $B^*$ onto a domain in $\dC$.
 Suppose the covering degree is $k$, then we can take $\zeta=z^{1/k}$ as a holomorphic coordinate on $B^*$ and this embeds $B^*$ holomorphically onto a punctured domain $\Omega^*=\Omega\setminus\{0\}$ in $\dC$. 

Now we may view $\text{Im}(\tau)$ as a positive harmonic function on $\Omega^*$, so by B\^ocher's theorem  we know 
$$\text{Im}(\tau)=-c\log |\zeta|+V(\zeta)$$
where $c\geq 0$ and $V$ extends smoothly across $0$.
Then one can directly check that the tangent cone at $p_\infty$ is given by the flat metric $\omega_0=\frac{\sqrt{-1}}{2}k^2|\zeta|^{2k-2}d\zeta\wedge d\bar\zeta$ on $\dC$. This is a cone of angle $2\pi k$. Since $B$ is a Ricci limit space we must have $k=1$. If $c=0$ then the metric is smooth across $p_\infty$. If $c>0$, then
by rescaling the special holomorphic coordinate $z$ we may assume $c=1$. Then the metric is a singular special K\"ahler metric of Type I.

\

\textbf{Case 2.} $A$ is conjugate to $-\Id, I_1^*$ or $(I_1^*)^{-1}$.
In this case $A$ has an eigenvector with eigenvalue $-1$. Then we can choose a local special holomorphic coordinate $z$ such that $dz$ transforms to $-dz$ under $A$. 

Similar reasoning as Case 1 shows that $\text{Im}(\tau)$ is a well-defined positive harmonic function on $B^*$, and $z^2$ is a globally defined holomorphic function on $B^*$ and we may assume $z^2(p_\infty)=0$. Then as above $\zeta=z^{2/k}$ is a holomorphic coordinate and defines a holomorphic embedding of $B^*$ into a punctured domain in $\dC$. 
As before, we get 
$$\text{Im}(\tau)=-c\log |\zeta|+V(\zeta)$$
for a  harmonic function $V$ smooth at $0$.
As above one can see $k$ must equal $1$, and the tangent cone at $p_\infty$ is $\dC/\dZ_2$. If $c=0$, then the singularity is of orbifold type so the metric is a singular special K\"ahler metric of Type III with $\beta=\frac{1}{2}$. If  $c\neq 0$, then rescaling the coordinate $z$ we can make $c=1$. This shows that the metric is a singular special K\"ahler metric of Type II.

\

\textbf{Case 3.} $A$ is elliptic or hyperbolic.
Let $(Y, \bar p)$ be a tangent cone at $p_\infty$.  It follows from \eqref{e:jumping monodromy} that  $Y$ must be 2 dimensional, so $Y\setminus\{\bar p\}$ is smooth and special K\"ahler, and by \eqref{e: convergence of trace} we know $Y\setminus\{ \bar p\}$ has elliptic or hyperbolic monodromy. 
 Again the interior curvature bound implies $Y$ has   quadratic curvature decay at infinity. Notice $Y\setminus \{\bar p\}$ has non-negative curvature, so one can see that $Y$ is asymptotic to  a flat metric cone $(C_\gamma, O)$ with angle $2\pi \gamma$ for some $\gamma\in [0, 1)$, in the $C^\infty$ Cheeger-Gromov topology.

Notice $C_\gamma$ itself is also a tangent cone at $p_\infty$. Since $C_\gamma$ is flat, by Remark \ref{rmk:flat special kahler}  we know the flat connection $\nabla$ coincides with the Levi-Civita connection, hence its monodromy around an oriented loop going counterclockwise around the singularity is given by $R_\gamma$. Now again by \eqref{e: convergence of trace} we get $\Tr(A)=\Tr(R_\gamma)$.  Notice the orientation of the rotation is a conjugation invariant in $SL(2;\dR)$. In particular we must have $A=R_\gamma$ since we have chosen $\sigma$ to be oriented counter-clockwise around $p_\infty$. 

The same argument shows that the monodromy of $Y\setminus \{\bar p\}$ is also given by $R_\beta$. Then applying this again to the singular point $\bar p$ of $Y$, it follows that there is a tangent cone at $\bar p$ which is isometric to $\textbf{C}_\beta$.  This means that we can find a sequence of annuli  $A_{r_j, s_j}(\bar p)$ in $Y$ with $r_j\rightarrow0$ and $s_j\rightarrow\infty$, whose boundary circles after rescaling both converge to the unit circle in $\textbf{C}_\beta$. Applying the Gau\ss-Bonnet theorem on such sequences of annuli it is easy to see that $Y$ must be a flat metric cone hence is isometric to $\textbf{C}_\beta$. 

The above discussion in particular shows that there is a unique tangent cone at $p_\infty$, which is given by $\textbf{C}_\beta$ for some $\beta\in (0, 1)$, and the original monodromy matrix satisfies $A=R_\beta$. By Corollary \ref{c:integral monodromy} and Lemma \ref{l:integral monodromy classification} we must have  $\beta\in \{\frac 13, \frac 23, \frac 14, \frac34, \frac 16, \frac 56\}$. In particular, $A$ must be elliptic.

Now we study the singular behavior of the limit metric near $p_\infty$.  First we can find local special holomorphic coordinates $(z, w)$ on $B^*$ so that under the monodromy along $\sigma$, we have
$$A\cdot (dz-\sqrt{-1}dw)=e^{2\pi\sqrt{-1}\beta} (dz-\sqrt{-1}dw).$$
This means that $d\zeta{\equiv}d((z-\sqrt{-1}w)^{1/\beta})$ is a well-defined global holomorphic $1$-form on $B^*$. Since $R_\beta\neq \text{Id}$, by adding some constants to $(z, w)$ we may assume the translation part in \eqref{e:affine monodromy} vanishes. This implies that $\zeta\equiv (z-\sqrt{-1}w)^{1/\beta}$ is indeed a globally defined function on $B^*$. In general $\tau$ may not be single-valued on $B^*$. But notice 
$$\omega_\infty=\frac{\sqrt{-1}}{2}\text{Im}(\tau)dz\wedge d\bar z=\frac{\sqrt{-1}}{2}\frac{\text{Im}(\tau)}{|1-\sqrt{-1}\tau|^{2}}\beta^2|\zeta|^{2\beta-2}d\zeta\wedge d\bar \zeta.$$
So $\frac{\text{Im}(\tau)}{|1-\sqrt{-1}\tau|^2}$ is single-valued on $B^*$. 
\begin{lemma}
For all $\epsilon>0$, there exists a $C(\epsilon)>0$ such that for all $r\in (0, 1/2]$, on $S_r(p_\infty)$ we have
$$ \frac{\Ima(\tau)}{|1-\sqrt{-1}\tau|^2}\geq C(\epsilon)r^{\epsilon}.$$
\end{lemma}
\begin{proof}
Suppose this fails then we can find a sequence $r_i\rightarrow 0$, such that \begin{equation}\label{e:a contradicting equation}\sup_{A_{r_i, 2r_i}(p_\infty)}\frac{\text{Im}(\tau)}{|1-\sqrt{-1}\tau|^2}\geq (1+\delta) \inf_{A_{r_i, 2r_i}(p_\infty)}\frac{\text{Im}(\tau)}{|1-\sqrt{-1}\tau|^2}\end{equation} for some $\delta>0$. Let $\widetilde U_i$ be the universal cover of $r_i^{-1}\cdot A_{r_i, 2r_i}(p_\infty)$, endowed with the rescaled metric. Then as $i\rightarrow\infty$ we know $\widetilde U_i$ converges to the universal cover $\widetilde{U}_\infty$ of $A_{1,2}(0)\subset \textbf{C}_\beta$, which is flat. We can find $\lambda_i>0$ such that $\sup_{\widetilde U_i}\lambda_i\cdot \text{Im}(\widetilde\tau)=1$. Suppose this supremum is achieved at some $q_i\in \overline{\widetilde U_i}$. Denote $D_i=\text{Re}(\lambda_i\cdot\widetilde\tau(q_i))$.  Then by Harnack inequality  it follows that $\text{Im}(\lambda_i\widetilde\tau)$ and $\text{Re}(\lambda_i\widetilde\tau-D_i)$ are locally uniformly bounded. So passing to a subsequence we obtain local convergence of $\lambda_i\widetilde\tau-D_i$ to a limit $\widetilde\tau_\infty$ on $\widetilde U_\infty$. The flatness of $\widetilde U_\infty$ implies $\widetilde\tau_\infty$ is a constant. This then contradicts \eqref{e:a contradicting equation}.
\end{proof}

The Lemma implies that $|\nabla|\zeta|^\beta|^2\leq C(\epsilon)^{-1}r^{-\epsilon}$. In particular, $\zeta$ extends continuously across $p_\infty$.  
 Moreover, $\zeta$ realizes $B^*$ as a finite cover of some punctured domain $D^*$ in $\dC$. So for some $k>0$, $\zeta^{1/k}$ defines a global holomorphic coordinate on $B^*$ which embeds $B^*$ into $\dC$. 
 Now we identify the upper half space with the unit disk $\mathbb D$ via the map 
\begin{equation}\label{e:xi tau transform}\tau\mapsto \xi=\frac{\tau-\sqrt{-1}}{\tau+\sqrt{-1}}\end{equation} Then the monodromy transformations on $\xi$ is given by  $\xi\mapsto e^{-4\pi\sqrt{-1}\beta}\xi$.
So $-\log |\xi|$ is well-defined on $B^*$. 
 By B\^ocher's theorem we  have 
$$-\log |\xi|=-c\log |\zeta|+v$$
for $c\geq 0$ and $v$ a smooth harmonic function on $B$.

We claim $c$ can not be zero. Otherwise $|\log|\xi||\leq C$ on $B^*$, which implies that $\tau$ is uniformly bounded on $B^*$ and has definite distance away from $\sqrt{-1}$. Then by taking limit as in the proof of the above Lemma we see $\widetilde\tau_\infty$ is not fixed by $R_\beta$ so is not invariant under the monodromy, which is a contradiction. 
 So we know $c>0$, and  as $p$ moves to $p_\infty$, $|\xi|\rightarrow0$.
 
  We may write 
$$\omega_\infty=\frac{\sqrt{-1}}{8} (1-|\xi|^2)\beta^2|\zeta|^{2\beta-2}d\zeta\wedge d\bar \zeta.$$
Since the tangent cone at $p_\infty$ is $\textbf{C}_\beta$ we see $k$ must be $1$. Without loss of generality we may assume $\zeta(p_\infty)=0$. We now divide into two cases

\begin{itemize}
	\item $\beta\in \{1/4, 3/4\}$. Then $\xi^2$ is holomorphic across $0$, and $\xi=F(\zeta)^{1/2}$ for a holomorphic function $F$ with $F(0)=0$. 
\item $\beta\in\{1/3,  2/3, 1/6,  5/6\}$. Then $\xi^3$ is holomorphic across $0$, and $\xi=F(\zeta)^{1/3}$ for a holomorphic function $F$ with $F(0)=0$.\end{itemize}
 These imply $\omega_\infty$ is a singular special K\"ahler metric of Type III. 
This finishes the proof of Theorem \ref{t:2d local version}.

\begin{proof}[Proof of Theorem  \ref{t:global-2d}]
By Theorem \ref{t:2d local version},  $d_\infty$ is a singular special K\"ahler metric on a compact Riemann surface. Since $\pi_1(\mathcal K)=\{1\}$, by \cite{SW}, $X_\infty$ is simply connected which implies that $X_{\infty}$ must be homeomorphic to $S^2$. 
\end{proof}

\begin{proof}[Proof of Theorem \ref{t:complete-space-2d}]  
By Theorem \ref{t:2d local version} we know $X_\infty$ is endowed with a singular special K\"ahler metric $\omega$. In particular, the curvature of $X_\infty\setminus\{p_\infty\}$ is positive. Then it is easy to see  that each end of $X_\infty$ is asymptotic to a unique cone at infinity. If $X_\infty$ has two ends, then it splits isometrically as a flat product $\dR\times S^1$. So we assume $X_\infty$ has only one end. Then an easy application of the Gau\ss-Bonnet theorem implies that $X_\infty$ is homeomorphic to $\dR^2$. Let $\sigma$ be a loop generating the fundamental group at infinity, and denote by $A$ the monodromy matrix along $\sigma$. 

\textbf{Case (a)}. There is a 1 dimensional asymptotic cone $Y$. Using the Gau\ss-Bonnet theorem as in the beginning of this subsection one can see that $Y$ must be $\dR_+$.
 Then by \eqref{e:jumping monodromy} we know $A$ must be conjugate to $I_1, I_1^{-1}$ or $\Id$. In particular, on $X_\infty\setminus \{p_\infty\}$ there is a local special holomorphic coordinate $z$ such that $dz$ is globally defined, so is  the positive harmonic function $\text{Im}(\tau)$. Just as in the proof of Theorem \ref{t:2d local version}, \textbf{Case 1}, the function $z$ is indeed a global coordinate on $X_\infty$.  
Then by similar arguments as in the proof of Proposition \ref{p: flat background geometry for complete} one can show that the flat metric $\omega^{\flat}\equiv\frac{\sqrt{-1}}{2}dz\wedge d\bar z=\text{Im}(\tau)^{-1}\omega$ is complete at infinity. So $X_\infty$ is biholomorphic to $\dC$. Now an application of B\^ocher's theorem yields that $\text{Im}(\tau)$ must be a constant. Hence the metric $\omega$ itself is a flat metric on $\dC$.  This is a contradiction.

 \textbf{Case (b)}. All asymptotic cones are 2 dimensional. In particular, they are all flat cones, and must be the unique $\textbf{C}_\beta$ such that $A$ is conjugate to $R_\beta$. This also implies that the tangent cone at $p_\infty$ must also be $\textbf{C}_\beta$. Then  the Gau\ss-Bonnet theorem implies that $X_\infty$ itself is flat, hence must be the cone $\textbf{C}_\beta$.  \end{proof}

\section{Classification of gravitational instantons}
\label{s-7}
\subsection{Uniqueness of asymptotic cones}

Let $(X^4, g)$ be a gravitational instanton, and we fix a hyperk\"ahler triple $\bom$.  If it  is flat, then it is isometric to a flat product $\dR^k\times \dT^{4-k}$ with $1\leq k\leq 3$. By Cheeger-Gromoll's splitting theorem, $X^4$  is isometric to a flat product $\dR\times\dT^3$ unless $X^4$ has only one end. In the following, we will always assume $X^4$ is non-flat and  
has only one end.
 We will also assume $\dim_{\ess}(Y)\leq 3$ for any asymptotic cone $(Y,d_Y, p_*)$, since otherwise $(X^4,g)$ is ALE and this case has already been classified by Kronheimer \cite{Kronheimer-Torelli}.

 Since $\int_{X^4}|\Rm_g|^2\dvol_g<\infty$, Theorem \ref{t:epsilon regularity HK} and Proposition \ref{p:regular set smooth} imply that any asymptotic cone $(Y, d_Y, p_*)$ is smooth away from  $p_*\in Y$. By Theorem \ref{t:complete-space} and Theorem \ref{t:complete-space-2d}, $(Y,d_Y)$ is a flat space isometric to one of the following: $\dR^3$, $\dR^3/\dZ_2$,  $\dR^2$, $\dR$, $\dR_{+}$, $S^1\times \dR^2$, $\dT^2\times\dR$, $S^1\times \dR$, a flat cone $\textbf{C}_\beta$ for $\beta\in \{\frac 12, \frac 13, \frac 23, \frac 14, \frac34, \frac 16, \frac 56\}$.

\begin{lemma}\label{l:asymptotic cone classification}Any asymptotic cone is a flat metric cone. 
\end{lemma}

\begin{proof}
It suffices to rule out $S^1\times \dR^2, \dR\times \dT^2$ and $S^1\times \dR$. Notice these spaces have moduli given by the moduli of the flat metrics on $S^1$ and $\dT^2$.  We first  show $S^1_{1/2}\times \dR^2\notin \mathcal T_\infty(X^4)$, where $S^1_R$ denotes a circle of diameter $R$. 

Suppose one asymptotic cone $Y$ is given by $S^1_{1/2}\times \dR^2$. 
Then we may find $r_i\rightarrow\infty$ such that $r_i^{-1}B(p, r_i)$ converges to the unit ball in $S^1_{1/2}\times \dR^2$. 
On the other hand, we know for all $r>0$ sufficiently large, $r^{-1}B(p, r)$ is $\epsilon(r)$-GH close to the ball $U_r$ of radius 1 around the vertex in an asymptotic cone $Y_r$, where $\lim_{r\rightarrow\infty} \epsilon(r)=0$. Notice $Y_r$ is not unique and we simply make arbitrary choices for all $r$. Let $\epsilon_0\in (0,1/100)$ be small so that any asymptotic cone $Y_r$ whose unit ball $U_r$ is $3\epsilon_0$-GH close to the unit ball in $S^1_{R}\times \dR^2$ for some $R>1/4$ must be itself of the form $S^1_{R'}\times \dR^2$ for some $R'\geq \frac{1}{8}$. 

Now fix $\widetilde r$ large so that $\epsilon(r)<\epsilon_0/2$ for $r\geq\widetilde r$. For any $i$ with $r_i\geq \widetilde r$, let $s\in [r_i, r_{i+1}]$ be the smallest number so that for all $r\in [s, r_{i+1}]$,  $r^{-1}B(p, r)$ is $\epsilon_0$-GH close to $S^1_R\times \dR^2$ for some $R\geq 1/8$. By assumption we know $s\leq \frac{1}{2}r_{i+1}$. We claim $s=r_i$. Otherwise, if $s>r_i$, then for all $s'\in [s/2, s]$, we have $2s'\in [s, r_{i+1}]$, so 
$(2s')^{-1}B(p, 2s')$ is $\epsilon_0$-GH close to some ball $S^1_R\times \dR^2$ for some $R>1/8$. In particular, $(s')^{-1}B(p, s')$ is $2\epsilon_0$-GH close to $S^1_{2R}\times \dR^2$. By assumption it follows that the unit ball $U_{s'}$ in $Y_{s'}$ is $3\epsilon_0$-GH close to the unit ball in $S^1_{2R}\times \dR^2$. By our choices of $\epsilon_0$ we conclude that $Y_{s'}$ is of the form $S^1_{R'}\times \dR^2$ for $R'\geq 2R-3\epsilon_0>1/8$. This contradicts the choice of $s$.

So for any sufficiently large $r$, we may write $Y_r=S^1_{f(r)}\times \dR^2$ for some $f(r)\in(1/8,5/8)$. Now we claim that \begin{align}\label{e:f-decay}f(2r)<\frac{3}{4} f(r), \quad \forall r > 0, \end{align} so that the desired contradiction immediately arises. In fact, if \eqref{e:f-decay} is true, then $f(r) \to 0$ as $r \to \infty$ which contradicts $f(r) > 1/8$.  To see the claim, we notice that by assumption $r^{-1}B(p, r)$ is $\epsilon(r)$-GH close to the unit ball in $S^1_{f(r)}\times \dR^2$. So $(2r)^{-1}B(p, r)$ is $\frac{1}{2}\epsilon(r)$-GH close to the half ball in $S^1_{f(r)/2}\times \dR^2$. On the other hand, $(2r)^{-1}B(p, 2r)$ is $\epsilon(2r)$-GH close to the unit ball in $Y_{2r}$. It follows that $f(2r)<\frac{3}{4}f(r)$ if $r$ is large. 

Hence we have proved that $S^1_{1/2}\times\dR^2\notin \mathcal T_\infty(X^4)$. By rescaling, $S^1_R\times \dR^2 \notin \mathcal T_\infty(X^4)$ for all $R$. Similar arguments also show that $S^1_R\times \dR\notin \mathcal T_\infty(X^4)$ for all $R$.  

Finally, we claim that the possible unit-area flat $\dT^2$  such that $\dR\times \dT^2\in \mathcal T_\infty(X^4)$ must form a compact moduli. Indeed, if not, applying Lemma \ref{l:tangent cone at infinity connected}, one can choose a sequence of  flat tori $(\dT^2, g_j^{\FF})$ with $\Area_{g_j^{\FF}}(\dT^2) = 1$ and $\diam_{g_j^{\FF}}(\dT^2) \to \infty$ such that after appropriate scaling-up, 
\begin{align}
(\dT^2, \tilde{g}_j^{\FF}) \xrightarrow{GH} \dR\times S_1^1.	
\end{align}
It follows that rescalings of $\dR\times \dT^2$ converge to $\dR^2 \times S_1^1 \in \mathcal{T}_{\infty}(X^4)$ in the pointed Gromov-Hausdorff sense, which is a contradiction.  Then similar arguments as above also show that $\dR\times \dT^2\not\in \mathcal T_\infty(X^4)$ for any flat torus $\dT^2$.
\end{proof}

\begin{proposition}\label{p:uniqueness-of-asymptotic-cone}
Let $(X^4,g)$ be a gravitational instanton. Then it has a unique asymptotic cone which is a flat metric cone $(C(W), d_C, p_*)$, where $W$ denotes the cross-section and $p_*$ is the cone vertex. 
\end{proposition}
\begin{proof}  By Lemma \ref{l:tangent cone at infinity connected}, $\mathcal{T}_\infty(X^4)$ is connected and compact. Denote by $d$ the maximal dimension of the elements in $\mathcal T_\infty(X^4)$.  
If $d=3$, then we choose some $Y\in \mathcal T_\infty(X^4)$ with $\dim_{\ess}(Y)=3$. Then any element in a small neighborhood $\mathcal U$ of $Y$ in $T_\infty(X^4)$ has dimension $3$, so by Lemma \ref{l:asymptotic cone classification} it must be $\dR^3$ or $\dR^3/\dZ_2$. So the connectedness of $\mathcal T_\infty(X^4)$ implies that $\mathcal T_{\infty}(X^4)=\{Y\}$. Similar arguments apply to the case $d=2$.
\end{proof} 
In the rest of this section  we will denote by $(Y, d_Y, p_*)$ the unique asymptotic cone of $X^4$, and denote $d= \dim_{\infty}(X^4)\equiv \dim_{\ess}(Y)$. Since  $X$ has only one end, $W$ is connected and $Y\neq\dR$. From Section \ref{s-3}, the renormalized limit measure on $Y$ is $\nu_Y=\chi\cdot \dvol_{g_Y}$, where $\chi$ is a constant if $d>1$ or $d=1$ and $G_\infty=\dR^3$; $\chi=c\cdot z^{\frac{1}{2}}$ if $d=1$ and $G_\infty=\mathscr{H}_1$ (here $z$ is the affine coordinate). 

\subsection{Nilpotent fibration on the end}

Denote $r(x)\equiv d_g(p,x)$ and $\hat{r}(y)\equiv d_Y(p_*,y)$ for $x\in X$ and $y\in Y$. Below we use $\tau(x)$ to denote a general function on the end of $X^4$ such that $\lim\limits_{r(x)\to\infty}\tau(x)=0$. The following is essentially due to Cheeger-Fukaya-Gromov \cite{CFG}.
We  give an outline of the arguments in Appendix \ref{a1}.
\begin{theorem} \label{t:CFG global}
	There exists a smooth fibration map 
$F: X^4\setminus K\rightarrow Y\setminus K'$,	where $K, K'$ are compact such that the following properties hold.
	\begin{enumerate}
		\item There are flat connections $\nabla_y$ with parallel torsion on the fibers $F^{-1}(y)$ which depend smoothly on $y\in Y\setminus\Omega$, such that each fiber $(F^{-1}(y), \nabla_y)$ is affine diffeomorphic to a  nilmanifold $\Gamma\setminus N$ for $\Gamma\subset N_L$, and the structure group of the fibration is reduced to  $((\mathfrak{Z}(N)\cap\Gamma)\setminus \mathfrak{Z}(N))\rtimes \Aut(\Gamma)\subset \text{Aff}(\Gamma\setminus N)$;
		\item $F$ is an asymptotic Riemannian submersion in the sense that for a tangent vector $v$ at $x\in X\setminus K$ which is orthogonal to the fiber of $F$, we have 
		\begin{equation} \label{e:complete almost riemannian submersion}
			(1-\tau(x)) |v|_{g}\leq |dF_x(v)|_{g_Y}\leq (1+\tau(x))|v|_{g},
		\end{equation}
		and for all $k>0$, there exists $C_k>0$ such that for all $x\in X\setminus K$, 
		\begin{equation}\label{e:derivative uniform decay}
			|\nabla^k F(x)|_{g, g_{Y}}\leq C_k r(x)^{-k};
		\end{equation}
		\item The second fundamental form  $\Pi$ of the fibers  satisfies that for all $k\geq0$,		\begin{align}\label{e:2nd fund form estimate}
			\begin{cases}
|\nabla^k\Pi(x)|=\tau(x) r(x)^{-1-k}, & \text{if}\ d=2, 3, \quad \text{or}\ \ \ \ \ \ \  d=1 \ \text{and}\ G_\infty=\dR^3;\\ |\nabla^k\Pi(x)|\leq \frac{1}{\sqrt{3}}(1+\tau(x))r(x)^{-1-k}, & \text{if} \ d=1 \ \text{and}\ G_\infty = \mathscr{H}_1. 
\end{cases}
		\end{align}

	\end{enumerate}
\end{theorem}

 In our setting, applying Corollary \ref{c: no infranil}, all the fibers are nilmanifolds. As in Section \ref{s-6} we say a tensor $\xi$ on the end of $X$ is \emph{$\mathcal N$-invariant} if its lift to the local universal covers is invariant under the full nilpotent group action of $N_L$. 

\begin{lemma}\label{l:diameter growth}
In the setting of the above theorem, there are constants $\delta_0\in(0,1)$, $C>0$ such that \begin{align}C\cdot \hat{r}(y)^{-\delta_0}\leq \diam_g(F^{-1}(y))\leq C \cdot\hat{r}(y)^{\delta_0},\quad \forall \ y\in Y\setminus K',\end{align}    where $\diam_g$ denotes the intrinsic diameter of the fiber.	
\end{lemma}
\begin{proof}
This is a direct consequence of the 	estimates on the second fundamental form \eqref{e:2nd fund form estimate}.
\end{proof}

\begin{theorem}\label{t:invariant almost HK metrics}
By making $K$ and $K'$ larger if necessary, there exists an $\mathcal N$-invariant definite triple $\bm{\omega}^{\dag}$ defined on $X\setminus K$, such that for all $k\in\dN$, we have
\begin{equation}|\nabla^k_{{\bom}}(\bm{\omega}^{\dag}-\bm\omega)|_{g_{\bom}}=O( r(x)^{-k-1+\delta_0}).
\label{e:invariant almost HK metrics}
\end{equation}
\end{theorem}
\begin{proof}
Following the same arguments as in the proof of Proposition \ref{p:existence of local invariant definite triples}, we obtain an $\mathcal{N}$-invariant definite triple $\bm{\omega}^{\dag}$ on $X\setminus K$. The  estimate \eqref{e:invariant almost HK metrics} can be proved using the diameter growth estimate for the collapsing fibers in Lemma \ref{l:diameter growth} as well as proposition 4.9 in \cite{CFG}.  
\end{proof}
Denote by $g^{\dag}$ be the quotient metric on $Y\setminus K
'$ induced by $\bm\omega^{\dag}$.  Then \eqref{e:complete almost riemannian submersion}, \eqref{e:derivative uniform decay} and \eqref{e:invariant almost HK metrics} together imply that for all $k$, 
 \begin{equation}
 	\lim_{r\rightarrow\infty}r^{k}\sup_{S_r(p_*)}|\nabla^k_{g_Y}(g^{\dag}-g_Y)|_{g_Y}(y)=0.
 \end{equation}
 An $\mathcal N$-invariant function $f$ on $X\setminus K$ can be viewed as a function on $Y\setminus K'$, and we may write 
\begin{align}\Delta_{\bm\omega^{\dag}} f=\Delta_{g^{\dag}} f+\langle  H, \nabla_{g^{\dag}}f\rangle, \end{align}
where $ H$ denotes the mean curvature vector field of the fibers of $F$,  viewed as a vector field on $Y\setminus K'$. By Theorem \ref{t:CFG global} item (4), and the arguments in the proof of Lemma \ref{l: convergence of second fundamental form}, we have
\begin{equation} \label{e:mean curvature decay}
	\lim_{k\rightarrow\infty}r^k\sup_{S_r(p_*)}|\nabla^k_{g_Y}  (H-\nabla_{Y} \log \chi)|=0, \quad \forall \ \ k\in\dN.
\end{equation}

\subsection{Perturbation to invariant hyperk\"ahler metrics}
\label{ss:6-3}

For $R\gg1$ we denote  \begin{align}
\mathcal Q_R\equiv Y\setminus B_R(p_*), \quad \mathcal X_R\equiv F^{-1}(\mathcal Q_R).
\end{align}
As in Section \ref{ss:triple-deformation} we identify an element in $\Omega^+_{\bm\omega'}(\mathcal X_R)\otimes \dR^3$ with a $(3\times 3)$ matrix-valued function $\bm f$ on $\mathcal X_R$, and an $\mathcal N$-invariant element is identified with such a function on $\mathcal Q_R$.

\begin{theorem} \label{t:complete invariant hk triple}
Given any $\epsilon_0\in (0, 1-\delta_0)$, there exist a number $R_0>0$, and an $\mathcal N$-invariant hyperk\"ahler triple $\bm\omega^{\Diamond}$ on $\mathcal X_{R_0}$ of the form 
$\bm\omega^{\Diamond}=\bm\omega^{\dag} + dd^*(\bm f\cdot \bm\omega^{\dag})$ such that  
$|\nabla^k_{\bm\omega^{\dag}}\bm f(x)|=O(r(x)^{2-\epsilon_0-k})$ for all $k\in\dN$. 
\end{theorem}
In particular, we also have that for all $k\geq 0$, 
$|\nabla^k_{{\bom}}(\bm\omega^{\Diamond}-\bm\omega)|_{g_{\bom}}\leq C_k\cdot r(x)^{-k-\epsilon_0}$. So the original hyperk\"ahler triple $\bm\omega$ is asymptotic to the $\mathcal N$-invariant triple $\bm\omega^{\Diamond}$.

The rest of this subsection is devoted to the proof of the above theorem.
The idea is  similar to the proof of  Theorem \ref{t:local inv hk triple}.  The difference is that due to the non-compactness of $\mathcal X_R$ we need to work in certain weighted  spaces. 

Given $\delta\in \dR$ and $k\in\dN$, we define the following weighted (semi-)norms of an $\mathcal N$-invariant function $f$ on $\mathcal X_R$(or equivalently a function on $\mathcal Q_R$): 
\begin{align*}\|f\|_{C_{\delta}^k(\mathcal Q_R)} & \equiv \sum\limits_{m=0}^k\sup_{r\geq R}\Big \{r^{-\delta+m}\cdot \|\nabla_{g_Y}^m f\|_{C^0(A_{r,2r}(p_*))}\Big\},
\\
[f]_{C_{\delta}^{k, \alpha}(\mathcal Q_R)} & \equiv \sup\limits_{r\geq R}\Big\{r^{-\delta+k+\alpha}\cdot[f]_{C^{k,\alpha}_{g_Y}(A_{r,2r}(p_*))}\Big\},
\\ 
\|f\|_{C_{\delta}^{k, \alpha}(\mathcal \mathcal Q_R)}& \equiv \|f\|_{C_{\delta}^k(\mathcal Q_R)} + [f]_{C_{\delta}^{k, \alpha}(Q_R)},\end{align*}
where 
\begin{align}
[f]_{C^{k,\alpha}(A_{r,2r}(p_*))} \equiv \sup\bigg\{\frac{|\nabla^k_{g_Y}f(y_1)-\nabla^k_{g_Y}f(y_2)|}{d_{g_Y}(y_1,y_2)^\alpha}\bigg|x,y\in A_{r,2r}(p_*), d_{g_Y}(y_1,y_2)<\text{inj}_{g_Y}(y_1)\bigg\}.
\end{align}
As usual the difference in the last formula is computed in terms of the parallel transport along the minimizing geodesic.
 These (semi-)norms obviously extend to $\mathcal N$-invariant matrix valued functions.

 Now we fix $k\geq 6$ and $\alpha\in (0, 1)$. The following provides a suitable right inverse of the Laplace operator for us. Notice that we do not impose the boundary conditions since we are only interested in the asymptotic behavior at infinity.   
 
\begin{proposition}\label{p:weighted-esimate-quotient-space}
 There exists a finite set $\Gamma\subset(0, 1)$ depending only on $Y$, such that for all $\delta\in (0, 1)\setminus \Gamma$ and for all $R\geq 1$, one can find a bounded linear map $\mathcal S_R:  C^{k, \alpha}_{-\delta}(\mathcal Q_R)\to C^{k+2, \alpha}_{-\delta+2}(\mathcal Q_R)$ with the properties that $\Delta_{\nu_Y}\circ \mathcal S_R=\Id$ and $\|\mathcal S_R\|\leq C$ for $C$ depending only on  $Y$, $\delta$, $k$, $\alpha$ (but not on $R$). 
\end{proposition}

\begin{proof}
If  $d=1$, then in terms of the affine coordinates on $\dR_+$ (see Section \ref{ss:3-3}), we have $\Delta_{\nu_Y}=Cz^{-1}\p_z^2$. In this case we  reduce to a simple ODE problem that we omit the proof. 

Now we consider the case $d\geq 2$. Then $\nu_Y$ is proportional to the volume measure, and $\Delta_{\nu_Y}$ is  the metric Laplace operator on the flat cone $Y=C(W)$, where
\begin{align}
W \equiv\begin{cases}
\text{spherical space form}\ \mathbb{S}^2\ \text{or}\ \dR \mathbb P^2, & d=3,
\\
S_{2\pi\beta}^1 \ \text{with}\ \beta\in \{\frac 12, \frac 13, \frac 23, \frac 14, \frac34, \frac 16, \frac 56, 1\}, & d=2.
  \end{cases}	
\end{align} 
First for any $\delta\in (0, 1)$ one can construct a  linear extension operator $E_R:C^{k,\alpha}_{-\delta}(\mathcal Q_R)\rightarrow C^{k, \alpha}_{-\delta}(\mathcal Q_{R/2})$ with $\|E_R\|\leq C$ independent of $R\geq 1$. For this purpose one can construct $E_1$ by the local construction in \cite{Seeley} or \cite{Sze},    then use rescaling to define $E_R$. 

Denote by $(r,\Theta)$ the polar coordinates on $Y$. 
Let $\Sigma(W)\equiv \{\lambda_j\}_{j=0}^{\infty}$ be the spectrum (allowing multiplicities) of $-\Delta_{W}$ with $0=\lambda_0<\lambda_1\leq \lambda_2\leq \ldots$ . Let $\{\varphi_j\}_{j=0}^{\infty}$ be an orthonormal set of eigenfunctions satisfying $-\Delta_W	\varphi_j = \lambda_j \cdot \varphi_j$ and $\|\varphi_j\|_{L^2(W)} = 1.$
Given a function $f\in C^{k,\alpha}_{-\delta}(\mathcal Q_R)$, we denote $\widetilde f\equiv E_R(f)$. Then there is an $L^2$-expansion of $f$ given by 
$
\widetilde f(r, \Theta) = \sum\limits_{j=0}^{\infty} f_j(r) \varphi_j(\Theta).
$
For $j>0$ we have
\begin{align}\label{e:estiamte of Fourier coefficients}
	|f_j(r)|\equiv |\int_{W} \widetilde f(r, \Theta)\varphi_j(\Theta)|=|\lambda_j^{-k}\int_W((-\Delta_W)^k\widetilde f)(r, \Theta) \varphi_j(\Theta)|\leq C(\delta)\|f\|_{C^{k,\alpha}_{-\delta}(\mathcal Q_R)}\lambda_j^{-k}r^{-\delta}.
\end{align}
Let $u(r,\Theta)$ be a formal solution
$u(r,\Theta) = \sum\limits_{j=0}^{\infty} u_j(r) \varphi_j(\Theta)
$ of  $\Delta_{g_Y}u = \tilde  f$. Then   $u_j(r)$ satisfies
\begin{align}
u_j''(r) + \frac{d-1}{r}\cdot u_j'(r) - \frac{\lambda_j}{r^2}\cdot u_j(r) = f_j(r).	\label{e:inhomogeneous}
\end{align}
For every $j\in\dN$,
 the corresponding homogeneous ODE 
\begin{align}
u_j''(r) + \frac{d-1}{r}\cdot u_j'(r) - \frac{\lambda_j}{r^2}\cdot u_j(r) = 0	
\end{align}
has the following fundamental solutions: 
\begin{enumerate}
\item When $j=0$ and $d=2$,
$\mathcal{G}_0(r)\equiv \log r$
and $\mathcal{D}_0(r) \equiv 1$.
\item When $j=0$ and $d=3$, $\mathcal{G}_0(r)\equiv 1$
and $\mathcal{D}_0(r) \equiv r^{-1}$.
	\item When $j\in\dZ_+$, there are a growing solution   $\mathcal{G}_j(r)\equiv r^{\mu_j^+}$ and a decaying solution $\mathcal{D}_j(r)\equiv r^{\mu_j^-}$. Here $\mu_j^+$ and $\mu_j^-$ are the positive and negative roots of the following algebraic equation,
	\begin{equation}
		\mu^2 + (d- 2) \mu - \lambda_j = 0.
	\end{equation} 
\end{enumerate}
Now we set $\Gamma=\{|\mu_j^-||j> 0\}\cap (0, 1)$, and $\underline\delta=\min \Gamma$. 
Let $j_0$ be the largest $j$ such that $\mu_j^->-1+\underline\delta$.  For $j=0$ we can directly integrate and define
$$u_0(r)\equiv \int_R^{r} s^{1-d}(\int_{R}^s t^{d-1}f_0(t)dt)ds.$$
It is easy to see 
 $|u_0(r)|\leq C(\delta) r^{2-\delta}\|f\|_{C^{k,\alpha}_{-\delta}(\mathcal Q_R)}$.
For $1\leq j\leq j_0$, we set 
\begin{align*}
u_j (r) = \frac{\mathcal{G}_j(r)}{\mathcal{W}_j(r)}\int_r^{R}\mathcal{D}_j(s)f_j(s)ds 
+ 	\frac{\mathcal{D}_j(r)}{\mathcal{W}_j(r)}\int_R^r\mathcal{G}_j(s)f_j(s)ds,
\end{align*}
For  $j>j_0$, we set
\begin{align*}
u_j (r) = \frac{\mathcal{G}_j(r)}{\mathcal{W}_j(r)}\int_r^{\infty}\mathcal{D}_j(s)f_j(s)ds 
+ 	\frac{\mathcal{D}_j(r)}{\mathcal{W}_j(r)}\int_R^r\mathcal{G}_j(s)f_j(s)ds.
\end{align*}
Here the wronskian is given by 
\begin{align*}\mathcal{W}_j(r)\equiv \mathcal{W}(\mathcal{G}_j(r), \mathcal{D}_j(r))=(\mu_j^+ - \mu_j^-) r^{\mu_+ + \mu_j -1} = \sqrt{(d-2)^2 + 4\lambda_j}\cdot r^{1-d}.\end{align*}
It follows from \eqref{e:estiamte of Fourier coefficients} that each $u_j$ is well-defined, with
\begin{equation*}
	|u_j(r)|\leq C(\delta)\lambda_j^{-k}r^{2-\delta}\| f\|_{C^{k, \alpha}_{-\delta}}
\end{equation*}
for $\delta\in (0,1)\setminus\Gamma$.
The Weyl law implies that $\lambda_j\leq C j^{2/d}$. By standard elliptic estimates $|\varphi_j|_{C^0}\leq C\lambda_j$ for $j\geq 1$. Since $k\geq d+1$,  the formal solution  $u$ converges in $C^0$ and for all $r\geq 3R/4$, 
$$|u(r, \Theta)|\leq C(\delta)r^{2-\delta}\|f\|_{C^{k,\alpha}_{-\delta}}. $$
It is easy to check that $\Delta_{g_Y}u=\widetilde f$ holds pointwise on $\mathcal Q_{3R/4}$. Using the standard interior elliptic estimates on the rescaled annulus $r^{-1}A_{r, 2r}(p_*)$ we obtain the bound $\|u\|_{C^{k+2,\alpha}_{-\delta+2}(\mathcal Q_R)}\leq C(\delta)\|f\|_{C^{k,\alpha}_{-\delta}(\mathcal Q_R)}$. 

Now we simply set $\mathcal S_R(f)=u$. Clearly $\mathcal S_R$ is a linear operator, and the above discussion gives the uniform bound on $\|\mathcal S_R\|$. 
\end{proof}

 Now fix $\delta_1\in (\epsilon_0, 1-\delta_0)\setminus \Gamma$.  We define the Banach space $\fA$ to be the completion of the space of  $(3\times 3)$ matrix-valued functions $\bm f$ on $\mathcal Q_R$ under the $C^{k+2, \alpha}_{-\delta_1+2}(\mathcal Q_R)$ norm, and define $\fB$ to be the completion of the same space under the $C^{k, \alpha}_{-\delta_1}(\mathcal Q_R)$ norm.

By Theorem \ref{t:invariant almost HK metrics}, for $R$ large  we know the map $\mathscr F:B_1(\bm 0)\subset \fA\rightarrow\fB$ is well-defined, with $\|\mathscr F(\bo)\|_{\fB}\leq CR^{-1+\delta_0+\delta_1}$.
 Now we let $\mathscr L(\bm f)\equiv \Delta_{\nu_Y}\bm f$ and $\mathscr N(\bm f)\equiv \mathscr F(\bm f)-\mathscr L(\bm f)$. 
 Then Proposition \ref{p:weighted-esimate-quotient-space} provides a linear operator $\mathscr{P}: \fB\rightarrow\fA$ with $\mathscr{L}\circ \mathscr{P}=\Id$, and 
$\|\mathscr P \bm v\|_{\fA}\leq C\|\bm v\|_{\fB}$ for all $\bm{v}\in\fB$, where $C>0$ is a constant independent of $R\geq 1$.
 
For any $\bm f\in \fA$, we have  
$\Delta_{\bm\omega^{\dag}}\bm f=\Delta_{g^{\dag}}\bm f+\langle  H, \nabla_{g^{\dag}}\bm f\rangle$.  
Using the fact that \begin{align*}\Delta_{g^{\dag}}\bm f=\Delta_{g_Y}\bm f+(g^{\dag}-g_Y)*\nabla^2\bm f +\nabla_{g_Y}g^{\dag}*\nabla_{g_Y} \bm f,\end{align*}
and  \eqref{e:mean curvature decay}, we have
\begin{align*}\|\Delta_{\bm\omega^{\dag}}\bm f-\Delta_{\nu_Y}\bm f\|_{\fB}\leq \epsilon(R) \|\bm f\|_{\fA}\end{align*}
for some $\epsilon(R)\rightarrow0$ as $R\rightarrow\infty.$
Applying \eqref{e:pointwise non-linear estimate} and the definition of the weighted spaces, we obtain
\begin{align*}\|\mathscr N(\bm f)-\mathscr N (\bm g)\|_{\fB}\leq (CR^{-\delta_1}+\epsilon(R))\|\bm f-\bm g\|_{\fA},\quad \forall\ \ \bm{f},\bm{g}\in B_{1}(\bo)\subset \fA.\end{align*}
 So we can apply Proposition  \ref{p:implicit function theorem} to obtain $R_0>0$ such that for $R=R_0$ there is some $\bm f\in \fA$ that satisfies the estimate $\|\bm f\|_{\fA}\leq CR_0^{-1+\delta_0+\delta_1}$. 

Finally, applying standard elliptic estimates to the equation $\mathscr L(\bm f)+\mathscr N(\bm f)=0$ on the rescaled annulus $r^{-1}A_{r, 2r}(p)$ as $r\rightarrow\infty$, we obtain higher derivative estimates. This finishes the proof of  Theorem \ref{t:complete invariant hk triple}.

\subsection{Proof of Theorem \ref{t:thm1.2}}
\label{ss:7-4}
Let $\bm\omega^{\Diamond}$ be the $\mathcal N$-invariant hyperk\"ahler triple constructed in Theorem \ref{t:complete invariant hk triple}, and let $g^{\Diamond}$ be the quotient metric on $\mathcal Q$ induced by $\bm\omega^{\Diamond}$.  We denote $\mathcal X \equiv \mathcal X_{R_0}$ and $\mathcal Q\equiv \mathcal Q_{R_0}$. We will define several families of model ends of gravitational instantons, which we will label by ``AL$\mathfrak{X}$" for some letter $\mathfrak{X}\in\{E,F,G,H,G^*,H^*\}$.  We adopt the terminology that when we say a gravitational instanton $(X^4,g)$ is AL$\mathfrak{X}$ it means that we can smoothly identify the end of $X^4$ with a model end in the family AL$\mathfrak{X}$ such that $|\nabla_g^k(g-g_{\text{model}})|_{g}=O(r^{-k-\epsilon})$ for some $\epsilon>0$ and for all $k\in\dN$, where $g_{\text{model}}$ denotes the model hyperk\"ahler metric. By Theorem \ref{t:complete invariant hk triple}, $(X^4, g)$ is AL$\mathfrak{X}$ if and only if $(\mathcal X, g_{\bm\omega^{\Diamond}})$ is AL$\mathfrak{X}$. To prove Theorem \ref{t:thm1.2}, we will classify the ends of the $\mathcal N$-invariant metric $\bm\omega^{\Diamond}$. Recall that we only need to consider the case when $X^4$ has only one end and is non-flat. Moreover, we assume that $X^4$ is not ALE, namely $\dim_\infty(X^4)\leq 3$. Theorem \ref{t:thm1.2} will follow from Theorem \ref{t:ALF end classification}, Theorem \ref{t:ALG end classification} and Theorem \ref{t:ALH* end classification}.  
\subsubsection{\emph{Case}  $\dim_{\infty}(X^4)=3$} 
 
\begin{definition}[ALF models]
ALF model ends are defined as follows
\begin{enumerate}\item An ALF-$A_k$ (for $k\in \dZ$) model end is the hyperk\"ahler metric constructed by applying the Gibbons-Hawking ansatz on $\dR^3\setminus K$ to the positive harmonic function $V=\frac{k+1}{2r}+c$, where $c>0$.

\item An ALF-$D_k$ model (for $k\in \dZ$)  end is a $\dZ_2$-quotient of an ALF-$A_{2k-5}$ end, where the $\dZ_2$-action covers the standard involution on $\dR^3$.
\item An ALF model end is an ALF-$A_k$ or ALF-$D_k$ model end for some $k\in \mathbb Z$. 
\end{enumerate}	
\end{definition}

\begin{theorem}\label{t:ALF end classification}
Any gravitational instanton $(X^4, g)$ with $\dim_{\infty}(X^4)=3$ is ALF.
\end{theorem}
\begin{remark}
 ALF-$A_k$ gravitational instantons  are classified by Minerbe \cite{Minerbe3}; they are all given by multi-Taub-NUT spaces. 
 ALF-$D_k$ gravitational instantons are classified by Chen-Chen \cite{CCII}; they are all given by the twistor space construction due to Cherkis-Hitchin-Ivanov-Kapustin-Lindstr\"om-Ro\v{c}ek \cite{Atiyah-Hitchin, LR, IR, Cherkis-Kapustin, Cherkis-Hitchin}. Notice that $k=b_2(X^4)\geq 0$. Conversely, any $k\in\dN$ can be achieved. 
\end{remark}

\begin{proof}
 We first assume $Y=\dR^3$. It is a standard fact that such $\bm\omega^{\Diamond}$ is given by the \emph{Gibbons-Hawking ansatz}. Indeed this is a special case of the discussion in Section \ref{ss:3-1}. This means that the metric $(\mathcal Q, g^{\Diamond})$ is a special affine metric 3-manifold. Denote by $V^{-1}(x)$ the length squared of the fibers of $F^{-1}(x)$ for $x\in \mathcal Q$. Then by Lemma \ref{l:diameter growth} we know for all $\sigma>0$ there is $C>0$ such that $Cr^{-\sigma}\leq V\leq Cr^{\sigma}$.  As in the proof of Proposition \ref{p:complete-background} this implies the corresponding flat background geometry $(\mathcal Q, g^\flat)$ is complete at infinity, hence must be isometric to  $Y\setminus K$ for some compact $K$. Notice $V$ is harmonic with respect to $g^\flat$. We denote the expansion 
$$V=\sum_{j\geq 0} (a_j^+ r^{\mu_j^+}+a_j^-r^{\mu_j^-})\varphi_j$$
where $\varphi_j$ is an $L^2$ orthonormal basis of Laplace eigenfunctions on the cross section of $Y$, with $-\Delta_{S^2} \varphi_j=\lambda_j\varphi_j$, $\lambda_j\geq 0$ and $\mu_j^{\pm}$ are the solutions to the equation
	$\mu^2+\mu-\lambda_j=0$.
	Notice $\lambda_0=0$ and $\lambda_j\geq 2$ for $j>0$. So we have $\mu_0^+=0, \mu_0^-=-1$, $\mu_j^{+}\geq 1$ and $\mu_j^{-1}\leq -2$ for $j>0$.
	The growth condition on $V$ implies that $a_j^{+}=0$ for all $j>0$. So we obtain 
	$$V=c+\frac{l}{2r}+O(r^{-2})$$
	Here $l$ is the degree of the $S^1$ bundle  $F:\mathcal X\rightarrow\mathcal Q$. 
	So we have proved that $\bm\omega^{\Diamond}$ hence $(X, g)$ is  ALF-$A_k$ for $k=l-1$.  By the positive mass theorem of  Minerbe \cite{Minerbe}, we know $k\geq 0$.
		
In the case $Y\equiv\dR^3/\dZ_2$,   by taking the $\dZ_2$-cover outside a compact set, we may reduce to the previous case. Then $(X^4, g)$ is an ALF-$D_k$ gravitational instanton. By Biquard-Minerbe \cite{BM}, we have $k\geq 0$.	\end{proof}

	\subsubsection{\emph{Case}  $\dim_{\infty}(X^4)=2$}
	\begin{definition}[ALG models] 
	ALG model ends are defined as follows.
	\begin{enumerate}
	\item Let $\beta\in\{\frac 12, \frac 13, \frac 23, \frac 14, \frac34, \frac 16, \frac 56, 1\}$. 
		 Let $\textbf{C}_\beta$ be the flat cone defined in Example \ref{example: cone-special-kahler} with the canonical hyperk\"ahler metric on $T^*\textbf{C}_{\beta}$. Taking a lattice sub-bundle in $T^*\textbf{C}_\beta$ which is invariant under the monodromy $\widetilde R_\beta$ (c.f. \eqref{e:rotation matrix admissible}), the induced torus bundle gives rise to a (flat)  ALG$_\beta$ model end. 
		 \item An ALG model end is an ALG$_\beta$ model end for some $\beta$ in the above list.
		 \end{enumerate}
		  \end{definition}
	\begin{definition}[ALG$^*$ models]  ALG$^*$ model ends are defined as follows.
	\begin{enumerate}
		\item 
		An ALG$^*$-$I_k$ (for $k\in \dZ_+$) model end is obtained by applying the Gibbons-Hawking ansatz on $S^1\times \dR^2\setminus K$ to the harmonic function $V=k\cdot\log r $, where $r$ is the radial distance function on $\dR^2$.
		\item An ALG$^*$-$I_k^*$ (for $k\in \dZ_+$) model end is a $\dZ_2$ quotient of an ALG$^*$-$I_{2k}$ model end, where $\dZ_2$ action covers the standard involution on $\dR^2$ and the rotation by $\pi$ on $S^1$.
		\item 	 An ALG$^*$ model end is an ALG$^*$-$I_k$ or ALG$^*$-$I_k^*$ model end for some $k\in\dZ_+$.
 \end{enumerate}

	\end{definition}
			 
	\begin{theorem}\label{t:ALG end classification}
	Any gravitational instanton $(X^4, g)$ with $\dim_{\infty}(X^4)=2$ is either ALG or ALG$^*$. 
		\end{theorem}
	\begin{remark}
	Combining the weighted analysis developed in \cite{CVZ2} and a direct generalization of Minerbe's positive mass theorem \cite{Minerbe}, one can conclude that $ALG_1$ and ALG$^*$-$I_k$ gravitational instantons do not exist.  We thank Gao Chen for pointing out this.  It is also proved in \cite{CVZ2} that any ALG$^*$-$I_k$ gravitational instanton satisfies $1\leq k\leq 4$.
		On the other hand, there exist ALG$_\beta$ gravitational instantons for all $\beta\in \{\frac 12, \frac 13, \frac 23, \frac 14, \frac34, \frac 16, \frac 56\}$, and there exist ALG$^*$-$I_k^*$ gravitational instantons for all $k\in\{1,2,3,4\}$, which follows from the work of  
 Hein \cite{Hein}. They live on the complement of a singular fiber of finite or $I_k^*$ monodromy on a rational elliptic surface. \cite{CVZ2} proved a partial converse to Hein's theorem. \end{remark}
\begin{proof}
	
	Since $\dim_{\infty}(X^4)=2$, $\bm\omega^{\Diamond}$ has local $\dT^2$ symmetry but may have global monodromy. We divide into several subcases.  Let $\sigma$ be a loop generating $\pi_1(\mathcal Q)$ which goes around the vertex $p_*\in Y$ once counterclockwise, and let $A_\sigma\in SL(2;\dZ)$ be the corresponding monodromy of the $\dT^2$ fiber. Notice that the quotient metric $g^{\Diamond}$ on $\mathcal Q$ is special K\"ahler metric, with monodromy conjugate to $A_\sigma$. 
	
First, assume $A_\sigma=\widetilde R_\beta$ for some $\beta\in\{\frac 12, \frac 13, \frac 23, \frac 14, \frac34, \frac 16, \frac 56, 1\}$.  In this case we know the asymptotic cone $Y$ is given by $\textbf{C}_\beta$. Then by the discussion at the end of Section \ref{s-6} we can find global holomorphic coordinates $\zeta=(z-\sqrt{-1}w)^{1/\beta}$, such that $(z, w)$ is a pair of local special holomorphic coordinates. Again $\tau$ is not single-valued in general but $\xi^k$ ($k=2$ or $k=3$ depending on $\beta$) is single-valued. Since the asymptotic cone of $(\mathcal Q ,g^{\Diamond})$ is $\textbf{C}_\beta$, similar to the proof of Theorem \ref{t:complete-space-2d}, one can show that the flat metric $\sqrt{-1}\beta^2|\zeta|^{2\beta-2}d\zeta\wedge d\bar\zeta$ is complete at infinity. It then follows that as $\zeta\rightarrow\infty$ we have $\xi^k\rightarrow 0$ so $\xi^k=\psi(\zeta^{-1})$ for  a holomorphic function $\psi$. In particular, $\tau=\sqrt{-1}+O(|\zeta|^{-1/k})$.  
	It then follows that the special K\"ahler metric $g^{\Diamond}$ is polynomially asymptotic to the standard flat cone metric in the $\zeta$ coordinate. Now the $\mathcal N$-invariant metric $g_{\bm\omega^{\Diamond}}$ is determined by $g^{\Diamond}$ via \eqref{e:semiflat metric formula}. It follows that $g_{\bm\omega^{\Diamond}}$ is ALG, so is $(X, g)$.

	Next consider the case $A=I_k$ for some $k\geq 1$. Then we have an invariant vector of $A_\sigma$. This implies that there is a globally $S^1$-action on $\mathcal X$. In particular, $\bm\omega^{\Diamond}$ is given by the Gibbons-Hawking ansatz on some special affine metric 3-manifold.  Similar to the case $\dim_{\infty}(X^4)=3$, the growth estimate on $V$ gives a complete flat background geometry at infinity whose asymptotic cone has dimension $2$. Then the flat background geometry is itself isometric to $(S^1\times\dR^2)\setminus K$ for some compact $K$. So we can use spectral decomposition to conclude that\begin{align*}V=k\cdot\log r + c  + O(r^{-\epsilon}).\end{align*} We may assume $c=0$ by changing the coordinates on $\dR^2$. In this case, $(X^4, g)$ is ALG$^*$-$I_k$. 
	
	Finally when $A=I_k^*$ for some $k\geq 1$. Then we pass to a double cover and reduce to the previous case. In this case  $(X, g)$ is ALG$^*$-$I_k^*$ . 
\end{proof}

	\subsubsection{\emph{Case}  $\dim_{\infty}(X^4)=1$}
	\begin{definition}[ALH models]
		An ALH model is the  hyperk\"ahler metric on the product $\dT^3\times [0, \infty)$ for some flat $\dT^3$.
	\end{definition}
	\begin{definition}[ALH$^*$ models]
	ALH$^*$ model ends are defined as follows
	\begin{enumerate}
		\item 
		An ALH$^*_b$ (for some $b\in \dZ_+$) model end is the the hyperk\"ahler metric obtained by applying the Gibbons-Hawking ansatz on the  product $\dT^2\times [0, \infty)$ to the harmonic function $V=bz$ for some $b\in \dZ_+$, where $\dT^2$ is a flat 2-torus with area $2\pi$ and $z$ is the standard coordinate on $[0, \infty)$. 
		\item
An ALH$^*$ model end is an ALH$^*_b$ model end for some $b\in \dZ_+$. Notice an ALH$^*$ model end is precisely a Calabi model end discussed in \cite{HSVZ}.
\end{enumerate}
		\end{definition}

	\begin{theorem}\label{t:ALH* end classification}
Any gravitational instanton $(X^4, g)$ with $\dim_{\infty}(X^4)=1$ is either ALH or ALH$^*$.
	\end{theorem}

	\begin{proof}
	In this case, $\bm\omega^{\Diamond}$, has either a $\dT^3$ or $\mathscr{H}_1$ symmetry. Then it is itself an ALH or ALH$^*$ model end. Consequently in the first case $(X^4, g)$ is ALH and in the second case it is ALH$^*$. 
	\end{proof}

	\begin{remark}
ALH gravitational instantons were constructed by Tian-Yau \cite{TY} and Hein \cite{Hein} on the complement of a smooth fiber in a rational elliptic surface. Chen-Chen \cite{CCIII} proved a Torelli theorem for ALH gravitational instantons; it is also shown that ALH gravitational instantons actually have an improved exponential decay rate. ALH$^*_b$ (for $1\leq b\leq 9$) gravitational instantons have two constructions: Tian-Yau  metrics \cite{TY} live on the complement of a smooth anti-canonical divisor in a weak del Pezzo surface, and Hein  metrics \cite{Hein} live on the complement of an $I_b$-fiber in a rational elliptic surface. Conversely by Remark \ref{r:6-21} below we know an ALH$^*_b$ gravitationl instanton must satisfy $1\leq b\leq 9$.\end{remark}

In the next subsection we will prove an exponential decay for  ALH$^*$ gravitational instantons.

\subsection{Exponential decay in the ALH$^*$ case}

Let $(X^4, g)$ be an ALH$^*_b$ gravitational instanton. As before we fix  a choice of hyperk\"ahler triple $\bom$. By Theorem \ref{t:ALH* end classification},  there exist $\epsilon>0$ and some compact set $K$ such that $X^4\setminus K$ is smoothly identified with an ALH$_b^*$ model end $(\Ca, \bom_{\Ca})$ with $|\nabla^k_{\bm\omega_{\Ca}}(\bm\omega-\bm\omega_{\Ca})|_{\bm\omega_{\Ca}}=O(r^{-k-\epsilon})$ for all $k\in\dN$, where $r$ is the distance function with respect to $\bm\omega_{\Ca}$. The goal of this subsection is to prove the following.  
\begin{theorem}\label{t:exponential decay}
	Let $(X^4,\bom)$ be an ALH$_b^*$ gravitational instanton. Then there exist $\delta_0>0$ and a diffeomorphism $F$ from the end of $\Ca$ to $X^4$ such that the following holds for all $k\in\dN$,
		\begin{align}|\nabla^k_{\bom_{\Ca}}(F^*\bom-{\bom_{\Ca}})|_{{\bom_{\Ca}}}\leq C_k\cdot e^{-\delta_0 r^{\frac{2}{3}}}.\label{e:ALH*-exp-decay}\end{align}
	\end{theorem}
	\begin{remark}
	The main interest in this theorem lies in the fact the asymptotic geometry of the Calabi model space is non-standard, which has several geometric meaningful scales. The latter is already seen in the analysis in \cite{HSVZ}. We expect the idea here can be applied to more general problems. 
	\end{remark}

	\begin{remark}\label{r:6-21}
	This theorem connects well with \cite{HSVZ, HSVZ2}. On the one hand, by construction the ALH$_b^*$ gravitational instantons of Tian-Yau and Hein all have the exponential decay properties as stated in Theorem \ref{t:exponential decay}. Such a decay rate is a typical rate of a decaying harmonic function on the model end $\Ca$.    On the other hand, under the improved decay assumption \eqref{e:ALH*-exp-decay}, \cite{HSVZ2} proved a partial converse to the Tian-Yau and Hein constructions in the complex-analytic sense. In particular, any ALH$_b^*$ gravitational instanton can be compactified to a rational elliptic surface or a weak del Pezzo surface. It also implies that an ALH$_b^*$ gravitational instanton must satisfy $1\leq b\leq 9$.
	\end{remark}
Now we summarize the geometry of $\Ca$. Recall $(\Ca, \omega_{\Ca})$ is given by applying the Gibbons-Hawking ansatz to $V=bz$ on $\dT^2\times [z_0, \infty)$ for some flat $\dT^2$ with area $2\pi$ and $z_0\geq 10$. Then $C^{-1}r^{\frac{2}{3}}\leq z\leq Cr^{\frac{2}{3}}$. 
 Notice  $\Ca$ admits a natural nilpotent group action which gives rise to the $\mathcal N$-structure, i.e., there is a nilpotent orbit $\mathcal{N}(\bx)$ at every point $\bx\in\Ca$. Moreover, $\diam_{g_{\mathcal{N}(\bx)}}(\mathcal N(\bx))\sim r(\bx)^{\frac{1}{3}}$, $\Injrad_{g_{\Ca}}(\bx)\sim r(\bx)^{-\frac{1}{3}}$, and $\Vol_{g_{\Ca}}(B_r)\sim r^{\frac{4}{3}}$; see Section 2 of \cite{HSVZ} for more details. Before proving Theorem \ref{t:exponential decay}, we introduce a simple but useful lemma.
\begin{lemma}\label{l:initial polynomial decay}
There is a triple of 1-forms $\bm\sigma$ such that ${\bom}=\bom_{\Ca}+d\bm\sigma$ and for all $k\in\dN$, $\epsilon>0$,
\begin{equation}\label{e:initial estimate decay}|\nabla^k_{\bom_{\Ca}} \bm \sigma|_{\bom_{\Ca}}=O(r^{\frac{1}{3}-k+\epsilon}).	
\end{equation}
\end{lemma}
 \begin{proof} Since the intrinsic diameter of the $\mathcal N$-orbits with respect to $\bm\omega$ has the order $r_{\bom}^{\frac13}$ for the $\bom$-distance function  $r_{\bom}$, the $\mathcal{N}$-orbits in the rescaled annulus $s^{-1}A_{s, 2s}(p)$ has diameter decay  $\sim s^{-\frac23}$.
 Now we repeat the construction of the $\mathcal N$-invariant hyperk\"ahler metric on the end of $X^4$. First, taking the average of $\bom$ along the $\mathcal N$-orbits, we obtain a closed definite triple $\bom^{\dag}$ which is cohomologous to $\bom$. Notice for any vector field $\zeta$ generating a family of diffeomorphisms $\phi_t(t\in [0, 1])$, we have that 
 \begin{align*}\phi_1^*\bm\omega-\bom=\int_0^1\frac{d}{dt}\phi_t^*\bom dt=d\Big(\int_0^1\phi_t^*(\zeta\lrcorner\bom)dt\Big).\end{align*}
 Using this we can write $\bom^{\dag}=\bom+d\bm\sigma_1$ for some $\bm\sigma_1$ satisfying $|\nabla^k_{\bom_{}} \bm \sigma_1|_{\bom_{ }}=O((r_{\bom})^{\frac{1}{3}-k+\epsilon})$ for all $k\in \dN$ and $\epsilon>0$
.
Applying the construction in Section \ref{ss:6-3} to $\bom^{\dag}$ (with $\delta_0=\frac{1}{3}$), we obtain a new $\mathcal N$-invariant hyperk\"ahler triple $\bom^{\Diamond}$. By Theorem \ref{t:complete invariant hk triple}, $\bom^{\Diamond}=\bom^{\dag}+d\bm\sigma_2$ with
$|\nabla^k_{\bom^{\dag}} \bm \sigma_2|_{\bom^{\dag}}=O((r_{\bom^{\dag}})^{\frac{1}{3}-k+\epsilon})$ for all $k\in\dN$ and $\epsilon>0$. Now in fact $\bom^{\Diamond}$ coincides with $\bom_{\mathcal C}$, as both are $\mathcal N$-invariant hyperk\"ahler triples which are asymptotic to each other at infinity. 
 \end{proof}

Passing to a finite cover of $\Ca$ we can assume $b=1$. 
Now we take a closer look at the deformations of hyperk\"ahler triples in Section \ref{ss:triple-deformation}.  
 A triple of $\bm\omega_{\Ca}$-self-dual $2$-forms $\bm{\theta}^+$ can be identified with a $(3\times 3)$ matrix-valued function $A_{\bm\theta^+}$. The $4$ dimensional space of $(3\times 3)$ matrices $M=(M_{ij})$ satisfying $M^{T}+M=\lambda \Id$ for some $\lambda\in\dR$ is isomorphic to $\dR^4$ via
$M\mapsto (\Tr(M)/3,M_{23}, M_{31}, M_{12})$. 
Define a linear operator: 
\begin{align*}R: \Omega^{+}_{\bm\omega}\otimes \dR^3\rightarrow C^\infty(\mathcal C)\otimes \dR^4; \quad \bm\theta^+ \mapsto \frac{1}{2}\Big(A_{\bm\theta^+}-(A_{\bm\theta^+})^T\Big)+\frac{1}{3} \Tr(A_{\bm\theta^+})\cdot \Id. \end{align*}
It follows that $R(\bm\theta^+)=0$ if and only if $A_{\bm\theta^+}$ is symmetric and trace-free. This can serve as a gauge fixing condition due to the fact that if the triple $\bm\omega_{\Ca}+d\bm\eta$ is hyperk\"ahler, then by the discussion before \eqref{e:elliptic-system}, $R(d^+\bm\eta)=0$  implies $
d^+\bm\eta= \mathfrak{F}(\TF(-Q_{\bm{\omega}_{\Ca}}-S_{d^-\bm\eta}))
$. The latter is elliptic in $\bm\eta$ when coupled with   $d^*\bm\eta=0$ (this can always be achieved).  In the proof of Theorem \ref{t:exponential decay}, we first fix the gauge such that $R(d^+\bm\omega)=0$, and then improve the decaying order using the convexity properties of the linearized elliptic PDEs. 

\

\textbf{Step 1 (Gauge fixing).}

\

On $(\Ca, \bm\omega_{\Ca})$, we 
choose the complex structures $J_1, J_2, J_3$ so that 
\begin{equation*}
\begin{cases}
	J_1dx=dy; J_1dz=z^{-1}\theta;\\
	J_2dy=dz; J_2dx=z^{-1}\theta;\\
	J_3dz=dx; J_3dy=z^{-1}\theta, 
	\end{cases}
\end{equation*}
where $\theta$ is the connection 1-form in the Gibbons-Hawking construction. 
It follows that $\omega_\alpha$ is K\"ahler with respect to $J_\alpha$, where
\begin{equation*}
	\begin{cases}
		\omega_1=zdx\wedge dy+dz\wedge\theta;\\
		\omega_2=zdy\wedge dz+dx\wedge\theta;\\
		\omega_3=zdz\wedge dz+dy\wedge \theta.
	\end{cases}
\end{equation*}
Without loss of generality we may assume $\bom_{\Ca}\equiv (\omega_1, \omega_2, \omega_3)$. 
Given a vector field $\xi$, we have the induced infinitesimal deformation $\mathcal L_\xi\bm\omega_{\Ca}=d(\xi\lrcorner\bm\omega_{\Ca}).$
For our purpose we only consider vector fields generated by a 4-tuple of functions $\underline f\equiv (f_0, f_1, f_2, f_3)$ via $\xi_{\underline f}=\nabla f_0+\sum_{\alpha=1}^3 J_\alpha\nabla f_\alpha.$
Then by straightforward computation, we obtain
\begin{equation*}
\begin{cases}
\mathcal L_{\xi_{\underline f}}\omega_\alpha\wedge \omega_\alpha=\frac{1}{2}\mathcal L_{\xi_{\underline f}}(\dvol_g)=2\Delta f_0\dvol_g,  \ \ \ \ \ \alpha=1, 2, 3;\\
	\mathcal L_{\xi_{\underline f}}\omega_1\wedge\omega_2=-\mathcal L_{\xi_{\underline f}}\omega_2\wedge\omega_1=2\Delta f_3\dvol_g;\\
	\mathcal L_{\xi_{\underline f}}\omega_2\wedge\omega_3=-\mathcal L_{\xi_{\underline f}}\omega_3\wedge\omega_2=2\Delta f_1\dvol_g;\\
	\mathcal L_{\xi_{\underline f}}\omega_3\wedge\omega_1=-\mathcal L_{\xi_{\underline f}}\omega_1\wedge\omega_3=2\Delta f_2 \dvol_g,
	\end{cases}
\end{equation*}
and if $f_\alpha=c_\alpha z$, $\alpha=0, 1, 2,3$,  then 
\begin{equation}\label{e:variation caused by z}
\begin{cases}
	{\xi_{\underline f}}\lrcorner\omega_1=-c_0z^{-1}\theta-c_1dz-c_2dx-c_3dy;\\
	{\xi_{\underline f}}\lrcorner\omega_2=c_0dy+c_1dx-c_2dz-c_3z^{-1}\theta;\\
	{\xi_{\underline f}}\lrcorner\omega_3=-c_0dx+c_1dy+c_2z^{-1}\theta-c_3dz.
	\end{cases}
\end{equation}
Define the linear operator 
$$\mathbb L: C^\infty(\mathcal C)\otimes \dR^4\mapsto C^\infty(\mathcal C)\otimes \dR^4; \underline f\equiv(f_0, f_1, f_2, f_3)\mapsto R((\mathcal L_{\xi_{\underline f}}\bm\omega_{\mathcal C})^{+})$$
By the above calculation we have 
$ \mathbb  L(\underline f)=\Delta_{\bom_{\mathcal C}}\underline f$. 
Denote by $\mathcal Q_{w}$ the region $\{z\geq w\}$. 
 We adopt the definition of weighted H\"older spaces for tensors on $\mathcal Q_w$ given in Section \ref{ss:6-3}.  All the norms appearing in this subsection are taken with respect to $\bom_{\mathcal C}$.

\begin{proposition}\label{p:solving Laplace equation on Calabi model}
Given any $\delta\in (-\infty, 0)\setminus \{-2\}$, $k\geq 20$, $\alpha\in (0,1)$, for all $w\geq z_0$, there exists a bounded linear map 
$$\Delta^{-1}_{\bom_{\mathcal C}}: C^{k-2, \alpha}_{\delta}(\mathcal{Q}_w)\rightarrow C^{k, \alpha}_{\delta+2}(\mathcal{Q}_w)$$
such that for any $v\in C^{k-2, \alpha}_{\delta}(\mathcal{Q}_w)$, $u=\Delta_{\bom_{\mathcal C}}^{-1}v$ solves $\Delta_{\bm\omega_{\mathcal C}} u=v$, and with \begin{align*}\|u\|_{C^{k, \alpha}_{\delta+2}(\mathcal Q_w)}\leq C(k, \alpha, \delta)\|v\|_{C^{k-2, \alpha}_{\delta}(\mathcal{Q}_w)}.\end{align*} 
\end{proposition}
\begin{remark}\label{r: non unique solution}
Here the solution $u$ is not unique. For example, the function $z$ is always harmonic. The latter fact will be useful later. 
\end{remark}

\begin{proof} 
The weighted $C^0$-estimate for $u$ follows from Proposition \ref{p:weighted-in-z} (with $\tau=\frac{3}{2}\delta$) and the relation $C^{-1}r^{\frac{2}{3}}\leq z\leq Cr^{\frac{2}{3}}$. The higher order estimate can be proved by standard weighted Schauder estimates. \end{proof}

Now we fix $\epsilon>0$ sufficiently small, $k\geq 20$, and $\alpha\in (0, 1)$.

\begin{proposition}[Gauge fixing I]\label{p:first gauge fixing}
	For $w\gg1$, there is a $C^{k-1,\alpha}$-diffeomorphism $F$ from $\mathcal Q_w$ onto the end of $\mathcal C$ such that $F^*{\bom}={\bom_{\mathcal C}}+d\bm\sigma'$,
	$R(d^+ \bm\sigma')=0$, and $\|\bm\sigma'\|_{C^{k-1, \alpha}_{\frac{1}{3}+\epsilon}(\mathcal Q_w)}<\infty$.
\end{proposition}

\begin{proof}  
We will apply the implicit function theorem to find the desired diffeomorphism. We write $\bm{\omega} = \bm{\omega}_{\Ca} + d\bm{\sigma}$, where $\bm \sigma$ has the growth in \eqref{e:initial estimate decay}.	
 Then we have $|\nabla^l_{{\bom_{\mathcal C}}}R(d^+\bm \sigma)|_{{\bom_{\mathcal C}}}=O(r^{-\frac{2}{3} - l +\delta})$
for all $l\geq 0$ and $\delta>0$. We will make further improvements of decay rate of $R(d^+\bm \sigma)$.

Let $\underline f\equiv - {\Delta_{\bom_{\Ca}}^{-1}}(R(d^+\bm \sigma))$, and denote by $F_t(t\geq 0)$ the family of diffeomorphisms generated by the vector field $\xi_{\underline f}$. Let $\widetilde{\bom}\equiv F_1^*{\bom}$. 
 Then we have
\begin{align*}\widetilde{\bom}-{\bom_{\mathcal C}}=d\bm\sigma+(F_1^*{\bom_{\mathcal C}}-{{\bom_{\mathcal C}}})+ d(F_1^*\bm\sigma-\bm\sigma).\end{align*}
Notice that 
\begin{align*}\frac{d}{dt}F_t^*{{\bom_{\mathcal C}}}=F_t^*(\mathcal L_{\xi_{\underline f}}{{\bom_{\mathcal C}}})= d(F_t^*(\xi_{\underline f}\lrcorner{\bom_{\mathcal C}})=d(\xi_{\underline f}\lrcorner{\bom_{\mathcal C}}+\bm\beta),\end{align*}
where $|\nabla^l_{{\bom_{\mathcal C}}}\bm\beta|_{{\bom_{\mathcal C}}}=O(r^{-\frac{1}{3}-l+\delta})$ for any $l\in\dN$, $\delta>0$. 
Similarly we have 
$|\nabla^l_{{\bom_{\mathcal C}}}\mathcal L_{\xi_{\underline f}}\bm\sigma|_{{\bom_{\mathcal C}}} = O(r^{-\frac{1}{3}-l+\delta})$.
Therefore, $\widetilde{\bom}-{\bom_{\mathcal C}}=d\bm{\sigma}'$,
where $\bm{\sigma}'=\bm{\sigma}'_1+\bm\sigma'_2$, $R(d^+\bm\sigma_1')=0$, and  \begin{align*}|\nabla^l_{{\bom_{\mathcal C}}}\bm\sigma'|_{{\bom_{\mathcal C}}} = O(r^{\frac{1}{3}-l+\delta}), \quad |\nabla^l_{{\bom_{\mathcal C}}}\bm\sigma_2'|_{{\bom_{\mathcal C}}} = O(r^{-\frac{1}{3}-l+\delta}).\end{align*}
In particular,
$|\nabla^l_{{\bom_{\mathcal C}}}R(d^+\bm{\sigma}')|_{{\bom_{\mathcal C}}} = O(r^{-\frac{4}{3}-l+\delta})$. Next we take $\underline h\equiv - {\Delta_{\bom_{\Ca}}^{-1}}(R(d^+\bm{\sigma}'))$, and denote by $G_t$  the family of diffeomorphisms generated by $\xi_{\underline h}$.  Let $\widetilde{\bom}'\equiv G_1^*\widetilde{\bom}$, and write $\widetilde{\bom}'-{\bom_{\mathcal C}}=d\bm\sigma''$. Similarly,  
  \begin{equation}\label{e:improved error} |\nabla^l_{{\bom_{\mathcal C}}}\bm\sigma''|_{{\bom_{\mathcal C}}} = O(r^{\frac{1}{3}-l+\delta}), \quad  |\nabla^l_{{\bom_{\mathcal C}}}R(d^+\bm{\sigma}'')|_{{\bom_{\mathcal C}}}= O(r^{-2-l+\delta}).
  \end{equation}
 Then we have 
$\underline{u}\equiv {\Delta_{\bom_{\Ca}}^{-1}}(R(d^+\bm\sigma''))=O( r^{\delta})$,	
 and $\xi_{\underline{u}}(\bx)\leq C r(\bx)^{-1+\delta}$ which is much smaller than $\Injrad(\bx)\sim r^{-\frac{1}{3}}(\bx)$ at $\bx$ as $r$ large. 

To apply the implicit function theorem on Banach spaces, we will use another way to generate diffeormorphisms from a vector field. Given a vector field $\xi$ whose $C^1$ norm at each point is much smaller than the injectivity radius of $(\mathcal C, \bom_{\mathcal C})$, we define $\Phi_\xi(x)\equiv \exp_x(\xi(x))$. 
For some fixed $\gamma>0$, we consider the map \begin{align*}\mathscr F: C^{k, \alpha}_{\gamma}(\mathcal Q_w)\otimes \dR^4\longrightarrow C^{k-2, \alpha}_{-2+\gamma}(\mathcal Q_w)\otimes \dR^4;\quad \underline f\mapsto R((\Phi_{\xi_{\underline f}}^*{\bm\omega})^+-\bom_{\mathcal C}).\end{align*}
First, $\mathscr F$ is a differentiable map. Indeed, this follows from the fact that $\exp_x(V)$ is a smooth map on the tangent bundle. Set $\mathscr L\equiv\Delta_{\bom_{\Ca}}$, and  
 $\mathscr N\equiv \mathscr F-\mathscr L$. Then $\mathscr P\equiv\Delta_{\bom_{\Ca}}^{-1}$ is a right inverse to $\mathscr L$ with $\|\mathscr P\|\leq L$ for $L>0$ independent of $w$. 
For $\eta>0$ sufficiently small and independent of $w$, we have for $\underline f, \underline g\in B_\eta(0)$ that \begin{align*}\|\mathscr N(\underline f)-\mathscr N(\underline g)\|_{C^{k-2, \alpha}_{-2+\gamma}(\mathcal Q_w)}\leq (3L)^{-1}\|\underline f-\underline g\|_{C^{k,\alpha}_{\gamma}(\mathcal Q_w)}.\end{align*}
Moreover, letting $\delta\equiv \frac{\gamma}{2}$, by \eqref{e:improved error} we have $\|\mathscr F(\bo)\|_{C^{k-2, \alpha}_{-2+\gamma}(\mathcal Q_w)}\leq C_0 w^{-\frac{3\gamma}{4}}$.
Applying Proposition \ref{p:implicit function theorem}, for $w\gg1$, there exists an $\underline{f}\in C^{k,\alpha}_{\gamma}(\mathcal Q_w)$, with $R((\Phi_{\xi_{\underline f}}^*\bom)^+-\bom_{\mathcal C})=0$ and  $\|\underline f\|_{C^{k,\alpha}_{\gamma}(\mathcal Q_w)}<2C_0Lw^{-\frac{3\gamma}{4}}$. 
 Since $|\xi_{\underline f}|\leq 2C_0 L w^{-\frac{3}{4}\gamma}r^{-1+\gamma}$, we see that for $w\gg1$, $F\equiv\Phi_{\xi_{\underline f}}$  is a diffeomorphism from $\mathcal Q_w$ into $\mathcal C$. Then $F^*\bom$ satisfies the desired properties.  
\end{proof}

In the above proposition, 
replacing $\bom$ by $F^*\bom$ and noticing $Q_{\bm\omega_{\mathcal C}}=\Id$, one sees that \eqref{e:elliptic-system} holds, i.e.,
\begin{equation}\label{e:d+ equation}d^+\bm\sigma=\mathfrak F(\TF(-S_{d^-\bm\sigma})).
\end{equation}
   By Proposition \ref{p:solving Laplace equation on Calabi model}, we can solve $d^*d\bm u=-d^*\bm\sigma$ and replace $\bm\sigma$ by $\bm\sigma+d\bm u$, so that
\begin{equation}\label{e:d* equation}d^*\bm\sigma=0,\end{equation}
and we still have the weighted estimate \begin{equation}\label{e:a rough bound}\|\bm\sigma\|_{C^{k-1, \alpha}_{\frac{1}{3}+\epsilon}(\mathcal Q_w)}<\infty.	
\end{equation}
Now \eqref{e:d+ equation} and \eqref{e:d* equation} form an elliptic system, with linearization at $\bm\sigma=0$ given by the Dirac operator $d^+\oplus d^*$ on $(\mathcal C, \bom_{\mathcal C})$. Notice at this point the pointwise norm $|\bm\sigma|$ may still grow at infinity.

\

\textbf{Step 2 ($|\bm\sigma|$ is decaying at infinity).} 

\

On $\Ca$, we can write  $\bm\sigma=\bm p_0 z^{-1}\theta+\bm p_1 dx+\bm p_2 dy+\bm p_3dz$, 
where $\bm p_j(j=0, 1,2,3)$ are globally defined $\dR^3$-valued functions on $\mathcal C$. Notice the pointwise norm  $|\bm\sigma|=(\sum_j |\bm p_j|^2)^{\frac{1}{2}}$. The following shows that $|\bm\sigma|$ is decaying at a polynomial rate at infinity.

\begin{proposition}
For all $\delta>0$ we have $\|\bm\sigma\|_{C^{k-1, \alpha}_{-\frac{1}{3}+\delta}(\mathcal Q_w)}<\infty$.
\label{l:a polynomial decay}
\end{proposition}
\begin{proof}
We denote 
$h(r)\equiv\|\bm\sigma\|_{C^{k-1, \alpha}_{-\frac{1}{3}+\delta}(A_{r, 2r})}, $
where $A_{r, 2r}=\mathcal Q_{r^{\frac{2}{3}}}\setminus \mathcal Q_{(2r)^\frac{2}{3}}$. Then we define 
\begin{equation}\label{e:definition of H}H(r)\equiv{h(2r)}/{h(r)}.	
\end{equation}
   It suffices to prove 
 \begin{equation}\limsup_{r\rightarrow\infty}H(r)\leq 2^{-1/3}.
\label{e:limsup control}	
 \end{equation}  
Then the conclusion follows from an easy iteration. 
 To prove \eqref{e:limsup control}, we use a contradiction argument. Suppose  there exists a $\delta>0$ such that 
  $\limsup_{r\rightarrow\infty}H(r)>2^{-1/3}+\delta>2^{-1/3}$.
Then we can find a sequence $r_j\rightarrow\infty$ so that $H(r_j)>2^{-1/3}+\delta$. 
Now we claim that $\liminf_{r\rightarrow\infty}H(r)=\infty$. 
This would imply that $|\sigma|$ is growing faster than any polynomial rate at infinity, and then we reach a contradiction with \eqref{e:a rough bound}. 

To prove the claim, we again use a contradiction argument. Suppose we can find $s_j\rightarrow\infty$ such that $H(s_j/2)>2^{-1/3}+\delta$ but $H(s_j)\leq C$ for some $C>0$.  Then we consider the sequence of rescaled spaces $(A_{s_j/2, 4s_j}, s_j^{-1}g_{\bm\omega})$. Passing to a subsequence they converge to the interval $(\frac{1}{2}, 4)$ in the asymptotic  cone $\dR_+$ of $\mathcal C$. The universal cover converges to a hyperk\"ahler limit $A_\infty$ which admits a fibration $\pi: A_\infty\rightarrow (\frac{1}{2}, 4)$ with fibers the Heisenberg algebra $\mathscr H_1$. Denote by $\widetilde{\bm\sigma}$ the lifted triple of 1-forms on the universal cover of $A_{s_j/2, 4s_j}$. Let $\widetilde{\bm\sigma}_j=h(s_j)^{-1}\widetilde{\bm\sigma}$. Passing to a subsequence we have weak $C^{k-1,\alpha}$ convergence of $\widetilde{\bm\sigma}_j$ to a limit $\widetilde{\bm\sigma}_\infty$ on $A_\infty$. Using the fact that ${\bm\sigma}$ satisfies the elliptic system \eqref{e:d+ equation}, \eqref{e:d* equation} and interior Schauder estimates we may assume $\widetilde{\bm\sigma}_j$ converges strongly in $C^{k-1, \alpha}$ to  $\widetilde{\bm\sigma}_\infty$ on $\pi^{-1}(1, 2)$ which is invariant under $\mathscr H_1$. Moreover, it satisfies the linear system $d^+\widetilde{\bm\sigma}_\infty=d^*\widetilde{\bm\sigma}_\infty=0$. A straightforward computation shows that then we must have $\widetilde{\bm\sigma}_\infty=\bm c_0z_\infty^{-1}\theta_\infty+ \bm c_1dx_\infty+\bm c_2dy_\infty+\bm c_3dz_\infty$, where  $(x_\infty, y_\infty, z_\infty, t_\infty)$ are the standard coordinates on $A_\infty$ (as given in Section \ref{ss:3-3}),  $\theta_\infty=dt_\infty+x_\infty dy_\infty$ and $\bm c_j$ are constant vectors. It follows that $|\widetilde{\bm\sigma}_\infty|=C |z_\infty|^{-1/2}.$ This contradicts our assumption.
\end{proof}

\

\textbf{Step 3 (Decay faster than any polynomial rate).}

\

On $\mathcal C$, the $\mathcal N$-invariant kernel space of the Dirac operator $d^+\oplus d^*$ acting on 1-forms is spanned by $dx, dy, dz$ and $z^{-1}\theta$. These forms decay exactly at the rate $r^{-1/3}$. So in order to improve the decay rate of $\bm\sigma$, we need to gauge out these elements. The first three forms are $d$-closed so are easy to deal with; the form $z^{-1}\theta$ is \emph{not} $d$-closed, and we have to invoke yet another implicit function theorem to eliminate it. For this reason we also make use of the variation of $\bm\omega_{\mathcal C}$ induced by $\xi_{z}$ (c.f. \eqref{e:variation caused by z}), which by Remark \ref{r: non unique solution} does not destroy the previous gauge fixing condition $R((\bm\omega-\bm\omega_{\mathcal C})^{+})=0$. We denote by $S(w)$ the hypersurface $\{z=w\}\subset \mathcal C$, endowed with the induced Riemannian metric from $\bm\omega_{\mathcal C}$. By calculation, we have $\text{Vol}(S(w))=Cw^{1/2}$.

\begin{proposition}[Gauge fixing II]\label{p:fix average}
	For  $w\gg1$, we can find a diffeomorphism $F_{w}$ from $\mathcal Q_w$ onto the end of $\mathcal C$, such that  $\bom_w\equiv F_w^*{\bom}=\bom_{\mathcal C}+d\bm\sigma_w$, with
	\begin{equation}\label{e:improved gauge}
		\begin{cases}
			\|\bm\sigma_w\|_{C^{k-1, \alpha}_{-\frac{1}{3}+\delta}(\mathcal Q_w)}<\infty \ \ \ \ \text{for all}  \ \ \delta>0;\\R(d^+\bm{\sigma}_w)=0(\Longrightarrow d^+\bm\sigma_w=\mathfrak F(\TF(S_{d^-\bm\sigma_w}))); \\
			d^*\bm\sigma_w=0;\\
			\int_{S({2^\frac{2}{3}w})} \bm p_j(\bm\sigma_w)=0, \ \ \ \ \ j=0, 1,2, 3.
		\end{cases}
	\end{equation} 
	
	\end{proposition}
\begin{proof}
Fix $\delta>0$ small, and define the Banach spaces
\begin{align*}
 \mathfrak{D}  & \equiv\Big(C^{k, \alpha}_{\frac{5}{3}+\delta}(\mathcal Q_w)\otimes \dR^4 \Big) \oplus \Big(C^{k, \alpha}_{\frac{5}{3}+\delta}(\mathcal Q_w)\otimes \dR^3 \Big),\\    \mathfrak{J} &   \equiv	 \Big(C^{k-2, \alpha}_{-\frac{1}{3}+\delta}(\mathcal Q_w)\otimes \dR^4\Big) \oplus \Big(C^{k-2, \alpha}_{-\frac{1}{3}+\delta}(\mathcal Q_w)\otimes \dR^3\Big) \oplus \dR^3,
 \end{align*}
 where we fix a standard norm on $\dR^3$.
Then we define a map 
$\mathscr F: \mathfrak{D} \to \mathfrak{J}$ by sending $(\underline f, \bm u)$ to
 \begin{align*}\Big(R((\Phi_{\xi_{\underline f}}^*(\bm\omega)-\bm\omega_{\mathcal C})^{+}), d^*(\bm\beta_{\xi_{\underline f}}+\Phi_{\xi_{\underline f}}^*\bm\sigma+d\bm u), \frac{1}{w^{\frac{5}{2}+\frac{3}{2}\delta}}\int_{S({2^\frac{2}{3}w})}\bm p_0(\bm\beta_{\xi_{\underline f}}+\Phi_{\xi_{\underline f}}^*\bm\sigma+d\bm u)\Big).\end{align*}
  In the above definition,  $\bm p_0(\bm\alpha)$ is the $z^{-1}\theta$-component of a triple of $1$-forms $\bm\alpha$,  and given  a vector field $\xi$, we denote $\Phi_{\xi}(x)\equiv\exp_x(\xi(x))$ and $\bm\beta_{\xi}\equiv \int_0^1 \Phi_{t\xi}^*(\frac{d}{dt}\Phi_{t\xi}\lrcorner\bom_{\mathcal C})dt$. Immediately,  $d\bm\beta_\xi=\Phi_{\xi}^*\bom_{\mathcal C}-\bom_{\mathcal C}$. 
 One can directly verify that $\mathscr F$ is a differentiable map and, by Proposition \ref{l:a polynomial decay}, $\|\mathscr F(0)\|\leq Cw^{-\frac{1}{2}\delta-\frac{3}{2}}$. Moreover, the differential $d\mathscr F$ at $0$ is given by $d\mathscr F_0=\mathscr L+\mathscr M$, where
\begin{align*}\mathscr L(\underline g, \bm v)&=\Big(\Delta_{\bom_{\Ca}}\underline g, d^*(\xi_{\underline g}\lrcorner \bom_{\mathcal C}+d\bm v), \frac{1}{w^{\frac{5}{2}+\frac{3}{2}\delta}}\int_{S({2^{\frac{2}{3}}w})} \bm p_0(\xi_{\underline g}\lrcorner\bom_{\mathcal C}+d\bm v)\Big), \\\mathscr M(\underline g, \bm v)&=\Big(R((\mathcal L_{\xi_{\underline g}}(d\bm\sigma))^{+}), d^*(\mathcal L_{\xi_{\underline g}}\bm\sigma),  \frac{1}{w^{\frac{5}{2}+\frac{3}{2}\delta}}\int_{S({2^{\frac{2}{3}}w})} \bm p_0(\mathcal L_{\xi_{\underline g}}\bm\sigma)\Big).\end{align*}
Notice that $\|\mathscr{M}\|$ is small when $w\gg1$. 
By \eqref{e:variation caused by z}, if $\underline g=z\underline{c}$, then $\bm p_0(\xi_{\underline g}\lrcorner\bom_{\mathcal C})={\bm {\widehat {c}}}\equiv (-c_0, -c_3, c_2)$.  
We define a right inverse  $\mathscr P: \mathfrak{J}\longrightarrow \mathfrak{D}$ of $\mathscr L$ 
by $\mathscr P(\underline h, \bm x, \bm q)\equiv(\underline g, \bm v)$ which is given  as follows. Given $(\underline h, \bm x, \bm q)$, we first let $\underline g_0={\Delta_{\bom_{\Ca}}^{-1}}(\underline h)$ and let $\bm v=\Delta^{-1}_{\bom_{\mathcal C}}(\bm x-d^*(\xi_{\underline g_0}\lrcorner \bom_{\mathcal C}))$. Then we let $\underline g=\underline g_0+z\underline c$ for a constant vector $\underline c$ with $c_1=0$ and for $\alpha\neq 1$,  $c_\alpha$ is uniquely determined by \begin{align*}\bm{\widehat c}=\frac{1}{\Vol(S(2^{\frac{2}{3}}w))}\cdot\Big(w^{\frac{5}{2}+\frac{3}{2}\delta}\cdot \bm q-  \int_{S({2^{\frac{2}{3}}w})}\bm p_0(\xi_{\underline{g}_0}\lrcorner\bm{\omega}_{\Ca} + d\bm v)\Big).\end{align*}  By \eqref{e:variation caused by z}, $d^*(\xi_{z\underline c}\lrcorner\bom_{\mathcal C})=0$. So it follows that $\mathscr L\circ \mathscr P=\Id$.

It is straightforward to check that $\|\mathscr P\|\leq C$ for $C>0$ independent of $w$. Moreover, one can directly estimate the non-linear term $\mathscr N\equiv \mathscr F-\mathscr L$. 
Then as in the proof of Proposition \ref{p:first gauge fixing} we can find a zero $(\underline f,\bm u)$  of $\mathscr F$ for $w$ large, using the implicit function theorem (Theorem \ref{p:implicit function theorem}).
 
Now we set $F_w\equiv \Phi_{\xi_{\underline f}}$. Then for $w\gg1$, $F_w$ is a diffeomorphism from $\mathcal Q_w$ into the end of $\mathcal C$. We  write $F_w^*\bom=\bom_{\mathcal C}+d\bm\sigma'$ with $\bm\sigma'=\bm \beta_{\underline f}+\Phi_{\xi_{\underline f}}^*\bm\sigma+d\bm u$. Then $R(d^+\bm{\sigma}')=0$ and $d^*\bm\sigma'=0$. By Proposition \ref{l:a polynomial decay} we have  $\|\bm\sigma_w\|_{C^{k-1, \alpha}_{-\frac{1}{3}+\delta}(\mathcal Q_w)}<\infty $ for all $\delta>0$. Also the last condition in \eqref{e:improved gauge} is satisfied for $j=0$. Now by adding to $\bm\sigma'$  the triple of 1-forms $\bm e_1 dx+\bm e_2 dy+\bm e_3 dz$ for appropriate constant vectors $(\bm e_1, \bm e_2, \bm e_3)$, we can make sure that $\bm\sigma'$ also satisfies the last condition in \eqref{e:improved gauge} for $j>0$.  Notice  $dx, dy, dz$ are $d$-closed and $d^*$-closed, and their norm decays at the rate $r^{-1/3}$, so the new $\bm\sigma$ still has the required decaying properties.
\end{proof}

We denote by $H_{w}(r)$ the function $H(r)$ given in \eqref{e:definition of H} associated to $\bm\sigma_w$.

\begin{proposition}\label{p:limsup H goes to zero}
	There exists $W_0>0$ such that for all $w\geq W_0$, we have $\limsup_{r\rightarrow\infty}H_w(r)=0$.  
\end{proposition}
\begin{proof}
First we claim $\lim_{w\rightarrow\infty} H_w(w^\frac{3}{2})=0. $ If not, then as in the proof of Proposition \ref{l:a polynomial decay} we can take a rescaled limit and obtain a non-trivial limit $\widetilde{\bm\sigma}_\infty$. Now the extra normalization condition \eqref{e:improved gauge} implies that the limit must be identically zero. A contradiction. 

Now suppose that $\limsup_{r\rightarrow\infty} H_w(r)>0$, then by \eqref{e:limsup control} we can find $w_j\rightarrow\infty$, and $r_j\geq w_j^{\frac{3}{2}}$ such that $H_{w_j}(r_j)\in (0, \frac{1}{2}\cdot 2^{-1/3})$. But then  we take rescaled limit and again get a limit $\widetilde{\bm\sigma}_\infty$. Still we obtain a contradiction with the fact that $|\widetilde{\bm\sigma}_\infty|$ must be decaying at order $z_\infty^{-1/2}$. 
\end{proof}

Now we let $w=W_0$ and let $\bm\sigma\equiv \bm\sigma_{W_0}$.
Proposition \ref{p:limsup H goes to zero} easily implies that for all $m\geq 0$
\begin{equation}\label{e:faster than any order decay}\|\bm\sigma\|_{C^{k-1, \alpha}_{-m}(\mathcal Q_{W_0})}<\infty.	
\end{equation}
This means that
$\bm\sigma$ decays faster than any polynomial rate.

\

\textbf{Step 4 (Exponential decay).}

\

\noindent First we prove an improvement of Proposition \ref{p:fix average}.

\begin{lemma}[Gauge fixing III]
	There exists $W_1\geq W_0$ such that for all $w>W_1$ we can find a diffeomorphism $F_{w}$ defined on the fixed $\mathcal Q_{W_1}$, such that $\bom_w\equiv F_w^*{\bom}=\bom_{\mathcal C}+d\bm \sigma_w$, with \begin{equation}\label{e:gauge fixing III}
		\begin{cases}
		\|\bm\sigma_w\|_{C^{k-1, \alpha}_{-10}(\mathcal Q_{W_1})}\leq C; \\
			R(d^+\bm\sigma_w)=0(\Longrightarrow d^+\bm\sigma_w=\mathfrak F(\TF(S_{d^-\bm\sigma_w})));\\
			d^*\bm\sigma_w=0;\\
			\int_{S(w+1)} \bm p_j(\bm\sigma_w)=0, \ \ \ \ \ j=0, 1,2, 3.
		\end{cases}
	\end{equation} 
	Furthermore, $F_w$ is of the form $\Phi_{\xi}$ for $\xi=\xi_{\underline{f_w}}$ with \begin{equation}\label{e:uniform bound on Fw}\|\underline{ f_w}\|_{C^{k, \alpha}_{-9}(\mathcal Q_{W_1})}\leq C.	
 \end{equation}

\end{lemma}
\begin{proof}
The proof is the same as Proposition \ref{p:fix average}.  The difference here is that we can  now fix the domain $\mathcal Q_{W_1}$. This follows from the fact that  the rapid decay of $\bm\sigma$ guaranteed by \eqref{e:faster than any order decay} implies $|\bm  p_0(\bm\sigma)|$ sufficiently small on $S(W_1)$ for $w_1\gg1$.  This also implies that the above constant $C$ is independent of $w$.
\end{proof}

Now we  exploit a different scale of the asymptotic geometry of $\mathcal C$.  We denote $A_{z_1, z_2}\equiv \mathcal Q_{z_1}\setminus \mathcal Q_{z_2}$, and $A^\infty_{z_1, z_2}\equiv \dT^2\times (z_1, z_2)\subset \dT^2\times\dR$, where $\dT^2$ is the flat $2$-torus involved in the definition of $\Ca$. Then for any fixed $C>0$, as $z\rightarrow\infty$, the rescaled annulus $z^{-1/2}A_{z-C, z+C}$ (with respect to the metric $\bom_{\mathcal C}$) collapses with uniformly bounded curvature to the domain  $A^\infty_{-C, C}$ in the product cylinder $\dT^2\times \dR$.
We define
\begin{align*}n_w(z)\equiv z^{-1}\int_{A_{z, z+1}}|\bm\sigma_w|_{g_{\Ca}}^2\dvol_{\Ca}.\end{align*}
 The following arguments are well-known in the study of asymptotically conical geometries (see for example, \cite{CT94, Hein-Sun}),   and they can easily be adapted to our setting.  Let $\underline\lambda_1$ be the first eigenvalue of $-\Delta_{\dT^2}$. 
\begin{lemma}[Convexity lemma]\label{l:convexity}
For all $\delta\in (-\underline\lambda_1, \underline\lambda_1)\setminus\{0\}$, there exists a $W_2=W_2(|\delta|)>W_1$ such that for all $w\geq W_1$, if  $\log (n_w(W_2+1))\geq  \log(n_w(W_2))+\delta$, then $\log(n_w(z+1))\geq\log(n_w(z))+\delta$ for all $z\geq W_2$. 
\end{lemma}
\begin{proof}
If this fails, then we can find a sequence $w_j\geq W_1$ and $z_j\rightarrow\infty$ such that $\log n_{w_j}(z_j+1)\geq \log n_{w_j}(z_j)+\delta$, but $\log n_{w_j}(z_j+2)<\log n_{w_j}(z_{j}+1)+\delta$. Passing to a subsequence we may assume $\widetilde A_j\equiv z_j^{-1/2}A_{z_j, z_j+3}$ collapses to the domain $A^\infty_{0,3}$ in $\dT^2\times \dR$. We want to take the limit of $\bm\sigma_{w_j}$. First set $\bm\sigma_j\equiv z_j^{-1/2}n_{w_j}(z_j+1)^{-1/2}\bm\sigma_{w_j}$. Then the average of $|\bm\sigma_j|$ over $\widetilde A_j$ is uniformly bounded. Moreover,   $\bm\sigma_j$ satisfies the elliptic system on $\widetilde A_j$:
\begin{equation*}
\begin{cases}
	d^+\bm\sigma_j=n_{w_j}(z_j+1)^{1/2}\cdot z_j^{\frac{1}{2}}\cdot \mathfrak F(\TF(S_{d^-\bm\sigma_j})) \\
	d^*\bm\sigma_j=0.
	\end{cases}
\end{equation*}
Since $\|\bm\sigma_w\|_{C^{k-1, \alpha}_{-10}(\mathcal Q_{W_1})}<C$ for all $w\geq W_1$, by interior elliptic estimates for   $d^+\oplus d^*$ over local universal covers one can see that $\|\bm\sigma_j\|_{C^{k-2}(z_j^{-1/2}A_{z_j+\tau, z_j+3-\tau})}\leq C(\tau)$ uniformly for any $\tau>0$ small. In particular, passing to a subsequence we may obtain $C^{k-3}$ convergence of $\bm\sigma_j$ to $\bm\sigma_\infty$ over local universal covers. Globally we obtain  a pair $(\bm q, \bm\lambda)$ on $A^\infty_{0,3}$, where $\bm q$ is a vector-valued function and $\bm\lambda$ is a vector-valued 1-form. This is similar to the discussion of convergence of hyperk\"ahler structures under codimension 1 collapsing in Section \ref{ss:3-1}: $\bm q$ is given by the $\bm\sigma_\infty(\p_t)$, and $\bm\lambda$ is given by the horizontal component of $\bm\sigma_\infty$.  The pair $(\bm q, \bm \lambda)$ satisfies $d \bm q+*d\bm \lambda=0, d^*\bm \lambda=0$. In particular both $ \bm q$ and $\bm \lambda$ are harmonic. Denote $n_\infty(z)\equiv\int_{A^\infty_{z, z+1}}|\bm q|^2+|\bm \lambda|^2$. Then by construction and the strong interior convergence  we have $\log(n_\infty(1))=0$,  $\log(n_\infty(0))\leq -\delta$, and $\log(n_\infty(2))\leq \delta$.

Given a vector-valued harmonic function $\underline u$ on $A^\infty_{0,3}$, it is easy to see via eigenfunction expansion that  
\begin{align*}\|\underline u\|_{L^2(A^\infty_{0, 1})}\cdot\|\underline u\|_{L^2(A^\infty_{2, 3})}\geq \|\underline u\|_{L^2(A^\infty_{1,2})}^2,\end{align*}
and the equality holds if and only if $\underline u$ is homogeneous, i.e., $\underline u=\underline ce^{\lambda z}\phi_\lambda$ for some eigenfunction $\phi_\lambda$ on $\dT^2$.  
Applying this to the above limit $( \bm q, \bm \lambda)$, it follows that $(\bm  q, \bm \lambda)$ must be homogeneous and $\delta$ must be an eigenvalue on $\dT^2$. This contradicts our hypothesis on $\delta$.
\end{proof}

\begin{lemma}\label{l:gap lemma}
There is a $\hat\delta\in (0, \lambda_1)$ and $W_3>W_1+W_2(\hat\delta)$ such that for all $w\geq W_1$ and $z\geq W_3$ we must have $|\log n_w(z+2)-\log n_w(z+1)|\geq\hat\delta$.
\end{lemma}
\begin{proof}
	Suppose the conclusion fails, then we get a contradicting sequence $w_j\geq W_1$ and $z_j\rightarrow\infty$ with $|\log n_{w_j}(z_j+2)-\log n_{w_j}(z_j+1)|<j^{-1}$. Notice by Lemma \ref{l:convexity} we know for $j\gg1$, $\log n_{w_j}(z_j+1)-\log n_{w_j}(z_j)\leq j^{-1}$. Then we can as in the above proof pass to a subsequence and obtain a limit pair $(\bm q, \bm \lambda)$ on $\dT^2\times \dR$. This time we use the last condition in \eqref{e:gauge fixing III} to conclude that $\int_{\{z=1\}} \bm q=\int_{\{z=1\}} \bm \lambda=0$.

	 Now it is easy to see, using eigenfunction expansion again, that  there exists a $\delta>0$ such that for a vector valued harmonic function $u$ on $A^\infty_{0, 3}$ with $\int_{\{z=1\}}u=0$, either
$\|u\|_{L^2(A^\infty_{1, 2})}\geq e^{\delta}\|u\|_{L^2(A^\infty_{0, 1})}$ or $ 
 \|u\|_{L^2(A^\infty_{1, 2})}\leq e^{-\delta}\|u\|_{L^2(A^\infty_{0, 1})}$.	This leads to a contradiction.
\end{proof}

\begin{lemma}
	For all $w>W_3$, we  have $\log n_{w}(w+1)-\log n_{w}(w)\leq -\hat\delta$.
\end{lemma}
\begin{proof}
If not, then by Lemma \ref{l:gap lemma}	we must have $\log n_w(w+1)-\log n_w(w)\geq \hat\delta$ for some $w>W_3$. Now Lemma \ref{l:convexity} implies that $\log n_w(z+1)-\log n_w(z)\geq \hat\delta$ for all $z\geq w$. This implies $\bm\sigma_w$ has exponential growth, contradiction. 
\end{proof}

Now given any $w>W_3$, by Lemma \ref{l:convexity} again we have $\log n_w(z+1)-\log n_w(z)\leq -\hat\delta$ for all $z\in [W_3+1, w]$. In particular we must have $n_w(z)\leq Cn_w(W_3)e^{-\hat\delta z}$. Using the elliptic system satisfied by $\bm\sigma_w$ one can see that passing to a subsequence $w_j\rightarrow\infty$, $\bm\sigma_{w_j}$ converges  in $C^\infty_{\text{loc}}$ to a smooth limit $\bm\sigma_\infty$ with $\|\nabla_{\bm\omega_{\mathcal C}}^k\bm\sigma_\infty\|_{\bom_{\mathcal C}}\leq C_ke^{-\delta_0 z}$ for all $k\geq 0$ and $\delta\in (0, \hat\delta)$. 

Finally by \eqref{e:uniform bound on Fw} we may also assume $F_{w_j}$ converges in $C^{k-2, \alpha}_{\text{loc}}$ to a limit $F_\infty$, which is again a diffomeomorphism from $\mathcal Q_{W_1}$ onto the end of $\mathcal C$, such that $F_\infty^*\bom=\bom_{\mathcal C}+d\bm\sigma_\infty$. Notice both $\bom$ and $\bm\sigma_\infty$ are smooth, so $F_\infty$ is indeed smooth.
This completes the proof of Theorem \ref{t:exponential decay}.

\section{Discussions and questions}

\subsection{Towards a bubble tree structure}

Let $(X_j^4, g_j, p_j)$ be a sequence of hyperk\"ahler manifolds such that $\overline{B_2(p_j)}$ is compact and $(X_j^4, g_j, p_j)\xrightarrow{GH}(X_\infty, d_\infty, p_\infty)$. Denote $d\equiv\dim_{\ess}(X_\infty)$.  If $d=4$ , then $X_\infty$ is a hyperk\"ahler orbifold, and it is well-known that there is a finite bubble tree structure associated to the convergence (c.f. \cite{Bando90}). Now we assume $d<4$ and $\int_{B_2(p_j)} |\Rm_{g_j}|^2\dvol_{g_j}\leq \kappa_0$ uniformly for some $\kappa_0>0$. It is more involved to describe the bubble tree structure in this case. Here we make some initial steps. 
 
By Theorem \ref{t:3d local version} and Theorem \ref{t:2d local version} we know there is a unique tangent cone at $p_\infty$, which is a flat metric cone and we denote it by $(Y, d_Y, p^*)$.
  Clearly $Y\in\mathcal B_{p_\infty}$. 
Given $j\geq1$ and $\lambda>0$ we denote by $X_{j, \lambda}$ the rescaled space $(X_j^4, \lambda^2g_j, p_j)$, and by $v_{j, \lambda}$ the volume of the unit ball around $p_j$ in $X_{j, \lambda}$. By the Bishop-Gromov volume comparison we know that for a fixed $j$, $v_{j, \lambda}$ is an increasing function of $\lambda$. So if we rescale sufficiently large we will get complete hyperk\"ahaler orbifolds as bubble limits. The following shows there is an essentially unique scale that leads to a complete hyperk\"ahler orbifold which is collapsing at infinity.

\begin{proposition}[Maximal scale non-collapsing bubble]\label{p:maximal bubble}
Given any sequence $j_i\rightarrow\infty$, passing to a subsequence we may find $\lambda_i\rightarrow\infty$ such that the rescaled spaces $X_{j_i, \lambda_i}$ converge to a complete hyperk\"ahler orbifold $(Z, d_Z, p_Z)$ such that \emph{$\text{Vol}(B_R(p_Z))=o(R^4)$ as $R\rightarrow\infty$.	}\end{proposition}
\begin{proof}
If this does not hold, then by the Bishop-Gromov volume comparison  one may find a sequence of ALE hyperk\"ahler orbifolds $Z_k\in \mathcal B_{p_\infty}$ such that the structure group $\Gamma_k$ at infinity satisfies that $|\Gamma_k|\rightarrow\infty$.  By  a diagonal sequence argument for each $k$ we may find sequences $j_{i,k}\rightarrow\infty$, $\lambda_{i,k}\rightarrow\infty$, and a domain $U_{i,k}\subset X_{j_{i,k}, \lambda_{i,k}}$ such that $\p U_{i,k}$ converges smoothly to the space form ${S}^3/\Gamma_k$. It follows from \cite{Bando90}  that $U_{i,k}$ is diffeomorphic to an ALE gravitational instanton with structure group $\Gamma_k$ at infinity, which implies $\chi(U_{i,k})=c_k$ for $i$ large. Notice that $c_k\rightarrow\infty$ as $|\Gamma_k|\rightarrow\infty$. On the other hand, by  the Chern-Gau\ss-Bonnet theorem we have a uniform bound on $\chi(U_{i,k})$ in terms of $\kappa_0$. This is a contradiction. \end{proof}

	Notice Theorem \ref{t:thm1.2} also holds for complete hyperk\"ahler orbifolds with finite energy, since the proof only uses the end structure at infinity.  So the bubble limits constructed in the above Proposition must be of type AL$\mathfrak X$, and is not ALE. This gives a rigorous explanation of the heuristic fact that non-ALE gravitational instantons are responsible for collapsing of hyperk\"ahler manifolds.

Next we show that the dimension of bubble limits can only increase when we zoom into a smaller scale. This is again true from the intuition. 
\begin{proposition}
	[Dimension monotonicity]\label{p:dimension monotonicity}
In the setting above,  we have $\dim_{\ess}(Z)\geq d$ for all  $Z\in \mathcal B_{p_\infty}$.
\end{proposition}
 \begin{proof}

From the definition of $\mathcal B_{p_\infty}$, we can find $\lambda_0>0$, $j_0>0$, and $\epsilon_j\rightarrow 0$ such that for all $j\geq j_0$ and $\lambda\geq \lambda_0$, there is an element $Z_{j, \lambda}\in \mathcal B_{p_\infty}$ satisfying $d_{GH}(X_{j, \lambda}, Z_{j, \lambda})\leq \epsilon_j$. Notice  a priori $Z_{j, \lambda}$ is not unique, and we simply make an arbitrary choice for each $j$ and $\lambda$. It is clear that we can take $Z_{j, \lambda_0}=Y$ for $j\gg1$. 

The conclusion in the case $d=1$ is trivial. We will first prove the case $d=3$.
 Notice every element $Z$ in $\mathcal B_{p_\infty}$ with $\dim_{\ess}Z=3$ belongs to the list given in Theorem \ref{t:complete-space}, among which there are exactly two elements $Z_1=\dR^3$ and $Z_2=\dR^3/\dZ_2$ that are metric cones.  Fix $\delta>0$ small such  that any $Z$ in $B_{2\delta}(Y)\cap \mathcal B_{p_\infty}$ satisfies $\dim_{\ess}(Z)\geq 3$, and $B_{2\delta}(Y)\cap \mathcal B_{p_\infty}\cap \{Z_1, Z_2\}=\{Y\}$. 
For $j$ large, we let $\lambda_j$ be the smallest $\lambda$ such that $d_{GH}(X_{j, \lambda}, Y)\geq \delta$. It is clear that $\liminf_{j\rightarrow\infty}\lambda_j=\infty$. 

We claim there is some $\tau>0$ such that $v_{j, \lambda_j}\geq\tau$ for all $j$ large.  
 Given this it follows that  $v_{j, \lambda}\geq\tau$ for all $\lambda\geq \lambda_j$, and then the conclusion follows easily. 
To prove the claim, suppose otherwise, then there is a sequence $j_i\rightarrow\infty$ such that $d_{GH}(X_{j_i, \lambda_{j_i}}, Y)=\delta$ but $v_{j_i, \lambda_{j_i}}\rightarrow0$. Passing to a further subsequence we may assume $X_{j_i,\lambda_{j_i}}$ converges to a limit $Z_\infty$ with $\dim_{\ess}(Z_\infty)=3$, and $d_{GH}(Y, Z_\infty)=\delta$. So $Z_\infty$ is one of  the spaces listed in Theorem \ref{t:complete-space} and $Z_\infty\notin \{Z_1, Z_2\}$. Similar to the proof of Lemma \ref{l:asymptotic cone classification} one can show this is impossible.  For example, suppose $Z_\infty=S^1_R\times\dR^2$ for some $R>0$. Then for $i$ large we know $Z_{j_i, \lambda_{j_i}}$ must also be $S^1_{R_i}\times \dR^2$ for $R_i\rightarrow R$. It follows that $Z_{j_i, \lambda_{j_i}/2}$ must be $S^1_{R_i'}\times \dR^2$ for some $R_i'\rightarrow R/2$. Then $d_{GH}(Z_{j_i, \lambda_{j_i}/2}, Y)>\frac{3}{2}\delta$ for $i$ large. This contradicts our choice of $\lambda_{j_i}$.

Now we consider the case $d=2$. We may assume for $j$ large $\int_{B_2(p_j)}|\Rm_{g_j}|^2\dvol_{g_j}\in [l\epsilon, (l+1)\epsilon)$ for some integer $l>0$,  where $\epsilon$ is the constant given in Theorem \ref{t:epsilon regularity HK}. We will prove the conclusion by induction on $l$.  First consider the case $l=1$. 
 Then we can proceed similarly to the case $d=3$. The energy bound implies that any $Z\in \mathcal B_{p_\infty}$ has at most one singularity. Hence by Theorem \ref{t:complete-space-2d}, any $Z\in \mathcal B_{p_\infty}$ with $\dim_{\ess}(Z)=2$ is either isometric to $\textbf{C}_\beta$ for $\beta\in \mathbb A\equiv \{\frac 12, \frac 13, \frac 23, \frac 14, \frac34, \frac 16, \frac 56, 1\}$, or $S^1\times \dR$. As above we fix $\delta>0$ small so that any $Z$ in $B_{2\delta}(Y)\cap \mathcal B_{p_\infty}$ satifies $\dim_{\ess}(Z)\geq 2$ and $B_{2\delta}(Y)\cap \mathcal B_{p_\infty}\cap \{\textbf{C}_\beta|\beta\in \mathbb A\}=\{Y\}$. Furthermore, we may assume $\beta=1$ if there is some $S^1_R\times \dR$ in $B_{2\delta}(Y)\cap \mathcal B_{p_\infty}$. For $j$ large let $\lambda_j$ be the smallest $\lambda$ such that $d_{GH}(X_{j, \lambda}, Y)\geq \delta$. Passing to a subsequence we may assume $X_{j, \lambda_j}$ converges to a limit $Z_\infty$ with $\dim_{\ess}(Z_\infty)\geq 2$ and $d_{GH}(Y, Z_\infty)=\delta$.  We claim any such limit $Z_\infty$ must satisfy $\dim_{\ess}(Z_\infty)\geq 3$. If not, then there is a limit $Z_\infty$ with $\dim_{\ess}(Z_\infty)=2$. By the choice of $\delta$ it follows that $Z_\infty=S^1_R\times\dR$ for some $R>0$ and  $Y=\dR^2$. Then similar reasoning as the case $d=3$ yields a contradiction. Given the claim then our conclusion follows from the result in the case $d=3$. 

 Suppose now the conclusion holds for $l\leq l_0$, and  consider the case $l=l_0+1$. If we run the same arguments as above, then in the end we can conclude that any limit $Z_\infty$ either satisfies $\dim_{\ess}(Z_\infty)\geq3$, or $\dim_{\ess}(Z_\infty)=2$ and $Z_\infty$ has at least 2 singularities. If the latter occurs, suppose $Z_\infty$ is given as the limit of some subsequence $X_{j_i, \lambda_{j_i}}$, then there exists $\tau>0$ such that for $i $ large, we have $\int_{B_{\tau\lambda_{j_i}}(p_{j_i})}|\Rm_{g_{j_i}}|^2\dvol_{g_{j_i}}<l_0\epsilon$. It follows from the induction assumption that any Gromov-Hausdorff limit of $X_{j_i, s_i \lambda_{j_i}}$ for $s_i\geq 1$ has dimension at least 2. Using this one can finish the proof of the case $l=l_0+1$.
\end{proof}

 \begin{corollary}
 	The following hold
 	\begin{itemize}
 		\item Any $Z\in \mathcal B_{p_\infty}$ with $\dim_{\ess}(Z)=3$ is isometric to $\dR^3$ or $\dR^3/\dZ_2$. If $d=3$ then any $Z\in \mathcal B_{p_\infty}$ is isometric to either the tangent cone $Y$, or an ALE or ALF hyperk\"ahler orbifold; 
 		\item Any $Z\in \mathcal B_{p_\infty}$ with $\dim_{\ess}(Z)=2$ and with a unique singularity is isometric to $\textbf{C}_\beta$ for some $\beta\in \mathbb A$.
 	\end{itemize}
 \end{corollary}
 
The first item says that the construction of Foscolo \cite{Foscolo} essentially gives the complete picture in the case $d=3$ (modulo further development of orbifold singularities in the ALE bubbles).  With more work one expects to obtain a full bubble tree structure. The latter may also be used to prove the following. Notice by Remark \ref{r:integeral monodromy} the statement is false without the uniform energy bound. 
 \begin{conjecture}[Integral monodromy]\label{conj:integral monodromy}
 	If $d=2$, then the singular special K\"ahler metric has integral monodromy. 
 \end{conjecture}
For hyperk\"ahler metrics on the K3 manifold, it may even be possible to explicitly classify all the possible bubble trees. 
 
 \subsection{Asymptotics of the period map}
 
Let $X_\infty^d$ be the Gromov-Hausdorff limit of a sequence of hyperk\"ahler metrics $g_j$ on the K3 manifold $\mathcal K$ with $d\equiv \dim_{\ess} (X_\infty^d)<4$. As before we make a choice of hyperk\"ahler triple $\bom_j$ for $g_j$. We can use Theorem \ref{t:local inv hk triple} and Theorem \ref{t:thm1.1} to obtain some information on the behavior of $\mathcal P(g_j)$ as $j\rightarrow\infty$. For example, suppose $d=3$. Then $X_\infty=\mathbb T^3/\mathbb Z_2$. We choose $\mathcal Q$ to be of the form $U/\mathbb Z_2$, where $U\subset \mathbb T^3=\mathbb R^3/\mathbb Z^3$ is the complement of a small neighborhood of the finitely many points which map to the singular set $\mathcal S$ in $X_\infty$ (notice $\mathcal S$ contains the 8 orbifold points, but it may also contain some other points). Take 3 disjoint geodesic circles $C_{\alpha} (\alpha=1, 2, 3)$ in $U$ which lift to lines in $\mathbb R^3$ parallel to the 3 coordinate axes. Denote by $l_\alpha$ the length of $C_{\alpha}$. For $j$ large, we have a circle bundle $F_j: \mathcal Q_j\rightarrow \mathcal Q$.  Denote by $E_{j,\alpha}=F_j^{-1}(C_{ \alpha})$ the 2-cycles in $\mathcal Q_j$. They span a 3 dimensional isotropic subspace of $H_2(\mathcal K; \mathbb Z)$. Since the hyperk\"ahler tripe $\bom_j^{\Diamond}$ given by Theorem \ref{t:local inv hk triple} is $\mathcal N$-invariant, passing to the $\mathbb Z_2$-cover, the metric $g_j^{\Diamond}$ is given by the Gibbons-Hawking ansatz on $U$. That is,  we may write $g_j^{\Diamond}=\epsilon_j^2(V_j \cdot g_{U}+V_j^{-1}\theta_j^2)$, where $V_j$ is a positive harmonic function on $U$ with $V_j\sim \epsilon_j^{-2}$, $d\theta_j=*dV_j$,  and $\epsilon_j\rightarrow 0$. It follows that  $\int_{E_{j, \alpha}}\omega_{j, \beta}=\int_{E_{j, \alpha}}\omega_{j, \beta}^{\Diamond}=\delta_{\alpha\beta}\cdot \epsilon_j^ 2l_{\beta}$, where $\epsilon_jV_j^{-1/2}$ is the length of the $S^1$ orbit with respect to $g_j^{\Diamond}$. We also know the volume $\int_{X_j}\omega_{j, \alpha}^2\sim \epsilon_j^2$.
	
	A simple consequence  is that this case can not occur for collapsing  \emph{polarized} K3 surfaces. For if not, then without loss of generality we may suppose for some $\lambda_j>0$ that  $\alpha_j=[\lambda_j\cdot \omega_{j, 1}]$ is a class in $H^2(X;\mathbb Z)$ with $\int_X\alpha_j^2=\sigma$ independent of $j$.  It follows that $\int_X \omega_{j, 1}^2=\sigma\lambda_j^{-2}$. So we must have $\epsilon_j\sim \lambda_j^{-2}$ and $\lambda_j\rightarrow\infty$. Then $\int_{E_{j, 1}}\alpha_j=\lambda_j\int_{E_{j,1}}\omega_{j, 1}\sim \lambda_j^{-3}$. Since the integral is always an integer, this is impossible. 
	
	Similarly one can treat the case $d=2$ and $d=1$, and in each case there is some isotropic subspace of $H_2(\mathcal K; \mathbb Z)$ on which we know the asymptotics of the period of the hyperk\"ahler triple. These isotropic subspaces also appear naturally in the Satake compactifications of the locally symmetric space $\Gamma\setminus O(3, 19)/(O(3)\times O(19))$.   One expects this is relevant to the conjecture in \cite{OO} mentioned in the introduction. We leave it for future work.

 \subsection{Topological properties and the $L^2$-curvature energy}

There are easy consequences of Theorem \ref{t:thm1.2} which yield topological restrictions on the underlying manifolds of a gravitational instanton. Notice any noncompact paracompact smooth manifold admits a complete Riemannian metric with quadratic curvature decay.
\begin{corollary}
	The  Euler characteristic of a  non-flat gravitational instanton  is positive and finite. 
\end{corollary}
\begin{proof}
This follows by applying the Chern-Gau\ss-Bonnet theorem \eqref{e:GBC1} to $B_r(p)$ and let $r$ tends to infinity. Using the asymptotics at infinity it is easy to see that the boundary term goes to zero, hence 
$$\chi(X)=\frac{1}{8\pi^2} \int_X |\Rm_g|^2\dvol_g>0.$$
\end{proof}
\begin{corollary}
	A non-flat gravitational instanton  has vanishing first Betti number.
\end{corollary}
\begin{proof}
By Theorem \ref{t:thm1.2}, we know there exists $\kappa\in \{1, \frac{4}{3}, 2, 3\}$ such that 	$\text{Vol}(B(p, r))\sim Cr^{\kappa}$ as $r\rightarrow\infty$. If $\kappa\leq 2$, then the conclusion follows from a result of Anderson \cite{Anderson}. If $\kappa=3$, then $X$ is ALF. Consider the rescaled spaces $(X^4, R^{-2}g, p)$ which collapse to its asymptotic cone as $R\rightarrow\infty$. If $b_1(X)>0$, then by \cite{NZ} we know the collapsing must have uniformly bounded curvature on compact sets. The latter implies $X$ is flat.  
\end{proof}

As an immediate application, we consider the smooth quadric $Q=\{x^2+y^2=1\}$ in $\dC^2$. Since $\pi_1(\dC^2\setminus Q)=\dZ$ and $\chi(\dC^2\setminus Q)=1$, it follows that  $\dC^2\setminus Q$ does not support any gravitational instanton. 
The interest of this example lies in the fact that it admits a nowhere vanishing holomorphic 2-form $\Omega$, but we have shown that the Calabi-Yau equation $\omega^2=C\Omega\wedge\bar\Omega$ does not have a solution which is complete at infinity and has finite energy. Notice $\dC^2$ admits a Ricci-flat K\"ahler metric $\omega_{\beta}$ with cone angle $2\pi\beta$ along $Q$ for any $\beta\in (0, 1]$, by a generalized Gibbons-Hawking ansatz \cite{Donaldson-conic}. It is an interesting question to understand the behavior of $\omega_\beta$ as $\beta\rightarrow 0$. Notice that $\mathbb C^2\setminus Q$ is the same as $\mathbb C\mathbb P^2\setminus D$ for $D$ a singular elliptic curve given by the union of a line and a conic.  In the case when $D$ is smooth,  it is a consequence of the result of Biquard-Guenancia in \cite{BG} that when $\beta\rightarrow0$, under suitable rescalings the conical K\"ahler-Einstein metrics on $\mathbb C\mathbb P^2\setminus D$ converge to the complete Calabi-Yau metric constructed by Tian-Yau in \cite{TY}. In our case one would expect a very different picture; it is interesting to explore the connection with certain algebro-geometric ``stability" notion.
\

It is natural to study when a complete Ricci-flat metric on an open 4 manifold has finite $L^2$ energy. In this regard we make the following conjetural topological criterion

\begin{conjecture}
[Energy Finiteness Conjecture] 
Let $(X, g)$ be a complete Ricci-flat 4-manifold, then 
\begin{align*}\int_{X}|\Rm_g|^2\dvol_g<\infty\end{align*}
if and only if 
$X$ has finite topological type.
\end{conjecture}

Even for exotic $\mathbb R^4$ we do not yet know the answer. The known infinite energy examples of gravitational instantons constructed by Anderson-Kronheimber-LeBrun have infinite Euler characteristic; see \cite{AKL}.

 \subsection{Generalizations}
The ideas and techniques developed in this paper can be likely adapted to more general settings. The first natural extension is to the case of K\"ahler-Einstein metrics on complex surfaces with non-positive Ricci curvature. In particular, the following question is sensible

\begin{problem}
	Classify complete K\"ahler-Einstein metrics with finite energy in complex dimension two.
  \end{problem}
   
   More generally, one can study the structure of collapsed  Einstein metrics and more general canonical metrics in 4 dimensions, and higher dimensional metrics of special holonomy, under suitable curvature assumptions. One interesting question is 
   
   \begin{question}
   	Do Proposition \ref{p:maximal bubble} and  \ref{p:dimension monotonicity} hold for general Einstein metrics in all dimensions?
   \end{question}

    Over the region where the collapsing is with bounded curvature, it is possible to extend the results of this paper to show that  the collapsing metric can be assumed to have genuine nilpotent symmetry. Thus it leads to the question of understanding the geometry of dimension reduction of canonical metrics under symmetry. Notice there has already been an extensive literature on the latter topic, mainly towards constructing examples. It seems important to systematically investigate the compactness properties of the dimension reduced equations.

\appendix

\section{Construction of regular fibrations}\label{a1}

Our goal here is to outline the proof of Theorem \ref{t:CFG} and Theorem \ref{t:CFG global}. The original construction is due to Cheeger-Fukaya-Gromov in \cite{CFG}. In our special case, the approach presented here is based on the harmonic splitting map of Cheeger-Colding \cite{CheegerColding} which makes it more convenient to obtain  higher regularity estimates. This observation has been used in \cite{NZ} to construct bundle maps with higher regularity.
 The following version can be proved using the $W^{1,p}$-convergence theory of harmonic functions with respect to renormalized measure; see Corollary 4.5 of \cite{AH} for more details.
\begin{theorem}
[Harmonic splitting map, \cite{CheegerColding}]	
Given any $\epsilon>0$ and $n\geq 2$, there exists some $\delta=\delta(n,\epsilon)>0$ such that the following holds. If $(M^n,g,p)$ is a Riemannian manifold satisfying $\Ric_g\geq -(n-1)\delta$ and 
$d_{GH}(B_4(p), B_4(0^d)) < \delta$, $B_4(0^4) \subset \dR^d$,	
then there exists a harmonic map 
\begin{align*}\Phi=(u^{(1)},\ldots, u^{(d)}): B_2(p)\to \dR^d\end{align*} such that the following properties hold:
\begin{enumerate}
\item $\Phi:B_2(p) \to B_2(0^d)\subset \dR^d$ is an $\epsilon$-Gromov-Hausdorff approximation;
\item $|\nabla u^{(\alpha)}|(x)\leq 1 + \epsilon$ holds for any $x\in B_2(p)$ and $1\leq \alpha\leq d$;
\item The following estimate holds
\begin{align*}
\sum\limits_{\alpha,\beta=1}^d \fint_{B_2(p)}|\langle \nabla u^{(\alpha)},  \nabla u^{(\beta)}\rangle -\delta_{\alpha\beta}|\dvol_g + \sum\limits_{\alpha=1}^d \fint_{B_2(p)}|\nabla^2 u^{(\alpha)}|^2 
\dvol_g < \epsilon.	
\end{align*}
 \end{enumerate}
\end{theorem}

We will also need a good cutoff function with uniform derivative estimates. Here we briefly review the standard heat flow regularization, and we refer  to
 Lemma 3.1 of \cite{MN} for  results on general RCD spaces. 
\begin{lemma}\label{l:good-cut-off-function}
Let $(X^n,g)$ be a Riemannian manifold with $\Ric_g\geq 0$. Assume that for any  $m\in\dN$, there exists a constant $\Lambda_{m}>0$ such that $|\nabla^m \Rm_g|\leq \Lambda_{m}$ uniformly on $X$. Then for any $p\in X$ and $r\in(0,1]$ with $\overline{B_{2r}(p)}$ compact, there exists a cut-off function $\psi:X^n\to [0, 1]$ which satisfies the following properties
\begin{enumerate}
\item $\psi\equiv 1$ on $B_r(p)$ and $\psi\equiv0$ on $X\setminus B_{2r}(p)$.

\item For any $m\in\dZ_+$, there exists a constant $C=C(m,n)>0$  such that $r^m|\nabla^m\psi|\leq C$. 
\end{enumerate}

\end{lemma}

\begin{proof}
Without loss of generality suppose $r=1$. The proof below can be made purely local, but to simplify notations we assume $X$ is complete. 
For any $q\in X$, we first take a cutoff function $\rho$ defined by
\begin{align*}
\rho(y)=	
\begin{cases}
	1, &  y \in B_1^g(q),
	\\
	2 - d_{g_j}(y,q), & y\in  A_{1,2}^g(q),
	\\
	0, & y\in X \setminus B_2^g(q).
\end{cases}
\end{align*}
For $t>0$, consider the heat flow 
$\psi_t \equiv 	H_t (\rho)$ of the $1$-Lipschitz cutoff function $\rho$. It is standard that on $X$ we have the pointwise estimate
$
|\nabla_g \psi_t|^2 + \frac{2t}{n}(\Delta_g \psi_t)^2 \leq 1. 
$
Then  for all $y \in X$, we have
\begin{align*}
|\psi_t(y)	 - \rho(y)| \leq  \int_0^t |\Delta_g \psi_s(y)| ds \leq \sqrt{2nt}.
\end{align*}
Now fix $\tau=\frac{1}{18n^2}$. Then $\psi_\tau(y)\in [\frac{2}{3}, 1]$ for $y\in B_1(q)$ and $\psi_\tau(y)\in [0, \frac{1}{3}]$ for $y\in X\setminus B_2(q)$. 

Next, we choose a smooth cutoff function $h:[0,1] \to [0,1]$
 which satisfies 
 \begin{align*}
h (s)= \begin{cases}
 	1, & \frac{2}{3}\leq s\leq 1,
 	\\
 	0, & 0 \leq s \leq \frac{1}{3}, 
 \end{cases}	
 \end{align*}
and set $\psi = h\circ \psi_{\tau}$. Since $\psi_t$ solves the heat equation, the higher order derivative estimate of $\psi$ follows from the standard parabolic estimate.
\end{proof}

\begin{proof}[Proof of Theorem \ref{t:CFG}]
We adopt the notation in the setting of Theorem \ref{t:CFG}. By a simple rescaling, we can assume $\Injrad_{g_{\infty}}(q)\geq 10$ for any $q\in \mathcal{Q}$. 
To begin with, we fixe $\epsilon>0$ sufficiently small, and   define the rescaled Riemannian metric $h_{\epsilon}\equiv \epsilon^{-1}\cdot g_{\infty}$ on $\mathcal{R}$. Throughout the proof we will denote by $\tau(\epsilon)$ a general function of $\epsilon$ satisfying $\lim_{\epsilon\rightarrow0}\tau(\epsilon)=0$.
Now for any $q\in\mathcal{Q}$, there is a harmonic coordinate system
  \begin{align*}\varpi_q\equiv(w_1,\ldots, w_d): B_5^{h_{\epsilon}}(q)\to B_5(0^d),\end{align*} such that 
  \begin{enumerate}
  \item[(i)] $\Delta_{h_{\epsilon}} w_{\alpha} = 0$ for any $1\leq \alpha\leq d$,
  \item[(ii)] $|h_{\epsilon,\alpha\beta} - \delta_{\alpha\beta}|_{C^0(B_4(q))}+|\p_{w_{\gamma}}h_{\epsilon,\alpha\beta}|_{C^0(B_4(q))}\leq \tau(\epsilon)$,	
  \end{enumerate}
where $h_{\epsilon,\alpha\beta}\equiv h_{\epsilon}(\nabla_{h_{\epsilon}}w_{\alpha}, \nabla_{h_{\epsilon}}w_{\beta})$.  
In particular, we have $
d_{GH}(B_{4}^{h_{\epsilon}}(q), B_{4}(0^d)) < \tau(\epsilon),$ where $0^d\in\dR^d$.	

In the following, we will also work with the rescaled metrics $h_i\equiv \epsilon^{-1}\cdot  g_i$ on $X_i^4$. 
Unless otherwise specified,  the metric balls below will be measured in terms of $h_i$ and $h_{\epsilon}$, respectively. 

We will prove the theorem in three steps. In the first step, using the harmonic splitting map, we will construct local fiber bundle maps over every ball in $\mathcal{Q}$ which looks like a ball in $\dR^d$. The second step is to glue the local fiber bundle maps by the well-behaved partition of unity. In the last step, we will show the desired estimates and identify the topology of the collapsing fibers.

\vspace{0.1cm}

\begin{flushleft}
{\bf Step 1 (construction of local fiber bundles).} 
\end{flushleft}

\vspace{0.1cm}

Let $\{\underline{q}_{\ell}\}_{\ell=1}^N$ be a  subset of $\mathcal{Q}$
  such that $\mathcal{Q}\subset\bigcup\limits_{\ell=1}^N B_1(q_{\ell})\subset \mathcal{R}$, and for all $1\leq \ell,\ell'\leq N$ with $\ell\neq \ell'$, we have 
$d_{h_{\epsilon}}(\underline{q}_{\ell}, \underline{q}_{\ell'}) > \frac{1}{2}$.	
 For every $1\leq \ell\leq N$, let $q_{i,\ell} \in X_i^4$
such that $(B_{4}(q_{i,\ell}) , h_i) \xrightarrow{GH}   (B_{4}(\underline{q}_{\ell}), h_{\epsilon})$. Then for any sufficiently large $i$, we have 
$
d_{GH}(B_4(q_{i,\ell}), B_4(0^d)) < 2\tau(\epsilon). 	
$
Then there exists a harmonic map 
\begin{align*}\Phi_{i,\ell}^* = (u_{i,\ell}^{(1)}, \ldots , u_{i,\ell}^{(k)}): B_3(q_{i,\ell}) \to \dR^d \end{align*}
 which satisfies the following integral estimates
\begin{align*}
\sum\limits_{\alpha,\beta=1}^d \fint_{B_3(q_{i,\ell})}|h_i(\nabla_{h_i} u_{i,\ell}^{(\alpha)}, \nabla_{h_i} u_{i,\ell}^{(\beta)}) -\delta_{\alpha\beta}| + \sum\limits_{\alpha=1}^d \fint_{B_3(q_{i,\ell})}|\nabla^2 u_{i,\ell}^{(\alpha)}|_{h_i}^2 \leq  \tau(\epsilon).	
\end{align*}

Since $(B_{4}(q_{i,\ell}), h_i)$ 
is collapsing with uniformly bounded geometry, the above integral estimate can be strengthened to 
the following pointwise estimate on $B_2(q_{i,\ell})$:
\begin{align*}
\sum\limits_{\alpha,\beta=1}^k|g_i(\nabla u_{i,\ell}^{(\alpha)}, \nabla u_{i,\ell}^{(\beta)}) -\delta_{\alpha\beta}| + \sum\limits_{\alpha=1}^k|\nabla^2 u_{i,\ell}^{(\alpha)}|^2 \leq \tau(\epsilon).	
\end{align*}
This implies that, for every $1\leq \ell\leq N$, the composition \begin{align*}\Phi_{i,\ell}\equiv (\varpi_{\ell})^{-1}\circ \Phi_{i,\ell}^*:B_{2}(q_{i,\ell}) \to B_{2}(\underline{q}_{\ell})\end{align*}
is a fiber bundle map, where the diffeomorphism $\varpi_{\ell}: B_{2}(\underline{q}_{\ell})\to B_2(0^d)$ is given by the harmonic coordinate system at $\underline{q}_{\ell}$. Moreover, $\Phi_{i, l}$ is a $\tau(\epsilon)$-Gromov-Hausdorff approximation for all $i$ large.  

\vspace{0.1cm}

\begin{flushleft}
{\bf Step 2 (gluing local bundle maps).} 
\end{flushleft}

\vspace{0.1cm}

Let us take domains with smooth boundary $\mathcal{Q}_i \subset \bigcup\limits_{\ell=1}^N B_2(q_{i,\ell})$ such that $(\mathcal{Q}_i,h_i) \xrightarrow{GH}(\mathcal{Q}, h_{\epsilon})$.
 We will 
glue the above local harmonic maps to obtain a fiber bundle map $
F_i: \mathcal{Q}_i \to \mathcal{Q}.	
$

For every $1\leq \ell \leq N$, let $\psi_{\ell}$ be the good cut-off function in Lemma \ref{l:good-cut-off-function} such that
\begin{align*}
\psi_{\ell}(y)
=
\begin{cases}
	1, & y\in B_{1}(q_{i,\ell}),
	\\
	0, & y \in X_i^4 \setminus B_{2}(q_{i,\ell}),\end{cases}	
\end{align*}
and for all $m\in\dZ_+$, $|\nabla^m \psi_{\ell}|\leq   C_m$ holds everywhere on $M_i^4$. 
Then we take the partition of unity subordinate to the cover $\{B_2(q_{i,\ell})\}_{\ell=1}^N$ of $\mathcal{Q}_i$ given by 
\begin{equation*}
\phi_{\ell}\equiv{\psi_{\ell}}({\sum\limits_{\ell=1}^N\psi_{\ell}})^{-1}.
\end{equation*}
It follows from volume comparison that the multiplicity in the above cover is bounded by some absolute constant $Q_0>0$.
We denote 
$\mathcal{B}_i \equiv \bigcup\limits_{\ell=1}^N B_2(q_{i,\ell})$ and $\mathcal{B}_{\infty} \equiv \bigcup\limits_{\ell=1}^N B_2(\underline{q}_{\ell})$.
For any $1\leq \ell\leq N$, we define  
\begin{align*}
\mathfrak{d}_{\epsilon}(\underline{x},\underline{y})
\equiv\sum\limits_{\alpha=1}^d|w_{\alpha}(\underline{x})-w_{\alpha}(\underline{y})|^2,	
\end{align*}
which is determined by the harmonic coordinate system $(w_1,\ldots, w_d)$ on $B_{2}(q_{\ell})$. It follows from the estimates on the harmonic coordinates that $|\mathfrak{d}_{\epsilon}-d_{h_\epsilon}^2|\leq \tau(\epsilon)$ holds on $B_2(\underline{q}_{\ell})$.
Then let us define the energy function $\mathfrak{E}:\mathcal{B}_i \times \mathcal{B}_{\infty} \to (0,\infty)$
by 
\begin{align*}
\mathfrak{E}(q_i, q_{\infty})\equiv \frac{1}{2}	\sum\limits_{\ell=1}^N \phi_{\ell}(q_i) \cdot \mathfrak{d}_{\epsilon}\Big(\Phi_{i,\ell}(q_i), q_{\infty}\Big).
\end{align*}
By convexity, for any $q_i\in \mathcal{B}_i$, the function $\mathfrak{E}(q_i, \cdot): \mathcal{B}_{\infty} \to [0,\infty)$ has a unique minimum $\mathfrak{z}(q_i)$. It is straightforward to verify that for any $q_i\in B_2(q_{i,\ell})$,
\begin{align*}  d_{h_\epsilon}\Big(\mathfrak{z}(q_i), \Phi_{i,\ell}(q_i)\Big) < \tau(\epsilon),\end{align*}
and 
\begin{align*}
\Big|h_i(\nabla_{h_i} (w_{\alpha}\circ \mathfrak{z}), \nabla_{h_i}  (w_{\beta}\circ \mathfrak{z})) 
-h_i(\nabla_{h_i} u_{i,\ell}^{(\alpha)}, \nabla_{h_i} u_{i,\ell}^{(\beta)}) 
\Big|\leq \tau(\epsilon).	
\end{align*}
Then we define the map 
\begin{align*}F_i: \mathcal{B}_i\to \mathcal{B}_{\infty},\quad q_i \mapsto \mathfrak{z}(q_i), \quad \forall \ \  q_i\in\mathcal{B}_i.\end{align*}
 Combining the above estimates on the harmonic splitting maps, harmonic coordinates, as well as the good cut-off functions, we conclude that $F_i$ is non-degenerate, and hence it is a fiber bundle map.  
Moreover, $F_i:\mathcal{B}_i\to\mathcal{B}_{\infty}$ is a $\tau(\epsilon)$-Gromov-Hausdorff approximation. 
Therefore, the proof of item (1) is complete by taking $\epsilon\rightarrow 0$ and $i\rightarrow\infty$. 

\vspace{0.1cm}

\begin{flushleft}
{\bf Step 3 (proof of the higher order regularity estimates).} 
\end{flushleft}

\vspace{0.1cm}

In this step, we will rescale everything back to the original metrics $g_i$ and $g_{\infty}$, respectively.
 Notice that the uniform estimates for the higher derivatives of the good cut-off functions (constructed in Lemma \ref{l:good-cut-off-function}) hold in our case, and the higher order estimates for the splitting maps $\Phi_{i,\ell}$ and the harmonic coordinates on $\mathcal{Q}$ hold as well. Then we obtain the pointwise estimate on the second fundamental form in item (2) and the higher order estimate $\nabla^k F_i$ in item (3). 
 We skip the details.

We will prove item (4) by contradiction. 
Assume that there exist a sequence $\eta_i\to 0$, a constant $\tau_0>0$, and a sequence of bundle maps $F_i:\mathcal{Q}_i \to \mathcal{Q}$ which are $\eta_i$-Gromov-Hausdorff approximations 
  such that for all sufficiently large $i$, 
\begin{align}
\left|\frac{|dF_i(v)|_{g_{\infty}}}{|v|_{g_i}} - 1\right| \geq \tau_0 	 \label{e:tau_0-difference}
\end{align}
holds for a sequence of vectors $v_i \in T_{x_i}\mathcal{Q}_i$ orthogonal to the fiber of $F_i$. We assume $|v|_{g_i} = 1$.
We take the universal cover of $B_{r_0}(x_i)$ for some sufficiently small constant $r_0>0$, which gives the following equivariant  $C^k$-convergence for any $k\in\dZ_+$:  
\begin{align*}
\begin{split}
 \xymatrix{
(\widetilde{B_{r_0}(x_i)}, \tilde{g}_i, \Gamma_i, \tilde{x}_i) \ar[d]_{\pi_i} \ar[rr]^{C^k} &&    (\widetilde{B}_{\infty}, \tilde{g}_{\infty}, \Gamma_{\infty}, \tilde{x}_{\infty}) \ar [d]^{\pi_{\infty}}
 \\
 (B_{r_0}(x_i), g_i)\ar[rr]^{GH} &&  (B_{r_0}(x_{\infty}), g_{\infty}),
 } 
 \end{split}
 \end{align*}
where $\pi_i: (\widetilde{B_{r_0}(x_i)}, \tilde{x}_i)\to (B_{r_0}(x_i), x_i)$ is the Riemannian universal cover with $\pi_i(\tilde{x}_i)=x_i$, $\Gamma_i\equiv \pi_1(B_{r_0}(x_i))$, and $\Gamma_{\infty}$ is a closed subgroup in $\Isom_{\tilde{g}_{\infty}}(\widetilde{B}_{\infty})$.
The above diagram of equivariant convergence implies that $\pi_{\infty}: \widetilde{B}_{\infty}\longrightarrow B_{r_0}(x_{\infty}) =\widetilde{B}_{\infty}/\Gamma_{\infty}$ is a Riemannian submersion.
Let $\widetilde{F}_i \equiv F_i\circ \pi_i$  and  $\tilde{v}_i$ be the lift of $v_i$ to $\tilde{x}_i$.
Then the $C^k$ convergence implies that $\widetilde{F}_i$ converges to $\pi_{\infty}$, and $\tilde{v}_i$
converges to a limiting vector $\tilde{v}_{\infty}$ with $|d\pi_{\infty}(v_{\infty})|_{g_{\infty}}=1$. This contradicts \eqref{e:tau_0-difference}, which completes the proof of item (4).

Based on the Gromov-Hausdorff estimate in item (1) and the second fundamental form estimate in item (3),
we can conclude that all the fibers of $F_i$ are almost flat manifolds in the sense that 
\begin{align*}
\diam(F_i^{-1}(q))^2 \cdot |\sec_{F_i^{-1}(q)}| < \tau(\epsilon), \quad \forall q\in\mathcal{Q},
\end{align*}
for any sufficiently large $i$. 
If $\epsilon$ is chosen sufficiently small, then item (5) and (6) follows from Gromov and Ruh's theorems on the almost flat manifolds and Fukaya's fibration theorem; see \cite{Gromov, Ruh, FuII, CFG}. 

  \end{proof}

\begin{proof}[Proof of Theorem \ref{t:CFG global}]

By Theorem \ref{t:CFG}, there exists a fiber bundle map 
\begin{align*}F_j: A_{2^j,2^{j+2}}^{g}(p) \to  A_{2^j, 2^{j+2}}^{d_Y}(p^*)\end{align*} for any sufficiently large integer $j\geq j_0$ such that $F_j: A_{1,4}^{g_j}(p)\to A_{1,4}(p^*)$ is a $\tau(\epsilon_j)$-Gromov-Hausdorff approximation with $\lim\limits_{j\to\infty}\tau(\epsilon_j) = 0$.  Here $A^{g_j}_{1,4}(p)$ is the scale down of $A_{2^j, 2^{j+2}}^{g}(p)$ by the factor $2^{-j}$.

Now we glue the above fiber bundle maps over the annuli and thus obtain a global fiber bundle map 
\begin{align*}
F: X^4\setminus B_{R_0}(p) \to Y^d \setminus B_{R_0}(p^*).	
\end{align*}
The procedure is well-known once we have the higher derivative estimates on the local bundle maps.
We outline the arguments. Denote by $\Sigma$ the cross-section of the flat cone $Y$.  
Now let us consider the two adjacent annuli $A_{2^j, 2^{j+2}}^g(p)$
and $A_{2^{j+1}, 2^{j+3}}^g(p)$. Then there exists some isometry $\rho_j\in \Isom(\Sigma)$ such that $|F_{j+1} - \rho_j \circ F_j| < \tau(\epsilon_j)$ holds on the intersection $A_{2^{j+1}, 2^{j+2}}(p)$. Moreover, the higher order regularity estimates in Theorem \ref{t:CFG} implies that 
the above approximation can be improved to the $C^k$ sense.
Then there exists a self-diffeomorphism $\sigma_j:A_{2^{j+1}, 2^{j+2}}(p)\to A_{2^{j+1}, 2^{j+2}}(p)$ which is close to the identity map such that 
$F_{j+1} = \rho_j \circ F_j \circ \sigma_j.$
One can choose $\sigma_j: F_j^{-1}(q)\to F_{j+1}^{-1}(q)$ as the normal projection from the fiber. It is indeed a diffeomorphism since the normal injectivity radius of each fiber has a uniform lower bound. We refer the readers to proposition A2.2 of \cite{CFG} for more details. 
Using the good cut-off function 
\begin{align*}\chi_j(y)
=
\begin{cases}
1, \quad y\in A_{2^{j+\frac{5}{4}}, 2^{j+\frac{7}{4}}}(p),
\\
0,	\quad y \in X^4\setminus A_{2^{j+1}, 2^{j+2}}(p), 
\end{cases}
\end{align*} given in Lemma \ref{l:good-cut-off-function}, we can construct a modified fibration $\widehat{F}_j: A_{2^j, 2^{j+3}}^g(p)\to A_{2^j, 2^{j+3}}^{d_Y}(p^*)$ which satisfies the properties in Theorem \ref{t:CFG}. Inductively, we finally obtain a global fiber bundle map 
\begin{align*}
F: X^4\setminus B_{R_0}(p) \to Y^d \setminus B_{R_0}(p^*),	
\end{align*}
which satisfies items (1), and (2).

The estimate on the second fundamental form in item (3) depends on the special limiting geometry in the hyperk\"ahler setting, for the sequence $A_{1, 4}^{g_j}$. 
There are three cases to analyze: 
\begin{itemize}
	\item ($d=3$) the limiting universal cover is flat, and the limit of the fibers are given by totally geodesic lines $\dR$.  Since we have convergence of the second fundamental form (see Lemma \ref{l: convergence of second fundamental form}, we get the conclusion.
\item ($d=2$) the limiting geometry is again flat and the proof is similar to above. 
\item ($d=1$) the limiting geometry is either the flat or nilpotent geometry. In the first case the proof is the same as above; in the second case we make use of Lemma  \ref{l:second fund form formula}, and one can compute explicitly the relation between the distance function and the function $z$.
\end{itemize}
\end{proof}

\section{Poisson's equation on the Calabi model space}
\label{a2} 
Let $(\Ca, \bom_{\Ca})$ be a Calabi model space. We identify $\Ca$ differentiably with   the product space $[2,\infty)\times \mathfrak{N}^3$, where we use the moment coordinate $z$ and $\mathfrak{N}^3$ is a nilmanifold. We first recall the separation of variables arguments in \cite{HSVZ}; see  Section 4 of \cite{HSVZ} for more details. Denote $\mathfrak{N}_z^3\equiv \{z\}\times \mathfrak{N}^3$, and 
let $\Lambda \equiv \{\Lambda_k\}_{k=0}^{\infty}$ be the spectrum of $-\Delta_{h_0}$ on the fixed slice $\mathfrak{N}_{z_0}^3$. Then we have  
$\Lambda_k = (2z_0)^{-1}\cdot \lambda_k  + 2z_0 \cdot j_k^2$ with $\lambda_k\geq j_k$ and $j_k\in\dZ_{\geq 0}$.
Given a continuous function $u$ on $\Ca$, we can write the $L^2$ expansion
\begin{align*}u(z,\bm{y}) = \sum\limits_{k=0}^{\infty}u_k(z)\cdot \varphi_k(\bm{y}),  \quad \bm{y}\in\mathfrak{N}_{z_0}^3,\end{align*}	
  where $-\Delta_{h_0}\varphi_k=\Lambda_k\varphi_k$.  
The equation $\Delta_{\bom_{\Ca}}u = v$ is equivalent to that for all $k$,
\begin{align}
\frac{d^2 u_k(z)}{dz^2} - (j_k^2 z^2 + \lambda_k) u_k(z) = v_k(z) \cdot z, \quad z\geq 1,	\label{e:non-homogeneous}
\end{align}
  where $v_k(z)$ is the corresponding coefficient in the expansion of $v$.
The corresponding homogeneous equation has two explicit fundamental solutions $\mathcal F_k(z)$ and $\mathcal U_k(z)$: 
\begin{enumerate}
\item $j_k=\lambda_k = 0$:  $\mathcal{F}_k(z) = z$ and $\mathcal{U}_k(z) = 1$.
\item $j_k=0$, $\lambda_k>0$:
$\mathcal{F}_k(z) = e^{\sqrt{\lambda_k}\cdot z}$ and $\mathcal{U}_k(z) = e^{-\sqrt{\lambda_k}\cdot z}$.	
  
\item $j_k>0$:
$
\mathcal{F}_k(z) = e^{-\frac{j_k \cdot z^2}{2}}H_{-h-1}(-\sqrt{j_k}\cdot z)$ and $
 \mathcal{U}_k(z) = e^{-\frac{j_k \cdot z^2}{2}}H_{-h-1}(\sqrt{j_k}\cdot z),
$
 	where $h$ satisfies $\lambda_k=(2h+1)j_k$, and 
 	$
 	H_{-h-1}(y) \equiv \int_0^{\infty}e^{-t^2 - 2 ty} t^h dt.$
\end{enumerate}
When $j_k>0$, we denote $y\equiv \sqrt{j_k}\cdot z$ and we define
  \begin{align}
  F_k(t)\equiv -t^2 + 2ty + h\log t  \quad  \text{and}\quad 	  U_k(t)\equiv -t^2 - 2ty + h\log t.
\nonumber  \end{align}
Let $t_k$ and $s_k$ be the unique positive critical point of $F_k(t)$ and $U_k(t)$ respectively:
\begin{align}
t_k = \frac{y}{2} + \sqrt{\frac{h^2}{2}+\frac{y^2}{4}}  \quad  \text{and} \quad 	s_k = -\frac{y}{2} + \sqrt{\frac{h^2}{2}+\frac{y^2}{4}}, 
\nonumber\end{align}
and define 
\begin{align}
\widehat{F}_k(z)\equiv -\frac{j_kz^2}{2} + F_k(t_k(z)) \quad \text{and} \quad 
\widehat{U}_k(z)\equiv -\frac{j_kz^2}{2} + U_k(s_k(z)).	
\nonumber\end{align}
The following two results are taken from Lemma 4.6 and Lemma 4.7  in \cite{HSVZ}.
\begin{lemma} \label{l:laplace-method}  The following uniform estimates hold:
 \begin{align*}
 \mathcal{F}_k(z) \leq (1+\sqrt{\pi})\widehat{F}_k(z)
 \quad \text{and} \quad 
  \mathcal{U}_k(z) \leq (1+\sqrt{\pi}) \widehat{U}_k(z). \end{align*}
\end{lemma}
\begin{lemma} \label{l:wronskian} We have
\begin{enumerate}
\item 
$\widehat{F}_k(z)$ is increasing for $z\geq 1$ and 
$\widehat{U}_k(z)$ is decreasing for $z\geq 1$.
\item There exists a uniform constant $C_0>0$ independent of $k$ such that 
\begin{align*}0<\mathcal{W}_k(z)^{-1}\cdot (e^{\widehat{F}_k(z)+\widehat{U}_k(z)})\leq C_0,\end{align*}
where $\mathcal{W}_k(z)\equiv\mathcal{F}_k'(z)\mathcal{U}_k(z) - \mathcal{F}_k(z)\mathcal{U}_k'(z)$ is the Wronskian of $\mathcal F_k$ and $\mathcal U_k$.
	\end{enumerate}
 \end{lemma}
 Denote $\mathcal{Q}_w\equiv \{x\in\Ca|z(\bx)\geq w\}$, and fix any $\tau\in(-\infty,0)\setminus\{-3\}$. The following is used in the proof of Proposition \ref{p:solving Laplace equation on Calabi model}.
 \begin{proposition}\label{p:weighted-in-z} 
 There exists a constant $C>0$ such that if $v\in C^{5}(\mathcal{Q}_w)$ for some $w>2$  satisfies 
\begin{equation*}\sum_{\ell=0}^5(z(\bx))^{3\ell/2}\cdot |\nabla_{g_{\Ca}}^{\ell} v(\bx)|_{g_{\Ca}}\leq \mathfrak b \cdot (z(\bx))^{\tau}, \quad \forall \ \ \bx \in \mathcal{Q}_w,\end{equation*}
 then $\Delta_{\bom_{\Ca}}u=v$ 
 has a solution $u\in C^{6}(\mathcal{Q}_w)$ satisfying \begin{align}
  |u(\bx)| \leq C \cdot \mathfrak b \cdot (z(\bx))^{3+\tau}, \quad \forall \ \ \bx \in \mathcal{Q}_w
\label{e:poisson-solution-estimate}
\end{align}
\end{proposition}
The proof depends on the estimates of solutions to the non-homogeneous equation \eqref{e:non-homogeneous}. We write the expansion of $v$ as $v(z,\bm{y}) = \sum\limits_{k=0}^{\infty}v_k(z)\cdot \varphi_k(\bm{y})$.
In the case $j_k=0$ and $\lambda_k=0$, we set  
\begin{align}u_k(z)\equiv \int_{w}^z( \int_{w}^t v_k(s) ds)dt. \label{e:solution j=lambda=0}	
\end{align}
In the case $j_k=0$ and $\lambda_k\neq 0$, we set 
\begin{align}
u_k(z)\equiv\frac{1}{2\sqrt{\lambda_k}}\Big(e^{-\sqrt{\lambda_k}z}\cdot \int_w^ze^{\sqrt{\lambda_k}t}\cdot v_k(t)\cdot tdt + e^{\sqrt{\lambda_k}z}\cdot \int_z^{\infty}e^{-\sqrt{\lambda_k}t}\cdot v_k(t)\cdot tdt\Big).	\label{e:solution j=0}
\end{align}
In the case $j_k\in\dZ_+$, we set
\begin{align}
u_k(z)
\equiv \frac{\mathcal{U}_k(z)}{\mathcal{W}_k(z)}\int_w^z 	 \mathcal{F}_k(t) \cdot v_k(t) \cdot t \cdot dt  + \frac{\mathcal{F}_k(z)}{\mathcal{W}_k(z)}\int_z^{\infty} 	 \mathcal{U}_k(t) \cdot v_k(t) \cdot t \cdot dt.  \label{e:particular-solution}
\end{align}

 \begin{lemma}
 \label{l:coefficients-estimate}
 There exists a constant $C>0$ independent of $k$ such that for all $w>2$, any solution given by \eqref{e:solution j=lambda=0}, \eqref{e:solution j=0} or \eqref{e:particular-solution} satisfies
 $$\sup_{z\geq w} |u_k(z)|z^{-2-\tau}\leq C\sup_{z\geq w}|v_k(z)|z^{-\tau}.$$
 \end{lemma}

\begin{proof}
[Proof of Lemma \ref{l:coefficients-estimate}] 
For \eqref{e:solution j=lambda=0} this is immediate. Below we will only treat the solution given by \eqref{e:particular-solution}. The case for \eqref{e:solution j=0}  can be dealt with in a similar fashion.
Denote $\mathfrak B_k=\sup_{z\geq w}|v_k(z)|z^{-\tau}$. 
We first estimate the second term in \eqref{e:particular-solution}. Applying Lemma \ref{l:laplace-method}, we have
\begin{align}
\int_z^{\infty}\mathcal{U}_k(t) \cdot v_k(t) \cdot t \cdot dt \leq C\cdot \mathfrak{B}_k \cdot  \int_0^{\infty} e^{\widehat{U}_k(u+z) + (1+\tau)\log (u+z)} du. \nonumber
\end{align}
We denote 
$\widetilde{U}_k(u) \equiv	\widehat{U}_k(u+z) + (1+\tau)\log (u+z)$.
By a simple computation, if $z>1$, then we have 
$\widetilde{U}_k'(0) = -j_k \cdot z + (1+\tau)z^{-1} < 0$
and $\widetilde{U}_k''(u) < 0 $ for all $u>0$. 
	Therefore, \begin{align}
\widetilde{U}_k(u)\leq 	\widetilde{U}_k(0)
+ \widetilde{U}_k'(0) \cdot u, \quad \forall \ \  u\geq 0. \nonumber
\end{align}
So it follows that 
\begin{align}
	\int_z^{\infty}\mathcal{U}_k(t) \cdot v_k(t) \cdot t \cdot dt
& \leq C \cdot \mathfrak{B}_k \cdot e^{\widetilde{U}_k(0)}\cdot \int_0^{\infty}e^{\widetilde{U}_k'(0) \cdot u}dt
\leq C \cdot \mathfrak{B}_k \cdot e^{\widehat{U}_k(z)}\cdot z^{2+\tau}. \nonumber
\end{align}
Therefore, combining the above estimate and Lemma \ref{l:wronskian}(2),
\begin{align}
\frac{\mathcal{F}_k(z)}{\mathcal{W}_k(z)}	\cdot 	\int_z^{\infty}\mathcal{U}_k(t) \cdot v_k(t) \cdot t \cdot dt
& \leq C \cdot \mathfrak{B}_k \cdot \frac{e^{\widehat{F}_k(z) + \widehat{U}_k(z)}}{\mathcal{W}_k(z)} \cdot z^{2+\tau}
 \leq  C \cdot \mathfrak{B}_k \cdot z^{2+\tau}. \nonumber
\end{align}
For the first term of \eqref{e:particular-solution}, we apply the uniform estimate in Lemma \ref{l:laplace-method}, the monotonicity of $\widehat{F}_k(z)$ in Lemma \ref{l:wronskian}, as well as Lemma \ref{l:wronskian}(2), to obtain \begin{align}
\frac{\mathcal{U}_k(z)}{\mathcal{W}_k(z)}	\cdot 	\int_w^z\mathcal{F}_k(t) \cdot v_k(t) \cdot t \cdot dt
  \leq    C \cdot \mathfrak{B}_k \cdot z^{2+\tau}.
\nonumber
\end{align}
Adding up the above two terms, we obtain the conclusion.\end{proof}
\begin{proof}
[Proof of Proposition \ref{p:weighted-in-z}]

First consider the case $\Lambda_k>0$. By the same computations as in the proof of Lemma 4.9 of \cite{HSVZ},   there is some constant $C>0$ independent of $k\in\dZ_+$ such that 
\begin{align*}|v_k(z)| & \leq C \cdot (\Lambda_k+1)^{-2}\cdot \Vol_{h_0}(\mathfrak{N}_{z_0}^3)^{\frac{1}{2}} \cdot \|(-\tau_{h_0})^2 v\|_{C^0(\mathfrak{N}_{z_0}^3)} \nonumber\\
& \leq  C \cdot (\Lambda_k+1)^{-2}\cdot \Vol_{h_0}(\mathfrak{N}_{z_0}^3)^{\frac{1}{2}} \cdot\Big( z^2 \|\nabla^4 v\|_{C^0(\mathfrak{N}_{z}^3)}  +  z^{\frac{3}{2}}\|\nabla^3 v\|_{C^0(\mathfrak{N}_{z}^3)} + z \|\nabla^2 v\|_{C^0(\mathfrak{N}_{z}^3)}\Big)
\nonumber\\
& \leq C \cdot (\Lambda_k+1)^{-2}\cdot  z(\bx)^{\tau}.
\end{align*}	
It is easy to see the same estimate also holds when $\Lambda_k=0$.

Now consider the formal solution $u(\bm x)=\sum_{k=0}^\infty u_k(z)\cdot\varphi_k(\bm y)$, where $u_k(z)$ is given by \eqref{e:solution j=lambda=0}, \eqref{e:solution j=0} and \eqref{e:particular-solution}. By Lemma \ref{l:coefficients-estimate} and Weyl's law, we see that $u(\bm x)$ is convergent and satisfies
  \begin{align}
  |u(\bx)| \leq C \cdot z(\bx)^{3+\tau} +  C \cdot \Big(\sum\limits_{k=2}^{\infty}\frac{1}{k^{\frac{4}{3}}}\Big)\cdot (z(\bx))^{3+\tau} \leq C \cdot (z(\bx))^{3+\tau},\nonumber
  \end{align}
 where $C>0$ is independent of $x\in\mathcal{Q}_w$ and $w>2$.
\end{proof}

 \bibliographystyle{plain}

\bibliography{HK4}

 \end{document}